\pdfoutput=1 

\documentclass[11pt]{amsart}

\usepackage{amssymb,amsthm,enumitem,colonequals,tikz-cd,microtype}
\usepackage[normalem]{ulem} 

\usepackage[osf]{Baskervaldx}
\usepackage[baskervaldx]{newtxmath}
\usepackage[cal=boondoxo]{mathalfa} 

\usepackage[top=3.75cm, bottom=3cm, left=3.5cm, right=3.5cm]{geometry}

\usepackage{xcolor}
\colorlet{darkblue}{blue!55!black}
\colorlet{darkcyan}{cyan!50!black}
\colorlet{darkgreen}{green!60!black}

\PassOptionsToPackage{hyphens}{url}
\usepackage{hyperref}
\hypersetup{
    colorlinks=true,
    linkcolor=darkblue,
    urlcolor=darkcyan,
    citecolor=darkgreen,
}

\def\eqref#1{\textcolor{darkblue}{(\ref{#1})}}

\usepackage[nameinlink]{cleveref} 
\Crefformat{section}{#2\S#1#3}
\Crefmultiformat{section}{#2\S\S#1#3}{ and~#2#1#3}{, #2#1#3}{, and~#2#1#3}

\usepackage[pagewise]{lineno}
\let\oldequation\equation
\let\oldendequation\endequation
\renewenvironment{equation}{\linenomathNonumbers\oldequation}{\oldendequation\endlinenomath}
\expandafter\let\expandafter\oldequationstar\csname equation*\endcsname
\expandafter\let\expandafter\oldendequationstar\csname endequation*\endcsname
\renewenvironment{equation*}{\linenomathNonumbers\oldequationstar}{\oldendequationstar\endlinenomath}
\let\oldalign\align
\let\oldendalign\endalign
\renewenvironment{align}{\linenomathNonumbers\oldalign}{\oldendalign\endlinenomath}
\expandafter\let\expandafter\oldalignstar\csname align*\endcsname
\expandafter\let\expandafter\oldendalignstar\csname endalign*\endcsname
\renewenvironment{align*}{\linenomathNonumbers\oldalignstar}{\oldendalignstar\endlinenomath}

\theoremstyle{plain}
\newtheorem{theorem}{Theorem}[section]
\newtheorem{lemma}[theorem]{Lemma}
\newtheorem{corollary}[theorem]{Corollary}
\newtheorem{proposition}[theorem]{Proposition}
\newtheorem{question}[theorem]{Question}

\theoremstyle{definition}
\newtheorem{definition}[theorem]{Definition}
\newtheorem{example}[theorem]{Example}
\newtheorem{remark}[theorem]{Remark}
\newtheorem{setup}[theorem]{Setup}

\newtheorem*{ack}{Acknowledgments}

\AddToHook{env/conjecture/begin}{\crefalias{theorem}{conjecture}}
\AddToHook{env/lemma/begin}{\crefalias{theorem}{lemma}}
\AddToHook{env/corollary/begin}{\crefalias{theorem}{corollary}}
\AddToHook{env/proposition/begin}{\crefalias{theorem}{proposition}}
\AddToHook{env/definition/begin}{\crefalias{theorem}{definition}}
\AddToHook{env/remark/begin}{\crefalias{theorem}{remark}}
\AddToHook{env/setup/begin}{\crefalias{theorem}{setup}}
\AddToHook{env/example/begin}{\crefalias{theorem}{example}}
\AddToHook{env/notation/begin}{\crefalias{theorem}{notation}}
\AddToHook{env/question/begin}{\crefalias{theorem}{question}}

\numberwithin{equation}{section}
\numberwithin{theorem}{section}

\title[Germwise singular Fourier--Mukai partnerships]{Germwise singular Fourier--Mukai partnerships}

\author[P.~Lank]{Pat Lank}
\address{P.~Lank,
Dipartimento di Matematica “F. Enriques”, Universit\`{a} degli Studi di Milano, Via Cesare
Saldini 50, 20133 Milano, Italy}
\email{lankp.math@gmail.com}

\date{\today}



\begin{document}
    
\begin{abstract}
    We study how local behavior on the base controls whether integral transforms induce fully faithful functors or equivalences between singularity categories of proper flat families. 
    By working with \'{e}tale germs, we show (under mild hypotheses) that the loci where this happens are open and are compatible with \'{e}tale base change. 
    Moreover, we give examples, exhibit an autoequivalence of the singularity category which detects regularity of the base, and compute these loci for arithmetic families.
\end{abstract}

\maketitle

\tableofcontents

\section{Introduction}
\label{sec:intro}
 
The singularity category of a Noetherian algebraic space $X$ is the Verdier quotient $D_{\operatorname{sg}}(X) := D^b_{\operatorname{coh}}(X)/\operatorname{Perf}(X)$.
Buchweitz introduced it for Noetherian rings \cite{Buchweitz:2021}, building on Eisenbud's work on matrix factorizations \cite{Eisenbud:1980}.
In the geometric setting, Orlov later studied this for schemes \cite{Orlov:2004}.
See \cite{Booth:2026} for a recent survey.
 
As its name suggests, $D_{\operatorname{sg}}(X)$ measures the singularities of $X$.
Indeed, $D_{\operatorname{sg}}(X)$ is trivial if, and only if, $X$ is regular \cite[Lemma 5.28]{GuisadoVillaalgordo/Lank:2026}.
For suitable schemes, its idempotent completion only depends on a formal neighborhood of the singular locus \cite{Orlov:2011}.
Nevertheless, it remains subtle to understand how geometric operations affect $D_{\operatorname{sg}}(X)$.
 
\subsection{Motivation}
\label{sec:intro_motivation}
 
This work concerns integral transforms between singularity categories in families.
Let $S$ be a Noetherian algebraic space and $f_i \colon Y_i \to S$ be proper flat morphisms of Noetherian algebraic spaces for $i\in\{1,2\}$.
Denote by $p_i \colon Y_1\times_S Y_2 \to Y_i$ the canonical projections.
For $E\in D^b_{\operatorname{coh}}(Y_1\times_S Y_2)$, the \textit{integral transform} with kernel $E$ is the functor $\Phi_E \colon D_{\operatorname{qc}}(Y_1) \to D_{\operatorname{qc}}(Y_2)$ given by $A \mapsto \mathbf{R}(p_2)_\ast (\mathbf{L}p_1^\ast A \otimes^{\mathbf{L}} E)$. 
 
An integral transform $\Phi_E$ induces a functor $\Phi_E^{\operatorname{sg}}\colon D_{\operatorname{sg}}(Y_1) \to D_{\operatorname{sg}}(Y_2)$ precisely if it preserves both $D^b_{\operatorname{coh}}$ and $\operatorname{Perf}$.
These conditions are equivalent to $E$ being relatively perfect over each $Y_i$ \cite[Corollary 1.9]{GuisadoVillaalgordo/Lank:2026}.
See also \cite{Dutta/Lank/ManaliRahul:2025,Ballard:2009, Rizzardo:2017}.
If $\Phi_E^{\operatorname{sg}}$ is an equivalence, we say that $Y_1$ and $Y_2$ are \textbf{singular Fourier--Mukai partners} over $S$.
 
Central to this work is the following question.
 
\begin{question}
    \label{question:locality}
    Is being fully faithful, or an equivalence, on singularity categories a local condition on the base $S$?
    If so, which notion of locality detects it?
\end{question}
 
On $D^b_{\operatorname{coh}}$, the answer is well understood: whether $\Phi_E$ is fully faithful or an equivalence on $D^b_{\operatorname{coh}}$ is detected on fibers.
See \cite{Ruiperez/Hernandez/Martin/SanchodeSalas:2009, GuisadoVillaalgordo/Lank/ManaliRahul/Pavic:2025, GuisadoVillaalgordo/Lank:2026}.
This fails for $D_{\operatorname{sg}}$.
 
\begin{example}
    Suppose $S$ is a nonregular Noetherian scheme. 
    Let $\mathcal{E}$ be a vector bundle on $S$ of positive rank.
    Set $Y_1 = Y_2 = \mathbf{P}_S(\mathcal{E})$ and $E$ to be the zero object.
    Since every fiber of $\mathbf{P}_S(\mathcal{E}) \to S$ is regular, the base change of $\Phi_E^{\operatorname{sg}}$ to each fiber is an equivalence between trivial categories.
    However, $\mathbf{P}_S(\mathcal{E})$ is not regular, so $\Phi_E^{\operatorname{sg}}$ itself is not an equivalence.
    See \Cref{ex:arithmetic_conic_bundle} for the same phenomenon with a nonzero kernel.
\end{example}
 
Testing after arbitrary field extensions is also unreliable.
 
\begin{example}
    \label{ex:kollar}
    Koll\'{a}r constructs a projective curve over a field of positive characteristic which is regular.
    However, the curve is singular after a purely inseparable field extension \cite[Remark 16]{Kollar:2011}.
    Hence, its singularity category is trivial, whereas that of its base change is not.
    See \Cref{rmk:kollar_curve}.
\end{example}
 
\subsection{Results}
\label{sec:intro_results}
 
\subsubsection{Openness}
\label{sec:intro_openness}
 
Our answer to \Cref{question:locality} is that locality should be measured by local rings of the base rather than by fibers.
For schemes, one could use local rings directly.
For algebraic spaces, the natural substitute is given by \'{e}tale germs. 
This approach is well behaved because regularity is \'{e}tale local.
 
\begin{definition} 
    Let $p\in |S|$.
    An \textit{\'{e}tale germ} of $(S,p)$ is an \'{e}tale morphism $s\colon U \to S$ from a scheme together with a point $u\in |s|^{-1}(p)$.
    We say that $\Phi_E^{\operatorname{sg}}$ is \textit{fully faithful} (resp.\ an \textit{equivalence}) \textit{at $p$} if there exists an \'{e}tale germ $(U,u)\to (S,p)$ such that the base change of $\Phi_E^{\operatorname{sg}}$ along $\operatorname{Spec}(\mathcal{O}_{U,u}) \to S$ is fully faithful (resp.\ an equivalence).
    Denote by $\operatorname{fm}_{\operatorname{sg}}(E)$ (resp.\ $\operatorname{FM}_{\operatorname{sg}}(E)$) the collection of such points.
\end{definition}
 
See \Cref{sec:germs}.
By \Cref{lem:singular_FM_locus_independence}, this notion does not depend on the choice of \'{e}tale germ.
In particular, if $S$ is a scheme, then $\Phi_E^{\operatorname{sg}}$ is fully faithful (resp.\ an equivalence) at $p$ if, and only if, its base change along $\operatorname{Spec}(\mathcal{O}_{S,p}) \to S$ is fully faithful (resp.\ an equivalence).

Our main result is the following.
 
\begin{theorem}
    \label{introthm:openness_singular_fm}
    Let $S$ be a quasi-compact quasi-separated quasi-excellent algebraic space.
    Suppose $f_i \colon Y_i \to S$ are proper flat morphisms of Noetherian algebraic spaces for $i\in\{1,2\}$.
    Denote by $p_i \colon Y_1\times_S Y_2 \to Y_i$ the canonical projections.
    If $E\in D^b_{\operatorname{coh}}(Y_1\times_S Y_2)$ is such that both $E$ and $\operatorname{\mathbf{R}\mathcal{H}\!\mathit{om}}(E, p_2^! \mathcal{O}_{Y_2})$ are relatively perfect over each $Y_i$, then $\operatorname{fm}_{\operatorname{sg}}(E)$ and $\operatorname{FM}_{\operatorname{sg}}(E)$ are open subsets.
\end{theorem}
 
See \Cref{thm:openness_singular_fm}.
This is the singularity category analog of \cite[Theorem 1.3]{GuisadoVillaalgordo/Lank:2026}.
The hypothesis on $E$ ensures that $\Phi_E$ and its right adjoint both induce functors on singularity categories.
For instance, it is satisfied whenever each $f_i$ is Gorenstein and $E$ is perfect.
On the other hand, if each $f_i$ is smooth, then any $E$ which is relatively perfect over each $Y_i$ is perfect.
See \Cref{ex:perfect_kernels,ex:kernels_for_smooth}.
 
The proof of \Cref{introthm:openness_singular_fm} combines generation with the base change techniques of \cite{GuisadoVillaalgordo/Lank:2026}.
We discuss the case of full faithfulness.
First, quasi-excellence ensures each $Y_i$ is $J\textrm{-}2$, and so $D^b_{\operatorname{coh}}(Y_i)$ admits a classical generator.
This is a variation of \cite[Theorem 1.1]{Dey/Lank:2024a} for Noetherian algebraic spaces, which we record in \Cref{prop:j_condition_stacks} and \Cref{cor:classical_generator_for_J2}.
 
Second, full faithfulness on singularity categories reduces to perfectness of the cone of the unit at a classical generator.
Third, derived pullback to a germ is essentially dense on bounded coherent complexes, and units are compatible with base change along quasi-affine morphisms.
See \Cref{lem:fully_faithful_iff_cone_of_classical_generator_is_perfect,lem:counit_isomorphism_for_localization} and \cite[Remark 4.29]{GuisadoVillaalgordo/Lank:2026}.

Combining these with \cite[Lemma 3.46]{GuisadoVillaalgordo/Lank:2026}, the complement of $\operatorname{fm}_{\operatorname{sg}}(E)$ is the image under the closed map $|f_1|$ of the closed subset of points of $|Y_1|$ at which the cone of the unit at the classical generator is not perfect.
Thus, the obstruction is measured through local rings of $Y_1$, rather than on the fibers of $f_1$.
 
\subsubsection{Germwise criteria}
\label{sec:intro_criteria}
 
We prove that singular Fourier--Mukai partnership is determined germwise, and that it suffices to check closed points.
 
\begin{proposition}
    \label{introprop:germwise_criteria}
    Assume the hypotheses of \Cref{introthm:openness_singular_fm}.
    The following are equivalent:
    \begin{enumerate}
        \item $\Phi_E^{\operatorname{sg}}$ is fully faithful (resp.\ an equivalence)
        \item $\Phi_E^{\operatorname{sg}}$ is fully faithful (resp.\ an equivalence) at all $p\in |S|$
        \item $\Phi_E^{\operatorname{sg}}$ is fully faithful (resp.\ an equivalence) at all closed points $p\in |S|$.
    \end{enumerate}
\end{proposition}
 
See \Cref{prop:germwise_criteria}.
This should be compared to its fiberwise counterparts for bounded derived categories \cite[Proposition 2.15]{Ruiperez/Hernandez/Martin/SanchodeSalas:2009} and \cite[Theorem 1.2]{GuisadoVillaalgordo/Lank/ManaliRahul/Pavic:2025}.
To the best of our knowledge, no such criterion exists in the literature for singularity categories.
 
For a morphism $t\colon T \to S$, denote by $E_t$ the derived pullback of $E$ to $Y_1\times_S Y_2\times_S T$.
If $L/k$ is a field extension and $t\colon \operatorname{Spec}(L) \to \operatorname{Spec}(k)$ is the induced morphism, we write $E_L$ for $E_t$.
Over a field, the local rings of \'{e}tale germs of the unique point are finite separable field extensions.
This yields the following.
 
\begin{corollary}
    \label{introcor:criteria}
    Assume the hypotheses of \Cref{introthm:openness_singular_fm} with $S=\operatorname{Spec}(k)$ for a field $k$.
    The following are equivalent:
    \begin{enumerate}
        \item $\Phi_E^{\operatorname{sg}}$ is fully faithful (resp.\ an equivalence)
        \item $\Phi_{E_L}^{\operatorname{sg}}$ is fully faithful (resp.\ an equivalence) for all finite separable field extensions $L/k$
        \item $\Phi_{E_L}^{\operatorname{sg}}$ is fully faithful (resp.\ an equivalence) for some finite separable field extension $L/k$.
    \end{enumerate}
\end{corollary}
 
See \Cref{cor:criteria}.
This is a singularity category analog of \cite[Lemma 2.12]{Orlov:2002}.
We refer to \cite{Verasdanis:2026} for related results concerning separable extensions in tt-geometry. 
Koll\'{a}r's curve shows that the separability assumption cannot be removed.
Indeed, taking $Y_1$ to be Koll\'{a}r's curve, $Y_2 = \operatorname{Spec}(k)$, and $E$ the zero object, $\Phi_E^{\operatorname{sg}}$ is an equivalence between trivial categories, whereas its base change along a suitable purely inseparable extension is not.
 
\subsubsection{Base change}
\label{sec:intro_base_change}
 
The loci are also compatible with \'{e}tale base change.
 
\begin{proposition}
    \label{introprop:base_change}
    Assume the hypotheses of \Cref{introthm:openness_singular_fm}.
    For any \'{e}tale morphism $t\colon S^\prime \to S$ of Noetherian algebraic spaces,
    \begin{displaymath}
        |t|^{-1}(\operatorname{fm}_{\operatorname{sg}}(E)) = \operatorname{fm}_{\operatorname{sg}}(E_t)
        \quad \textrm{and} \quad
        |t|^{-1}(\operatorname{FM}_{\operatorname{sg}}(E)) = \operatorname{FM}_{\operatorname{sg}}(E_t).
    \end{displaymath}
\end{proposition}
 
See \Cref{prop:base_change}.
A base change result for $D^b_{\operatorname{coh}}$ appears in \cite[Theorem 1.3]{GuisadoVillaalgordo/Lank:2026}, which allows arbitrary Noetherian base change.
This fails for singularity categories.
In the example following \Cref{introcor:criteria}, both loci are all of $|\operatorname{Spec}(k)|$, whereas both loci of the base change along a suitable purely inseparable extension are empty.
Thus, \'{e}tale base change provides a natural setting for the base-change theorem.
 
\subsection{An autoequivalence}
\label{sec:intro_autoequivalence}
 
Singularity categories can also exhibit autoequivalences which are not seen by $D^b_{\operatorname{coh}}$.
Let $f\colon Y \to S$ be a proper flat morphism of Noetherian algebraic spaces with $S$ regular.
Denote by $\mathcal{I}_{\Delta_f}$ the ideal sheaf of the diagonal of $f$.
On $D^b_{\operatorname{coh}}$, $\Phi_{\mathcal{I}_{\Delta_f}}$ is an autoequivalence for elliptic fibrations \cite[Theorem 1.4]{GuisadoVillaalgordo/Lank:2026}, extending \cite{Ruiperez/Hernandez/Martin/SanchodeSalas:2009}.
 
However, it is not even faithful in general.
For instance, if the canonical morphism $\mathcal{O}_S \to \mathbf{R} f_\ast \mathcal{O}_Y$ is an isomorphism, then $\Phi_{\mathcal{I}_{\Delta_f}}(\mathcal{O}_Y) \cong 0$ by \Cref{lem:diagonal_ideal_kills_pullbacks}.
Thus, $\Phi_{\mathcal{I}_{\Delta_f}}$ is not faithful on $D^b_{\operatorname{coh}}(Y)$ (e.g.\ take the identity and zero morphism on $\mathcal{O}_Y$).
In any case, over a regular base, we prove that it induces an autoequivalence of $D_{\operatorname{sg}}(Y)$.
 
\begin{proposition}
    \label{prop:ideal_sheaf_diagonal}
    Let $f\colon Y \to S$ be a proper flat morphism of Noetherian algebraic spaces where $S$ is regular.
    Denote by $\Delta_f \colon Y \to Y\times_S Y$ the diagonal morphism of $f$ with associated ideal sheaf $\mathcal{I}_{\Delta_f}\subseteq \mathcal{O}_{Y\times_S Y}$.
    Then $\Phi_{\mathcal{I}_{\Delta_f}}$ induces an autoequivalence on $D_{\operatorname{sg}}(Y)$. 
\end{proposition}
 
Here no assumption is made on the singularities of the fibers of $f$.
While \Cref{prop:ideal_sheaf_diagonal} requires a regular base, we prove a complementary result for a singular base if the unit on the structure sheaf is an isomorphism.
 
\begin{proposition}
    \label{introprop:explicit_loci_diagonal}
    Let $S$ be a quasi-compact quasi-separated quasi-excellent algebraic space. 
    Suppose $f\colon Y \to S$ is a proper Gorenstein morphism of algebraic spaces such that the canonical morphism $\mathcal{O}_S \to \mathbf{R}f_\ast \mathcal{O}_Y$ is an isomorphism.
    Denote by $\mathcal{I}_{\Delta_f}$ the ideal sheaf of the diagonal $\Delta_f \colon Y \to Y\times_S Y$.
    Then:
    \begin{displaymath}
        \operatorname{fm}_{\operatorname{sg}}(\mathcal{I}_{\Delta_f}) = \operatorname{FM}_{\operatorname{sg}}(\mathcal{I}_{\Delta_f}) = \operatorname{reg}(S).
    \end{displaymath}
\end{proposition}
 
See \Cref{prop:explicit_loci_diagonal}. 
Thus, the loci for the ideal sheaf of the diagonal are explicit, and they detect regularity of the base.
 
\begin{corollary}
    \label{cor:autoequivalence_detects_regularity}
    Let $S$ be a quasi-compact quasi-separated quasi-excellent algebraic space. 
    Then the following are equivalent:
    \begin{enumerate}
        \item $S$ is regular 
        \item there exists $f\colon Y \to S$ as in \Cref{introprop:explicit_loci_diagonal} such that $\Phi_{\mathcal{I}_{\Delta_f}}^{\operatorname{sg}}$ is an autoequivalence.
    \end{enumerate}
\end{corollary}
 
\subsection{Arithmetic families}
\label{sec:intro_arithmetic}
 
In \Cref{sec:arithmetic_families}, we study the loci in arithmetic families.
Under the hypotheses of \Cref{introthm:openness_singular_fm}, let $S$ be an integral Noetherian scheme of Krull dimension one.
If $\Phi_E^{\operatorname{sg}}$ is fully faithful (resp.\ an equivalence) at the generic point, then it is fully faithful (resp.\ an equivalence) at all but finitely many closed points.
See \Cref{cor:dimension_one_base}.
 
In \Cref{ex:arithmetic_conic_bundle}, we give a conic bundle $f\colon Y\to\operatorname{Spec}(\mathbb{Z}[\sqrt{-3}])$ as in \Cref{introprop:explicit_loci_diagonal}, where both loci are the complement of the unique prime $\mathfrak{p}_2$ over $2$.
The fiber of $f$ over $\mathfrak{p}_2$ is a regular conic.
Hence, the failure at $\mathfrak{p}_2$ is not visible on fibers and comes from the singularity of the base.
Moreover, the loci contain the prime $(5)$, over which $Y$ has an $A_1$-singularity of the form $xy=5^2$.
Hence, the loci are strictly larger than the open subset of points over which $Y$ is regular.
See \Cref{cor:regular_total_spaces}.
In contrast, the loci of \cite[Theorem 1.3]{GuisadoVillaalgordo/Lank:2026}, where $\Phi_{\mathcal{I}_{\Delta_f}}$ is fiberwise fully faithful or an equivalence on $D^b_{\operatorname{coh}}$, are empty.
 
\subsection{Organization}
\label{sec:intro_organization}
 
\Cref{sec:preliminaries} recalls background on generation and algebraic spaces.
\Cref{sec:generators} studies classical generators for bounded derived categories of Noetherian algebraic spaces.
\Cref{sec:auxiliary} collects results on integral transforms between singularity categories.
\Cref{sec:germs} introduces the loci $\operatorname{fm}_{\operatorname{sg}}(E)$ and $\operatorname{FM}_{\operatorname{sg}}(E)$.
\Cref{sec:opennes_base_change} proves the main results.
\Cref{sec:examples} discusses examples, and proves \Cref{prop:ideal_sheaf_diagonal,introprop:explicit_loci_diagonal} and \Cref{cor:autoequivalence_detects_regularity}.
\Cref{sec:arithmetic_families} computes the loci in arithmetic families.

\begin{ack}
    The author appreciates discussions with El\'{i}as {Guisado Villalgordo}, Paolo Stellari, and Fei Peng on topics relevant to this work.
    The author was supported by the ERC Advanced Grant 101095900-TriCatApp.
\end{ack}

\paragraph*{\textbf{AI disclosure}}

The author used large language models (LLMs) for assistance with proofreading (e.g.\ grammar, notation, exposition) and checking proofs (e.g.\ finding gaps, pointing out errors).
The main results and methods are due to the author.

\section{Preliminaries}
\label{sec:preliminaries}

In our work, a `natural' morphism refers to something which is functorial, whereas a `canonical' morphism refers to being constructed and not necessarily functorial.

\subsection{Generation}
\label{sec:preliminaries_generation}

We discuss a form of generation for triangulated categories. 
See \cite{Bondal/VandenBergh:2003}. 
Let $\mathcal{T}$ be a triangulated category with shift functor $[1]\colon \mathcal{T} \to \mathcal{T}$. 
Fix a subcategory $\mathcal{S} \subseteq \mathcal{T}$. 
We say that $\mathcal{S}$ is \textbf{thick} if it is a triangulated subcategory which is closed under direct summands. 
Set $\langle \mathcal{S} \rangle$ to be the smallest thick subcategory of $\mathcal{T}$ containing $\mathcal{S}$. 
Write $\langle G \rangle := \langle \mathcal{S} \rangle$ if $\mathcal{S}$ consists of a single object $G$. 
Denote by $\operatorname{add}( \mathcal{S})$ the smallest strictly full subcategory of $\mathcal{T}$ containing $\mathcal{S}$ that is closed under shifts, finite coproducts, and direct summands. 
Inductively, define $\langle  \mathcal{S} \rangle_0$ to consist of all zero objects, $\langle \mathcal{S} \rangle_1 := \operatorname{add}(\mathcal{S})$, and 
\begin{displaymath}
    \langle  \mathcal{S} \rangle_n := \operatorname{add} \{ \operatorname{cone}(\phi) \mid \phi \in \operatorname{Hom}_{\mathcal{T}} (\langle  \mathcal{S} \rangle_{n-1}, \langle  \mathcal{S}\rangle_1) \}.
\end{displaymath}
It can be verified that $\langle  \mathcal{S} \rangle = \cup_{n=0}^\infty \langle \mathcal{S} \rangle_n$. 
If $E\in \langle  \mathcal{S}\rangle$, we say $E$ is \textbf{finitely built by $\mathcal{S}$}.
A \textbf{classical generator} of $\mathcal{T}$ is an object $G\in \mathcal{T}$ such that $\langle G \rangle =  \mathcal{T}$.
Recall that an exact functor $F\colon \mathcal{T} \to \mathcal{T}^\prime$ between triangulated categories is \textbf{essentially dense} if for every $S\in \mathcal{T}^\prime$ there exists $T\in \mathcal{T}$ such that $S\in \langle F(T) \rangle_1$.

\subsection{Algebraic spaces}
\label{sec:preliminaries_stacks}

We follow \cite{StacksProject}. 
Throughout, we work over a fixed base scheme. 
Let $X$ be a quasi-compact quasi-separated algebraic space. 

\subsubsection{Notions}
\label{sec:prelim_notions}

An \textbf{\'{e}tale presentation} of $X$ is an \'{e}tale, finitely presented, surjective morphism to $X$ from a scheme. 
Denote by $|X|$ the underlying topological space of $X$. 
It consists of equivalence classes of morphisms from fields to $X$ \cite[\href{https://stacks.math.columbia.edu/tag/03BT}{Tag 03BT}]{StacksProject}. 
A morphism from a field to a quasi-separated algebraic space is an affine morphism \cite[\href{https://stacks.math.columbia.edu/tag/09TF}{Tag 09TF}]{StacksProject}.
Moreover, as $X$ is quasi-separated, a morphism to $X$ from a quasi-affine scheme is quasi-affine \cite[Lemma 2.18]{GuisadoVillaalgordo/Lank:2026}.

\subsubsection{Categories \& functors}
\label{sec:prelim_categories}

$\operatorname{Mod}(X)$ is the Grothendieck abelian category of sheaves of $\mathcal{O}_X$-modules on the small \'{e}tale site $X_{\textrm{\'{e}tale}}$ of $X$ \cite[\href{https://stacks.math.columbia.edu/tag/03LT}{Tag 03LT}]{StacksProject}. 
$\operatorname{Qcoh}(X)$ is the full subcategory of $\operatorname{Mod}(X)$ consisting of quasi-coherent sheaves. 
$D(X) := D(\operatorname{Mod}(X))$ is the derived category of $\operatorname{Mod}(X)$. 
$D_{\operatorname{qc}}(X)$ is the full subcategory of $D(X)$ consisting of complexes with quasi-coherent cohomology sheaves. 
$\operatorname{Perf}(X)$ is the full subcategory of perfect complexes in $D_{\operatorname{qc}}(X)$. 
If $X$ is Noetherian, $\operatorname{coh}(X)$ is the full subcategory of $\operatorname{Mod}(X)$ consisting of coherent sheaves and $D^b_{\operatorname{coh}}(X)$ denotes the full subcategory of $D(X)$ consisting of bounded pseudocoherent complexes.

Let $f\colon Y \to X$ be a morphism of quasi-compact quasi-separated algebraic spaces. 
By \cite[\href{https://stacks.math.columbia.edu/tag/07A6}{Tag 07A6}]{StacksProject}, there exists an adjoint pair 
\begin{displaymath}
    \mathbf{L}f^\ast \colon D(X) \leftrightarrows D(Y) \colon \mathbf{R}f_\ast.
\end{displaymath}
By \cite[\href{https://stacks.math.columbia.edu/tag/08FA}{Tags 08FA} \& \href{https://stacks.math.columbia.edu/tag/08F4}{08F4}]{StacksProject}, $\mathbf{L}f^\ast$ and $\mathbf{R}f_\ast$ preserve quasi-coherent cohomology.
Moreover, \cite[\href{https://stacks.math.columbia.edu/tag/07A6}{Tag 07A6}]{StacksProject} says the restrictions of $\mathbf{L}f^\ast$ and $\mathbf{R}f_\ast$ form an adjoint pair on $D_{\operatorname{qc}}$. 
Applying \cite[\href{https://stacks.math.columbia.edu/tag/07A4}{Tag 07A4}]{StacksProject}, $\mathbf{L}f^\ast$ is monoidal. 
By \cite[\href{https://stacks.math.columbia.edu/tag/0E55}{Tag 0E55}]{StacksProject}, $\mathbf{R}f_\ast$ admits a right adjoint $f^\times$ on $D_{\operatorname{qc}}$.
If $f\colon Y \to X$ is a proper morphism of Noetherian algebraic spaces, then the so-called `upper shriek' functor $f^!$ on $D_{\operatorname{qc}}$ exists. 
See \cite{Neeman:2023} and \cite[Appendix C]{GuisadoVillaalgordo/Lank:2026}. (Due to Neeman's improvement, $f^!$ is defined on all of $D_{\operatorname{qc}}$, and not only on $D^+_{\operatorname{qc}}$.)

\subsubsection{Support}
\label{sec:preliminaries_stacks_support}

For $M \in \operatorname{Qcoh}(X)$, $\operatorname{Supp}(M) :=|p|(\operatorname{Supp}(p^\ast M)) \subseteq |X|$ where $p \colon U \to X$ is any \'{e}tale presentation from a scheme.
By \cite[\href{https://stacks.math.columbia.edu/tag/07TY}{Tag 07TY}]{StacksProject}, $\operatorname{Supp}$ is independent of the \'{e}tale presentation. 
It coincides with the notion of support for abelian sheaves on the small \'{e}tale site \cite[\href{https://stacks.math.columbia.edu/tag/04K7}{Tag 04K7}]{StacksProject}.
More generally, if $E \in D_{\operatorname{qc}}(X)$, 
\begin{displaymath}
    \operatorname{Supp}(E) := \bigcup_{j \in \mathbb{Z}} \operatorname{Supp}(\mathcal{H}^j(E)) \subseteq |X|.
\end{displaymath}
Given a closed subset $Z \subseteq |X|$, $E\in D_{\operatorname{qc}}(X)$ is \textbf{supported on $Z$} if $\operatorname{Supp}(E) \subseteq Z$. 
If $E\in \operatorname{Qcoh}(X)$ and is of finite type, then $\operatorname{Supp}(E)\subseteq |X|$ is closed \cite[\href{https://stacks.math.columbia.edu/tag/07TZ}{Tag 07TZ}]{StacksProject}. 
Denote by $D_{\operatorname{qc},Z}(X)$ the strictly full subcategory of $D_{\operatorname{qc}}(X)$ consisting of objects supported on $Z$. 
If $X$ is locally Noetherian, quasi-coherent sheaves of finite type coincide with coherent sheaves \cite[\href{https://stacks.math.columbia.edu/tag/07UB}{Tag 07UB}]{StacksProject}.
If $X$ is Noetherian, then $D^b_{\operatorname{coh},Z}(X) := D^b_{\operatorname{coh}}(X) \cap D_{\operatorname{qc},Z}(X)$.

\subsubsection{Perfect complexes}
\label{sec:prelim_perfects}

Perfect complexes are defined on the small \'{e}tale site of $X$. 
See \cite[\href{https://stacks.math.columbia.edu/tag/08G4}{Tag 08G4}]{StacksProject}. 
In particular, a complex is \textbf{strictly perfect} if it is a bounded complex whose terms are direct summands of finite free modules, whereas a complex is \textbf{perfect} if it is locally strictly perfect. 
For any $Z\subseteq |X|$ closed such that $|X|\setminus Z$ is quasi-compact: perfect complexes of $D_{\operatorname{qc},Z}(X)$ coincide with compacts of $D_{\operatorname{qc},Z}(X)$ and $D_{\operatorname{qc},Z}(X)$ is singly compactly generated \cite[\href{https://stacks.math.columbia.edu/tag/0AED}{Tags 0AED}, \href{https://stacks.math.columbia.edu/tag/09M8}{09M8}, \& \href{https://stacks.math.columbia.edu/tag/0AEC}{0AEC}]{StacksProject}.

\subsubsection{Internal homs}
\label{sec:prelim_internal_homs}

Fix $E,G\in D(X)$. 
The derived tensor product $E\otimes^{\mathbf{L}} G$ and the derived sheaf Hom $\operatorname{\mathbb{R}\mathcal{H}\! \mathit{om}}(E,G)$ are objects of $D(X)$. 
Here $(-)\otimes^{\mathbf{L}} E$ is left adjoint to $\operatorname{\mathbb{R}\mathcal{H}\! \mathit{om}}(E,-)$ on $D(X)$. 
By \cite[\href{https://stacks.math.columbia.edu/tag/08F5}{Tags 08F5} \& \href{https://stacks.math.columbia.edu/tag/0FPQ}{0FPQ}]{StacksProject}, $D_{\operatorname{qc}}$ is a symmetric monoidal category with $\otimes^{\mathbf{L}}$, i.e.\ $E\otimes^{\mathbf{L}} G \in D_{\operatorname{qc}}(X)$ for all $E,G\in D_{\operatorname{qc}}(X)$. 
On the other hand, $\operatorname{\mathbb{R}\mathcal{H}\! \mathit{om}}(E,G)$ need not have quasi-coherent cohomology even if $E$ and $G$ belong to $D_{\operatorname{qc}}(X)$. 
However, the formation of $\operatorname{\mathbb{R}\mathcal{H}\! \mathit{om}}(E,-)$ is \'{e}tale local \cite[\href{https://stacks.math.columbia.edu/tag/04LX}{Tags 04LX} \& \href{https://stacks.math.columbia.edu/tag/08JB}{08JB}]{StacksProject}. 
The endofunctor $(-)\otimes^{\mathbf{L}} E$ on $D_{\operatorname{qc}}(X)$ preserves small coproducts, and so \cite[Theorem 8.4.4]{Neeman:2001} implies this functor admits a right adjoint $\operatorname{\mathbf{R}\mathcal{H}\! \mathit{om}}(E,-)$ on $D_{\operatorname{qc}}(X)$. 
See \cite[\S 2.3.5]{GuisadoVillaalgordo/Lank:2026}.

\section{Generators}
\label{sec:generators}

\begin{remark}
    We prove a straightforward variation of \cite[Theorem 1.1]{Dey/Lank:2024a} on the small \'{e}tale site for Noetherian algebraic spaces. 
    It is included for convenience.
\end{remark}

\begin{lemma}
    \label{lem:reduced_subscheme_generation}
    Let $X$ be a Noetherian algebraic space.
    Suppose $i\colon Z \to X$ is a closed immersion.
    Then $D^b_{\operatorname{coh},|i|(|Z|)}(X) = \langle \mathbf{R}i_\ast D^b_{\operatorname{coh}}(Z) \rangle$. 
\end{lemma}

\begin{proof}
    Let $\mathcal{I}$ be the corresponding ideal sheaf of $i\colon Z \to X$.
    By \cite[Lemma 3.14]{GuisadoVillaalgordo/Lank:2026}, every $E\in D^b_{\operatorname{coh},|i|(|Z|)}(X)$ is finitely built by its cohomology sheaves.
    Hence, it suffices to show that every coherent sheaf $E$ supported on $|i|(|Z|)$ belongs to $\langle \mathbf{R}i_\ast D^b_{\operatorname{coh}}(Z) \rangle$.
    By \cite[\href{https://stacks.math.columbia.edu/tag/07UK}{Tag 07UK}]{StacksProject}, there exists an $n\geq 0$ such that $\mathcal{I}^n E\cong 0$. 
    This yields a filtration
    \begin{displaymath}
        0 \cong \mathcal{I}^n E \subseteq \mathcal{I}^{n-1}E \subseteq \cdots \subseteq \mathcal{I}E \subseteq E.
    \end{displaymath}
    Here each subquotient $\mathcal{I}^{j-1} E/ \mathcal{I}^j E$ is annihilated by $\mathcal{I}$. 
    Consequently, by \cite[\href{https://stacks.math.columbia.edu/tag/08AM}{Tags 08AM} \& \href{https://stacks.math.columbia.edu/tag/04CJ}{04CJ}]{StacksProject}, each subquotient belongs to $\langle \mathbf{R}i_\ast D^b_{\operatorname{coh}}(Z) \rangle$.
\end{proof}

\begin{lemma}
    \label{lem:classical_generation_from_irreducible_components}
    Let $X$ be a Noetherian algebraic space. 
    Suppose $|X| = \cup_{r=1}^n |i_r|(|Z_r|)$ where $i_r\colon Z_r \hookrightarrow X$ are closed immersions such that each $D^b_{\operatorname{coh}}(Z_r)$ admits a classical generator $G_r$. 
    Then $D^b_{\operatorname{coh}}(X)$ is classically generated by $G:=\oplus_{r=1}^n \mathbf{R} (i_r)_\ast G_r$.
\end{lemma}

\begin{proof}
    We induct on $n\geq 1$.
    The case $n=1$ follows from \Cref{lem:reduced_subscheme_generation}.
    Assume the claim holds for Noetherian algebraic spaces which admit a closed cover given by at most $n-1$ closed immersions.
    Let $Z^\prime$ be the scheme-theoretic union of $Z_1,\ldots,Z_{n-1}$.
    For each $r=1,\ldots,n-1$, the morphism $i_r \colon Z_r \to X$ factors through the closed immersion $i^\prime \colon Z^\prime \to X$ by some closed immersion $i^\prime_r \colon Z_r \to Z^\prime$. 
    Note that $i_n$ need not necessarily factor through $i^\prime$.
    See \cite[\href{https://stacks.math.columbia.edu/tag/0CZ1}{Tags 0CZ1} \& \href{https://stacks.math.columbia.edu/tag/0CYZ}{0CYZ}]{StacksProject}.
    The induction hypothesis applies to $Z^\prime$, and so,
    \begin{displaymath}
        D^b_{\operatorname{coh}}(Z^\prime) = \langle \bigoplus_{r=1}^{n-1} \mathbf{R} (i^\prime_r)_\ast G_r \rangle.
    \end{displaymath}
    By \Cref{lem:reduced_subscheme_generation},
    \begin{displaymath}
        D^b_{\operatorname{coh},|i^\prime|(|Z^\prime|)}(X) = \langle \bigoplus_{r=1}^{n-1} \mathbf{R} (i_r)_\ast G_r \rangle.
    \end{displaymath}
    Choose $E\in D^b_{\operatorname{coh}}(X)$. 
    By \cite[Lemma 3.14]{GuisadoVillaalgordo/Lank:2026}, any object in $D^b_{\operatorname{coh}}(X)$ is finitely built by its cohomology sheaves.
    Hence, we can assume that $E$ is a coherent sheaf. 
    Applying \cite[\href{https://stacks.math.columbia.edu/tag/07UP}{Tag 07UP}]{StacksProject}, there exists a short exact sequence 
    \begin{displaymath}
        0 \to E_1 \to E \to E_2 \to 0
    \end{displaymath}
    where $\operatorname{Supp}(E_1) \subseteq |i^\prime| (|Z^\prime|)$ and $\operatorname{Supp}(E_2) \subseteq |i_n| (|Z_n|)$. 
    Hence, we may further impose that $E\in D^b_{\operatorname{coh},|i^\prime| (|Z^\prime|)}(X)$ or $E\in D^b_{\operatorname{coh},|i_n| (|Z_n|)}(X)$.
    By \Cref{lem:reduced_subscheme_generation}, $E\in \langle G \rangle$.
\end{proof}

\begin{definition}
    Let $X$ be a locally Noetherian algebraic space.
    Denote by $\operatorname{reg}(X)$ the collection of $p\in |X|$ for which $X$ is regular at $p$ (see \cite[\href{https://stacks.math.columbia.edu/tag/0AH9}{Tag 0AH9}]{StacksProject}).
    We say that $X$ is
    \begin{enumerate}
        \item \textbf{$J\textrm{-}0$} if the regular locus $\operatorname{reg}(X)$ contains a nonempty open;
        \item \textbf{$J\textrm{-}1$} if the regular locus $\operatorname{reg}(X)$ is open;
        \item \textbf{$J\textrm{-}2$} if for every morphism $Y\to X$ which is locally of finite presentation from a scheme $Y$, the regular locus $\operatorname{reg}(Y)$ is open.
    \end{enumerate}
\end{definition}

\begin{remark}
    \label{rmk:regular_locus_by_presentation}
    Let $X$ be a Noetherian algebraic space.
    Suppose $s\colon U \to X$ is an \'{e}tale presentation.
    Applying \cite[\href{https://stacks.math.columbia.edu/tag/0AH9}{Tag 0AH9}]{StacksProject}, $X$ is regular at a point $p\in |X|$ if, and only if, $U$ is regular at some (equivalently, every) point of $|s|^{-1}(p)$.
    Hence, $|s|(\operatorname{reg}(U))=\operatorname{reg}(X)$.
\end{remark}

\begin{lemma}
    \label{lem:existence_classical_generator_implies_open_cohomological_regular_locus}
    Let $X$ be a Noetherian algebraic space. 
    If $D^b_{\operatorname{coh}}(X)$ admits a classical generator, then $\operatorname{reg}(X)$ is open.
\end{lemma}

\begin{proof}
    Choose a classical generator $G$ for $D^b_{\operatorname{coh}}(X)$. 
    Pick an \'{e}tale presentation $s\colon U \to X$ from an affine scheme.
    By \cite[Proposition 1.4]{Lank:2026c}, 
    \begin{displaymath}
        \mathbf{L}s^\ast \colon D^b_{\operatorname{coh}}(X) \to D^b_{\operatorname{coh}}(U)
    \end{displaymath}
    is essentially dense. 
    Hence, $\mathbf{L}s^\ast G$ is a classical generator for $D^b_{\operatorname{coh}}(U)$. 
    By \cite[Lemma 2.9]{Iyengar/Takahashi:2019}, $\operatorname{reg}(U)$ is open.  
    By \Cref{rmk:regular_locus_by_presentation}, $|s|(\operatorname{reg}(U))=\operatorname{reg}(X)$. 
    Since $s$ is \'{e}tale, $|s|$ is an open morphism.
    This completes the proof.
\end{proof}

\begin{remark}
    Recall that a Noetherian algebraic space $X$ being integral is equivalent to $X$ being reduced and $|X|$ irreducible \cite[\href{https://stacks.math.columbia.edu/tag/0AD3}{Tag 0AD3}]{StacksProject}. 
\end{remark}

\begin{lemma}
    \label{lem:generation_verdier_quotient}
    Let $\mathcal{T}$ be a triangulated category.
    Fix a thick subcategory $\mathcal{K}\subseteq\mathcal{T}$.
    Denote by $q\colon \mathcal{T}\to\mathcal{T}/\mathcal{K}$ the Verdier quotient.
    Choose $E,G\in \mathcal{T}$.
    If $q(E)\in\langle q(G) \rangle_n$ for some $n\geq 1$, then $E\in\langle \{G\oplus K: K\in \mathcal{K}\} \rangle_{6n-3}$.
\end{lemma}

\begin{proof}
    This is proved by induction on $n$.
    For $m\geq 0$, write $\langle G \oplus \mathcal{K} \rangle_m := \langle \{G\oplus K: K\in \mathcal{K}\} \rangle_m$.
    Consider the case $n=1$.
    Then $q(E)\in\langle q(G) \rangle_1$. 
    Hence, there exist an object $E^\prime\in\mathcal{T}$ and an isomorphism $q(E\oplus E^\prime) \cong \bigoplus_{i\in I} q(G)[s_i]$ (for some finite set $I$). 
    This corresponds to a roof in $\mathcal{T}$,
    \begin{displaymath}
        \begin{tikzcd}
            & F & \\
            {E\oplus E^\prime} && {\bigoplus_{i\in I} G [s_i]}
            \arrow["f"', from=1-2, to=2-1]
            \arrow["g", from=1-2, to=2-3]
        \end{tikzcd}
    \end{displaymath}
    where $\operatorname{cone}(f)$, $\operatorname{cone}(g)\in\mathcal{K}$. 
    Consequently, it follows that $F\in \langle G \oplus \mathcal{K} \rangle_2$, and so, $E\in \langle G \oplus \mathcal{K} \rangle_3$.

    Assume that we have proved the claim for all $\leq n$.
    Choose $q(E)\in\langle q(G) \rangle_{n+1}$. 
    There exists a distinguished triangle
    \begin{equation}
        \label{eq:finitely_building_verdier_quotients}
        q(A) \to q(E \oplus E^\prime) \to q(B) \to q(A)[1]
    \end{equation}
    where $q(A)\in\langle q(G) \rangle_n$ and $q(B)\in\langle q(G) \rangle_1$.
    Moreover, there exists a distinguished triangle in $\mathcal{T}$,
    \begin{displaymath}
        A^\prime \to E^{\prime\prime} \to B^\prime \to A^\prime[1]
    \end{displaymath}
    whose image under $q \colon \mathcal{T} \to \mathcal{T}/\mathcal{K}$ is isomorphic to the distinguished triangle $\eqref{eq:finitely_building_verdier_quotients}$.
    Hence, $q(A^\prime)\cong q(A)\in\langle q(G) \rangle_n$ and $q(B^\prime)\cong q(B)\in\langle q(G) \rangle_1$.
    By the induction hypothesis, $A^\prime\in\langle G \oplus \mathcal{K} \rangle_{6n-3}$ and $B^\prime\in\langle G \oplus \mathcal{K} \rangle_3$. 
    Therefore, ${E}^{\prime\prime} \in \langle G \oplus \mathcal{K} \rangle_{6n+1}$.
    A similar argument with a roof corresponding to $q(E\oplus E^\prime)\cong q(E^{\prime\prime})$ gives us $E\in\langle G \oplus \mathcal{K} \rangle_{6n+3}$.
\end{proof}

\begin{proposition}
    \label{prop:j_condition_stacks}
    For a Noetherian algebraic space $X$, the following are equivalent:
    \begin{enumerate}
        \item\label{item:jcond1} every integral closed subspace $Z$ of $X$ is $J\textrm{-}0$,
        \item\label{item:jcond2} $D^b_{\operatorname{coh}}(Z)$ admits a classical generator for every integral closed subspace $Z$ of $X$.
    \end{enumerate}
    Moreover, if these equivalent conditions hold, $D^b_{\operatorname{coh}}(Y)$ admits a classical generator for any closed subspace $Y$ of $X$.
\end{proposition}

\begin{proof}
    The last claim follows from the first claim with \Cref{lem:classical_generation_from_irreducible_components}.
    It is clear from \Cref{lem:existence_classical_generator_implies_open_cohomological_regular_locus} that $\eqref{item:jcond2}\implies \eqref{item:jcond1}$.
    We check the converse. 
    Let $Y \to X$ be a closed immersion with integral source. 
    We will use Noetherian induction on the closed subsets $W$ of $|Y|$ to show that $D^b_{\operatorname{coh}}(W^\prime)$ admits a classical generator where $i\colon W^\prime \to Y$ is a closed immersion with $|i|(|W^\prime|) = W$. 
    There is nothing to show when $W$ is empty, and so, we assume it is not.
    Assume that we have proved the claim for all properly contained closed subsets of $W$. 
    By \Cref{lem:reduced_subscheme_generation}, it suffices to prove the claim for $D^b_{\operatorname{coh}}$ of $W$ with the reduced induced algebraic space structure $W_{red} \to Y$. 
    Since $W_{red}$ is a Noetherian algebraic space, $|W_{red}|$ has a finite number of irreducible closed components. 
    Endow each of these with their reduced induced algebraic space structure. 
    Applying \Cref{lem:classical_generation_from_irreducible_components}, we may reduce to the case in which $W_{red}$ is integral. 
    In such a case, the hypothesis ensures $W_{red}$ is $J\textrm{-}0$. 
    If $W_{red}$ is regular, \cite[Lemma 5.28]{GuisadoVillaalgordo/Lank:2026} says $D^b_{\operatorname{coh}}(W_{red})=\operatorname{Perf}(W_{red})$ admits a classical generator \cite[\href{https://stacks.math.columbia.edu/tag/09IY}{Tags 09IY} \& \href{https://stacks.math.columbia.edu/tag/09SR}{09SR}]{StacksProject}. 
    Consider the case where $W_{red}$ is not regular. 
    By the $J\textrm{-}0$ hypothesis, there exists a nonempty open immersion $j\colon U \to W_{red}$ with $|U|\subseteq \operatorname{reg}(W_{red})\neq |W_{red}|$. 
    Since $U$ is regular, $D^b_{\operatorname{coh}}(U)$ admits a classical generator (see loc.\ cit.) $G_U$.
    Let $Z \subseteq W_{red}$ be the closed subspace associated to $|W_{red}|\setminus |U|$ with its reduced induced structure.
    By \cite[Lemma A.4]{Lank:2026c}, there exists a Verdier localization sequence
    \begin{displaymath}
        D^b_{\operatorname{coh},|Z|}(W_{red}) \to D^b_{\operatorname{coh}}(W_{red}) \xrightarrow{\mathbf{L}j^\ast} D^b_{\operatorname{coh}}(U).
    \end{displaymath}
    Hence, there exists $G^\prime_U\in D^b_{\operatorname{coh}}(W_{red})$ such that $\mathbf{L}j^\ast G^\prime_U\cong G_U$.
    By the induction hypothesis, $D^b_{\operatorname{coh}}(Z)$ admits a classical generator. 
    Applying \Cref{lem:reduced_subscheme_generation}, it follows that $D^b_{\operatorname{coh},|Z|}(W_{red})$ admits a classical generator $G$. 
    Thus, \Cref{lem:generation_verdier_quotient} implies that $D^b_{\operatorname{coh}}(W_{red})$ admits a classical generator.  
    Indeed, it follows that 
    \begin{displaymath}
        D^b_{\operatorname{coh}}(W_{red})\subseteq \langle D^b_{\operatorname{coh},|Z|}(W_{red}) \oplus G^\prime_U \rangle \subseteq \langle G\oplus G^\prime_U \rangle.
    \end{displaymath} 
    This completes the proof.
\end{proof}

\begin{corollary}
    \label{cor:classical_generator_for_J2}
    For any $J\textrm{-}2$ quasi-compact quasi-separated algebraic space $X$, $D^b_{\operatorname{coh}}(X)$ admits a classical generator. 
\end{corollary}

\begin{proof}
    Let $i\colon Z \to X$ be a closed immersion from an integral algebraic space. 
    Choose an \'{e}tale presentation $s\colon V \to Z$. 
    Then $i\circ s$ is of finite presentation, and so $\operatorname{reg}(V)$ is open.
    By \Cref{rmk:regular_locus_by_presentation}, $\operatorname{reg}(Z)$ is open because $|s|$ is an open morphism.
    As $Z$ is integral, its generic point is contained in $\operatorname{reg}(Z)$. 
    Hence, $Z$ is $J\textrm{-}0$. 
    Thus, the claim follows from \Cref{prop:j_condition_stacks}.
\end{proof}

\begin{lemma}
    \label{lem:quasi_excellent_etale_local}
    The property of being quasi-excellent is local in the \'{e}tale topology for schemes.
\end{lemma}

\begin{proof}
    We apply \cite[\href{https://stacks.math.columbia.edu/tag/0349}{Tag 0349}]{StacksProject}. 
    By \cite[\href{https://stacks.math.columbia.edu/tag/021N}{Tags 021N} \& \href{https://stacks.math.columbia.edu/tag/034C}{034C}]{StacksProject}, being locally Noetherian is local for the \'{e}tale topology.
    By  \cite[\href{https://stacks.math.columbia.edu/tag/01TT}{Tags 01TT}, \href{https://stacks.math.columbia.edu/tag/01TW}{01TW}, \& \href{https://stacks.math.columbia.edu/tag/0HAC}{0HAC}]{StacksProject}, quasi-excellence is local in the Zariski topology. 

    Consider an \'{e}tale morphism $Y \to X$ of affine schemes where $X$ is quasi-excellent. 
    By \cite[\href{https://stacks.math.columbia.edu/tag/01VE}{Tags 01VE}, \href{https://stacks.math.columbia.edu/tag/01TW}{01TW}, \& \href{https://stacks.math.columbia.edu/tag/0HAC}{0HAC}]{StacksProject}, $Y$ is quasi-excellent.

    Lastly, let $Y\to X$ be an \'{e}tale surjective morphism of affine schemes with $Y$ quasi-excellent.
    Write $X=\operatorname{Spec}(R)$.
    The first paragraph ensures that $R$ is Noetherian.
    Since $R\to\mathcal{O}_Y(Y)$ is \'{e}tale and faithfully flat, it is faithfully flat and regular, and so $R$ is a $G$-ring \cite[Theorem 32.2(ii)]{Matsumura:1989}.
    It remains to show that $X$ is $J\textrm{-}2$.
    Choose a morphism of schemes $X^\prime \to X$ which is locally of finite presentation.
    Form the fibered square 
    \begin{displaymath}
        \begin{tikzcd}
            {Y^\prime} & {X^\prime} \\
            Y & {X.}
            \arrow["t", from=1-1, to=1-2]
            \arrow[from=1-1, to=2-1]
            \arrow[from=1-2, to=2-2]
            \arrow[from=2-1, to=2-2]
        \end{tikzcd}
    \end{displaymath}
    By base change, $t$ is \'{e}tale and surjective.
    Applying \cite[\href{https://stacks.math.columbia.edu/tag/0AH9}{Tag 0AH9}]{StacksProject}, $X^\prime$ is regular at a point $p\in |X^\prime|$ if, and only if, $Y^\prime$ is regular at some (equivalently, every) point of $|t|^{-1}(p)$.
    Hence, $|t|(\operatorname{reg}(Y^\prime))=\operatorname{reg}(X^\prime)$.
    Since $t$ is \'{e}tale, $|t|$ is an open morphism on the underlying topological spaces.
    Moreover, $Y$ being quasi-excellent implies $Y^\prime$ is quasi-excellent \cite[\href{https://stacks.math.columbia.edu/tag/0HAC}{Tag 0HAC}]{StacksProject}, and so $\operatorname{reg}(Y^\prime)$ is an open subset of $|Y^\prime|$. 
    Thus, $\operatorname{reg}(X^\prime)$ is open, which shows that $X$ is $J\textrm{-}2$.
\end{proof}

\begin{definition}
    Let $X$ be an algebraic space. 
    We say that $X$ is \textbf{quasi-excellent} if there exists an \'{e}tale surjective morphism to $X$ from a quasi-excellent scheme. 
\end{definition}

\begin{remark}
    \label{rmk:quasi-excellet_is_j-2}
    By \cite[\href{https://stacks.math.columbia.edu/tag/03E6}{Tags 03E6} \& \href{https://stacks.math.columbia.edu/tag/03E8}{03E8}]{StacksProject}, an algebraic space $X$ being quasi-excellent is equivalent to every \'{e}tale morphism to $X$ from a scheme having quasi-excellent source.
    Moreover, any quasi-compact quasi-separated quasi-excellent algebraic space $X$ is $J\textrm{-}2$. 
    Indeed, a quasi-excellent scheme is $J\textrm{-}2$ \cite[\href{https://stacks.math.columbia.edu/tag/0HA8}{Tags 0HA8} \& \href{https://stacks.math.columbia.edu/tag/07R3}{07R3}]{StacksProject}.
    Hence, for any \'{e}tale presentation $U \to X$ by an affine scheme, the source is $J\textrm{-}2$.
    Take any morphism $Y\to X$ which is locally of finite presentation from a scheme. 
    Denote by $p_2\colon U\times_X Y \to Y$ and $p_1 \colon U\times_X Y \to U$ the canonical morphisms from the fiber product formed by $Y\to X$ and $U \to X$. 
    By \cite[Lemma 2.18]{GuisadoVillaalgordo/Lank:2026}, $U\to X$ is quasi-affine. 
    Hence, $U\times_X Y \to Y$ is an \'{e}tale presentation by a scheme. 
    Base change implies $p_1$ is locally of finite presentation. 
    As $U$ is $J\textrm{-}2$, $U\times_X Y$ has an open regular locus. 
    By \Cref{rmk:regular_locus_by_presentation}, $|p_2|(\operatorname{reg}(U\times_X Y)) = \operatorname{reg}(Y)$. 
    Since $p_2$ is \'{e}tale, $|p_2|$ is an open morphism on topological spaces. 
    Thus, $\operatorname{reg}(Y)$ is open.
\end{remark}

\section{Auxiliary}
\label{sec:auxiliary}

\begin{lemma}
    \label{lem:restrict_adjunction_to_Verdier_quotient}
    Let $\mathcal{T}$ and $\mathcal{S}$ be triangulated categories.
    Suppose $\mathcal{A}\subseteq \mathcal{T}$ and $\mathcal{B}\subseteq \mathcal{S}$ are thick subcategories. 
    Consider a pair of adjoint exact functors 
    \begin{displaymath}
        F \colon \mathcal{T} \leftrightarrows \mathcal{S} \colon G.
    \end{displaymath}
    If $F(\mathcal{A})\subseteq \mathcal{B}$ and $G(\mathcal{B})\subseteq \mathcal{A}$, then there exists an induced pair of adjoint exact functors 
    \begin{displaymath}
        \overline{F} \colon \mathcal{T}/\mathcal{A} \leftrightarrows \mathcal{S}/\mathcal{B} \colon \overline{G}.
    \end{displaymath}
    In particular, if $\epsilon$ and $\eta$ are respectively the counit and unit of $F\dashv G$, then the unit and counit of $\overline{F}\dashv \overline{G}$ are respectively induced by $\pi_\mathcal{A} (\eta_U)$ and $\pi_\mathcal{B} (\epsilon_V)$ for all $\pi_\mathcal{A} (U)\in \mathcal{T}/\mathcal{A}$ and $\pi_{\mathcal{B}} (V)\in \mathcal{S}/\mathcal{B}$.
    Here $\pi_{?}$ is the Verdier quotient functor, e.g.\ $\mathcal{T} \to \mathcal{T}/\mathcal{A}$. 
    Furthermore, the strictly full subcategory $\mathcal{T}_1$ of $E\in \mathcal{T}$ such that $\operatorname{cone}(\eta_E) \in \mathcal{A}$ is thick, and similarly for $\mathcal{S}_1$ given by $B\in \mathcal{S}$ such that $\operatorname{cone}(\epsilon_B) \in \mathcal{B}$.
\end{lemma}

\begin{proof}
    By \cite[\href{https://stacks.math.columbia.edu/tag/05RJ}{Tag 05RJ}]{StacksProject}, the hypotheses $F(\mathcal{A})\subseteq\mathcal{B}$ and $G(\mathcal{B})\subseteq\mathcal{A}$ induce exact functors $\overline{F}\colon \mathcal{T}/\mathcal{A}\to\mathcal{S}/\mathcal{B}$ and $\overline{G}\colon \mathcal{S}/\mathcal{B}\to\mathcal{T}/\mathcal{A}$ such that $\overline{F} \circ \pi_{\mathcal{A}}=\pi_{\mathcal{B}} \circ F,$ and $\overline{G} \circ \pi_{\mathcal{B}}=\pi_{\mathcal{A}} \circ G$. 
    Since the Verdier quotient has the same objects as the original category, define
    \begin{displaymath}
        \overline\eta_{\pi_{\mathcal{A}} (U)}:=\pi_{\mathcal{A}}(\eta_U)\colon \pi_{\mathcal{A}}(U)\to(\overline{G}\circ \overline{F})(\pi_{\mathcal{A}}(U))
    \end{displaymath}
    for $\pi_{\mathcal{A}}(U)\in\mathcal{T}/\mathcal{A}$, and
    \begin{displaymath}
        \overline \epsilon_{\pi_{\mathcal{B}}(V)}:=\pi_{\mathcal{B}}(\epsilon_V)\colon  (\overline{F}\circ \overline{G})(\pi_{\mathcal{B}} (V))\to \pi_{\mathcal{B}}(V)
    \end{displaymath}
    for $\pi_{\mathcal{B}}(V)\in\mathcal{S}/\mathcal{B}$.

    We check the naturality of $\overline{\epsilon}$ and $\overline{\eta}$.
    Choose $\pi_{\mathcal{A}}(U) \in \mathcal{T}/\mathcal{A}$ and $\alpha \colon \pi_{\mathcal{A}} (U) \to \pi_{\mathcal{A}} (U^\prime)$ in $\mathcal{T}/\mathcal{A}$. 
    Choose a roof which represents $\alpha$,
    \begin{displaymath}
        \begin{tikzcd}
            & X & \\
            U && {U^\prime}
            \arrow["{\alpha_1}"', from=1-2, to=2-1]
            \arrow["{\alpha_2}", from=1-2, to=2-3]
        \end{tikzcd}
    \end{displaymath}
    where $\operatorname{cone}(\alpha_1)\in \mathcal{A}$. 
    By the naturality of $\eta$, the following diagram commutes,
    \begin{displaymath}
        \begin{tikzcd}
            &&& {GF(X)} & \\
            & X & {GF(U)} \\
            U &&&& {GF(U^\prime)} \\
            && {U^\prime}
            \arrow["{GF(\alpha_1)}", dashed, from=1-4, to=2-3]
            \arrow["{GF(\alpha_2)}", from=1-4, to=3-5]
            \arrow["{\eta_X}", from=2-2, to=1-4]
            \arrow["{\alpha_1}"', from=2-2, to=3-1]
            \arrow["{\alpha_2}", from=2-2, to=4-3]
            \arrow["{\eta_U}"'{pos=0.2}, dashed, from=3-1, to=2-3]
            \arrow["{\eta_{U^\prime}}", from=4-3, to=3-5]
        \end{tikzcd}
    \end{displaymath}
    where $\operatorname{cone}((G \circ F)(\alpha_1))\in \mathcal{A}$.
    Indeed, there exists an isomorphism
    \begin{displaymath}
        \operatorname{cone}((G\circ F)(\alpha_1)) \cong (G\circ F) (\operatorname{cone}(\alpha_1)),
    \end{displaymath}
    and the hypothesis implies $(G \circ F)(\mathcal{A})\subseteq \mathcal{A}$.

    Applying $\pi_\mathcal{A}$ yields 
    \begin{displaymath}
        \begin{aligned}
            \overline{\eta}_{\pi_{\mathcal{A}} (U^\prime)} \circ \alpha
            &= \pi_\mathcal{A} (\eta_{U^\prime}) \circ \pi_\mathcal{A}(\alpha_2) \circ (\pi_\mathcal{A} (\alpha_1))^{-1} 
            \\&= \pi_\mathcal{A} (\eta_{U^\prime} \circ \alpha_2) \circ (\pi_\mathcal{A} (\alpha_1))^{-1} 
            \\&= \pi_\mathcal{A} ( (G \circ F)(\alpha_2) \circ \eta_X) \circ (\pi_\mathcal{A} (\alpha_1))^{-1} 
            \\&= \pi_\mathcal{A} ( (G \circ F)(\alpha_2)) \circ \pi_\mathcal{A}(\eta_X) \circ (\pi_\mathcal{A} (\alpha_1))^{-1} 
            \\&= \pi_\mathcal{A} ( (G \circ F)(\alpha_2)) \circ (\pi_{\mathcal{A}} ((G \circ F)(\alpha_1)))^{-1} \circ \pi_\mathcal{A}(\eta_U)
            \\&= \overline{G}\overline{F}(\alpha) \circ \overline{\eta}_{\pi_\mathcal{A}(U)}.
        \end{aligned}
    \end{displaymath}
    Thus, $\overline{\eta}$ is a natural transformation.
    The same argument proves $\overline{\epsilon}$ is natural.

    We are left to verify the triangle identities. 
    The first follows from a direct computation:
    \begin{displaymath}
       \begin{aligned}
            \overline{\epsilon}_{\overline{F}(\pi_{\mathcal{A}} (U))} \circ \overline{F} (\overline{\eta}_{\pi_{\mathcal{A}} (U)})
            &= \overline{\epsilon}_{\overline{F}(U)} \circ \overline{F} (\pi_\mathcal{A} (\eta_U))
            \\&= \overline{\epsilon}_{\pi_\mathcal{B} (F(U))} \circ \pi_\mathcal{B} (F(\eta_U))
            \\&= \pi_\mathcal{B} (\epsilon_{F(U)}) \circ \pi_\mathcal{B} (F(\eta_U))
            \\&= \pi_\mathcal{B} (\epsilon_{F(U)} \circ F(\eta_U))
            \\&= \pi_\mathcal{B} (1_{F(U)})
            \\&= 1_{\overline{F}(U)}.
       \end{aligned}
    \end{displaymath}
    A similar argument shows that $\overline{G}(\overline{\epsilon}_{\pi_{\mathcal{B}}(V)}) \circ \overline{\eta}_{\overline{G}(\pi_{\mathcal{B}}(V))} = 1_{\overline{G}(V)}$ for all $\pi_{\mathcal{B}}(V)\in \mathcal{S}/\mathcal{B}$.

    Lastly, we prove that the strictly full subcategory of $E\in \mathcal{T}$ such that $\operatorname{cone}(\eta_E) \in \mathcal{A}$ is thick. 
    Denote by $\mathcal{T}_1$ the strictly full subcategory of such objects in $\mathcal{T}$. 
    The argument for $B\in \mathcal{S}$ satisfying $\operatorname{cone}(\epsilon_B) \in \mathcal{B}$ is similar.
    If $A\in \mathcal{T}_1$ and $n\in \mathbb{Z}$, then 
    \begin{displaymath}
        \operatorname{cone}(\eta_{A[n]})\cong \operatorname{cone}(\eta_A)[n]\in \mathcal{A}.
    \end{displaymath}
    Hence, $\mathcal{T}_1$ is closed under shifts. 

    Pick $A,B\in \mathcal{T}_1$ and $\alpha \colon A \to B$. 
    By the naturality of $\eta$, there exists a morphism of distinguished triangles
    \begin{displaymath}
        \begin{tikzcd}
            A & B & {\operatorname{cone}(\alpha)} & {A[1]} \\
            {(G\circ F)(A)} & {(G\circ F)(B)} & {(G\circ F)(\operatorname{cone}(\alpha))} & {(G\circ F)(A)[1].}
            \arrow["\alpha", from=1-1, to=1-2]
            \arrow["{\eta_A}"', from=1-1, to=2-1]
            \arrow[from=1-2, to=1-3]
            \arrow["{\eta_B}"', from=1-2, to=2-2]
            \arrow[from=1-3, to=1-4]
            \arrow["{\eta_{\operatorname{cone}(\alpha)}}"', from=1-3, to=2-3]
            \arrow["{\eta_A [1]}"', from=1-4, to=2-4]
            \arrow[from=2-1, to=2-2]
            \arrow[from=2-2, to=2-3]
            \arrow[from=2-3, to=2-4]
        \end{tikzcd}
    \end{displaymath}
    See e.g.\ the proof of \cite[Lemma 5.3.6]{Neeman:2001} or \cite[Proposition (3.3.8)]{Lipman/Hashimoto:2009}. 
    Applying $\pi_\mathcal{A}$ yields a morphism of distinguished triangles,
    \begin{displaymath}
        \begin{tikzcd}[row sep=1em,column sep=0.5em]
            {\pi_\mathcal{A}(A)} & {\pi_\mathcal{A}(B)} & {\pi_\mathcal{A}(\operatorname{cone}(\alpha))} & {\pi_\mathcal{A}(A)[1]} \\
            {(\pi_\mathcal{A} \circ G \circ F)(A)} & {(\pi_\mathcal{A} \circ G \circ F)(B)} & {(\pi_\mathcal{A} \circ G \circ F)(\operatorname{cone}(\alpha))} & {(\pi_\mathcal{A} \circ G \circ F)(A)[1].}
            \arrow["{\pi_\mathcal{A}(\alpha)}", from=1-1, to=1-2]
            \arrow["{\pi_\mathcal{A}(\eta_A)}"', from=1-1, to=2-1]
            \arrow[from=1-2, to=1-3]
            \arrow["{\pi_\mathcal{A}(\eta_B)}"', from=1-2, to=2-2]
            \arrow[from=1-3, to=1-4]
            \arrow["{\pi_\mathcal{A}(\eta_{\operatorname{cone}(\alpha)})}"', from=1-3, to=2-3]
            \arrow["{\pi_\mathcal{A}(\eta_A) [1]}"', from=1-4, to=2-4]
            \arrow[from=2-1, to=2-2]
            \arrow[from=2-2, to=2-3]
            \arrow[from=2-3, to=2-4]
        \end{tikzcd}
    \end{displaymath}
    Since $\pi_\mathcal{A}(\eta_A)$ and $\pi_\mathcal{A}(\eta_B)$ are isomorphisms, so is $\pi_\mathcal{A}(\eta_{\operatorname{cone}(\alpha)})$. 
    Hence, $\operatorname{cone}(\eta_{\operatorname{cone}(\alpha)}) \in \mathcal{A}$, which shows that $\mathcal{T}_1$ is closed under extensions. 

    Next we check that $\mathcal{T}_1$ is closed under direct summands. 
    Let $I$ be a finite set.
    Choose $E_i \in \mathcal{T}$ with $i\in I$. 
    Assume $\oplus_{i\in I} E_i \in \mathcal{T}_1$. 
    By the naturality of $\eta$, 
    \begin{displaymath}
        \begin{tikzcd}
            {\bigoplus_{i\in I} E_i} & {GF(\bigoplus_{i\in I} E_i)} \\
            {E_i} & {GF(E_i).}
            \arrow["{\eta_{\oplus_{i \in I} E_i}}", from=1-1, to=1-2]
            \arrow["{p_i}", from=2-1, to=1-1]
            \arrow["{\eta_{E_i}}", from=2-1, to=2-2]
            \arrow["{GF(p_i)}"', from=2-2, to=1-2]
        \end{tikzcd}
    \end{displaymath}
    Moreover, the unique (dashed) morphism whose restriction to each $GF(E_i)$ is $GF(p_i)$ makes the following diagram commute:
    \begin{displaymath}
        \begin{tikzcd}
            {\bigoplus_{i\in I} E_i} & {\bigoplus_{i\in I} GF(E_i)} \\
            & {GF(\bigoplus_{i\in I} E_i).}
            \arrow["{\oplus_{i \in I} \eta_{E_i}}", from=1-1, to=1-2]
            \arrow["{\eta_{\oplus_{i \in I} E_i}}"', from=1-1, to=2-2]
            \arrow[dashed, from=1-2, to=2-2]
        \end{tikzcd}
    \end{displaymath}
    The dashed morphism is an isomorphism because $G\circ F$ is additive.
    Applying the octahedral axiom yields 
    \begin{displaymath}
        \operatorname{cone}(\oplus_{i \in I} \eta_{E_i}) 
        \cong \operatorname{cone}(\eta_{\oplus_{i \in I} E_i}).
    \end{displaymath}
    However, the distinguished triangle 
    \begin{displaymath}
        \bigoplus_{i\in I} E_i \to \bigoplus_{i\in I} GF(E_i) \to \operatorname{cone}(\oplus_{i \in I} \eta_{E_i}) \to (\bigoplus_{i\in I} E_i)[1]
    \end{displaymath}
    is the direct sum of the distinguished triangles
    \begin{displaymath}
        E_i \xrightarrow{\eta_{E_i}} GF(E_i) \to \operatorname{cone}(\eta_{E_i}) \to E_i[1].
    \end{displaymath}
    Hence, it follows that 
    \begin{displaymath}
        \bigoplus_{i\in I} \operatorname{cone}(\eta_{E_i}) \cong \operatorname{cone}(\bigoplus_{i\in I} \eta_{E_i}).
    \end{displaymath}
    Since $\oplus_{i\in I} E_i \in \mathcal{T}_1$, $\operatorname{cone}(\eta_{\oplus_{i \in I} E_i})\in \mathcal{A}$.
    As $\mathcal{A}$ is thick, every direct summand of $\operatorname{cone}(\eta_{\oplus_{i \in I} E_i})$ is in $\mathcal{A}$. 
    Consequently, $\operatorname{cone}(\eta_{E_i}) \in \mathcal{A}$ for each $i\in I$.
    This completes the proof.
\end{proof}

\begin{setup}
    \label{setup:relatively_perfect_adjoint_pair}
    Let $S$ be a quasi-compact quasi-separated quasi-excellent algebraic space.
    Fix proper flat morphisms $f_i \colon Y_i \to S$ of Noetherian algebraic spaces for $i\in \{1,2\}$. 
    Denote by $p_i \colon Y_1\times_S Y_2 \to Y_i$ the canonical projections. 
    Choose $E\in D^b_{\operatorname{coh}}(Y_1 \times_S Y_2)$ such that $E$ is relatively perfect over each $Y_i$.
    Set $E_R := \operatorname{\mathbf{R}\mathcal{H}\!\mathit{om}}(E, p_2^! \mathcal{O}_{Y_2})$.
    By \cite[Proposition 4.19 \& Corollary 4.20]{GuisadoVillaalgordo/Lank:2026}, $E_R\in D^b_{\operatorname{coh}}(Y_1\times_S Y_2)$, $\Phi_E$ is left adjoint to $\Phi_{E_R}$, and the adjunction restricts to $D^b_{\operatorname{coh}}$. 
    Assume that $E_R$ is relatively perfect over each $Y_i$. 
    Applying \cite[Corollary 1.9]{GuisadoVillaalgordo/Lank:2026}, both $\Phi_E$ and $\Phi_{E_R}$ induce functors on $D_{\operatorname{sg}}$, which are respectively denoted as $\Phi_E^{\operatorname{sg}}$ and $\Phi_{E_R}^{\operatorname{sg}}$. 
    By \Cref{lem:restrict_adjunction_to_Verdier_quotient}, $\Phi_E^{\operatorname{sg}}$ and $\Phi_{E_R}^{\operatorname{sg}}$ form an adjoint pair.
    Write $\epsilon$ and $\eta$ respectively for the counit and unit of the adjoint pair $\Phi_E \dashv \Phi_{E_R}$ on $D_{\operatorname{qc}}$.
    By \Cref{cor:classical_generator_for_J2}, each $D^b_{\operatorname{coh}}(Y_i)$ admits a classical generator $G_i$. 
    Indeed, $f_i$ is proper, and hence, finitely presented.
    Since $S$ is quasi-excellent, \Cref{rmk:quasi-excellet_is_j-2} says $S$ is $J\textrm{-}2$, and so each $Y_i$ is $J\textrm{-}2$. 
    Moreover, for any $K\in D_{\operatorname{qc}}(Y_1\times_S Y_2)$ and $t\colon T \to S$, we let $K_t := \mathbf{L}(t^\prime)^\ast K$ where 
    \begin{displaymath}
        t^\prime \colon Y_1 \times_S Y_2 \times_S T \to Y_1 \times_S Y_2
    \end{displaymath}
    is the canonical morphism obtained by base change.
\end{setup}

\begin{lemma}
    \label{lem:fully_faithfull_left_adjoint_dsg}
    Assume \Cref{setup:relatively_perfect_adjoint_pair}.
    Then $\Phi_E^{\operatorname{sg}}$ is fully faithful on $D_{\operatorname{sg}}$ if, and only if, $\operatorname{cone}(\eta_A) \in \operatorname{Perf}(Y_1)$ for all $A\in D^b_{\operatorname{coh}}(Y_1)$. 
\end{lemma}

\begin{proof}
    By \Cref{lem:restrict_adjunction_to_Verdier_quotient}, $\pi^{\operatorname{sg}}_{Y_1}(\eta_A)$ is the component for the unit of $\Phi_E^{\operatorname{sg}}\dashv \Phi_{E_R}^{\operatorname{sg}}$ at any $\pi^{\operatorname{sg}}_{Y_1} (A) \in D_{\operatorname{sg}}(Y_1)$.
    Moreover, by \cite[\href{https://stacks.math.columbia.edu/tag/07RB}{Tag 07RB}]{StacksProject}, $\Phi_E^{\operatorname{sg}}$ is fully faithful if, and only if, $\pi^{\operatorname{sg}}_{Y_1}(\eta_A)$ is an isomorphism for all $\pi^{\operatorname{sg}}_{Y_1} (A) \in D_{\operatorname{sg}}(Y_1)$.
    Furthermore, $\pi^{\operatorname{sg}}_{Y_1}(\eta_A)$ is an isomorphism for all $\pi^{\operatorname{sg}}_{Y_1} (A) \in D_{\operatorname{sg}}(Y_1)$ is equivalent to $\operatorname{cone}(\eta_A)\in \operatorname{Perf}(Y_1)$ for all $A\in D^b_{\operatorname{coh}}(Y_1)$.
    Indeed, use the fact that $\operatorname{Perf}(Y_1)$ is thick, and that $\operatorname{Perf}(Y_1) = \ker (\pi^{\operatorname{sg}}_{Y_1})$ \cite[\href{https://stacks.math.columbia.edu/tag/05RI}{Tag 05RI}]{StacksProject}.
    Therefore, the desired claim follows, which completes the proof. 
\end{proof}

\begin{lemma}
    \label{lem:fully_faithfull_right_adjoint_dsg}
    Assume \Cref{setup:relatively_perfect_adjoint_pair}.
    Then $\Phi_{E_R}^{\operatorname{sg}}$ is fully faithful on $D_{\operatorname{sg}}$ if, and only if, $\operatorname{cone}(\epsilon_B) \in \operatorname{Perf}(Y_2)$ for all $B\in D^b_{\operatorname{coh}}(Y_2)$. 
\end{lemma}

\begin{proof}
    This is dual to \Cref{lem:fully_faithfull_left_adjoint_dsg}.
\end{proof}

\begin{remark}
    \label{rmk:serre_subcategory}
    Let $X$ be a locally Noetherian algebraic space. 
    In this case, $\operatorname{coh}(X)$ is a Serre subcategory of $\operatorname{Qcoh}(X)$.
    Indeed, apply \cite[\href{https://stacks.math.columbia.edu/tag/02MP}{Tags 02MP}, \href{https://stacks.math.columbia.edu/tag/07UC}{07UC}, \& \href{https://stacks.math.columbia.edu/tag/07UD}{07UD}]{StacksProject}.
    Denote by $D_{\operatorname{coh}}(X)$ the strictly full subcategory of $D(X)$ consisting of objects with coherent cohomology. 
    There exists a canonical functor 
    \begin{displaymath}
        D(\operatorname{coh}(X)) \to D_{\operatorname{coh}}(X).
    \end{displaymath}
    See \cite[\href{https://stacks.math.columbia.edu/tag/06UP}{Tag 06UP}]{StacksProject}.
\end{remark}

\begin{lemma}
    \label{lem:enhanced_bounded above_equivalence}
    Let $X$ be a Noetherian algebraic space. 
    Then the functor 
    \begin{displaymath}
        D^-(\operatorname{coh}(X)) \to D^-_{\operatorname{coh}(X)}(\operatorname{Qcoh}(X))
    \end{displaymath}
    is an equivalence.
\end{lemma}

\begin{proof}
    By \Cref{rmk:serre_subcategory}, $\operatorname{coh}(X)$ is a Serre subcategory of $\operatorname{Qcoh}(X)$.
    Moreover, if $G\to F$ is a surjection from a quasi-coherent $\mathcal{O}_X$-module to a coherent $\mathcal{O}_X$-module, then there exists a coherent submodule $G^\prime \subseteq G$ which surjects onto $F$.
    Indeed, by \cite[\href{https://stacks.math.columbia.edu/tag/07V9}{Tag 07V9}]{StacksProject}, $G$ is the filtered colimit of its coherent submodules.  Choosing one of these provides the required $G^\prime$. 
    Thus, the claim follows from \cite[\href{https://stacks.math.columbia.edu/tag/0FCL}{Tag 0FCL}]{StacksProject}.
\end{proof}

\begin{lemma}
    \label{lem:enhanced_to_D_coh_equivalences}
    Let $X$ be a Noetherian algebraic space. 
    Then the functors 
    \begin{displaymath}
        D^-(\operatorname{coh}(X)) \to D^-_{\operatorname{coh}}(X)
    \end{displaymath}
    and
    \begin{displaymath}
        D^b(\operatorname{coh}(X)) \to D^b_{\operatorname{coh}}(X)
    \end{displaymath}
    are equivalences.
\end{lemma}

\begin{proof}
    Consider the commutative diagram 
    \begin{displaymath}
        \begin{tikzcd}
            {D^-(\operatorname{coh}(X))} & {D^-_{\operatorname{coh}}(X)} \\
            {D^-(\operatorname{Qcoh}(X))} & {D^-_{\operatorname{qc}}(X).}
            \arrow[from=1-1, to=1-2]
            \arrow[from=1-1, to=2-1]
            \arrow[from=1-2, to=2-2]
            \arrow[from=2-1, to=2-2]
        \end{tikzcd}
    \end{displaymath}
    By \Cref{lem:enhanced_bounded above_equivalence}, the left vertical functor is fully faithful.
    By \cite[\href{https://stacks.math.columbia.edu/tag/09TN}{Tag 09TN}]{StacksProject}, the bottom functor is an equivalence. 
    By construction, the right vertical arrow is fully faithful. 
    Hence, the top functor is fully faithful. 
    If $E\in D^-_{\operatorname{coh}}(X)$, the object $E^\prime \in D^-(\operatorname{Qcoh}(X))$ which corresponds to it by \cite[\href{https://stacks.math.columbia.edu/tag/09TN}{Tag 09TN}]{StacksProject} will have coherent cohomology.
    Hence, by \Cref{lem:enhanced_bounded above_equivalence}, $E^\prime$ is in the essential image of the left vertical arrow.
    Consequently, the top horizontal arrow is essentially surjective, which finishes the proof for the bounded above case. 
    The bounded case follows immediately from the bounded above case.
\end{proof}

\begin{lemma}
    \label{lem:counit_isomorphism_for_localization}
    Let $f\colon Y \to X$ be a finitely presented morphism to an affine Noetherian scheme from an algebraic space.
    Denote by $l_p \colon \operatorname{Spec}(\mathcal{O}_{X,p}) \to X$ the canonical morphism induced by some $p\in |X|$. 
    Form the fibered square
    \begin{displaymath}
        \begin{tikzcd}
            {Y\times_X \operatorname{Spec}(\mathcal{O}_{X,p}) } & {\operatorname{Spec}(\mathcal{O}_{X,p}) } \\
            Y & {X.}
            \arrow["{f^\prime}", from=1-1, to=1-2]
            \arrow["{g_p}"', from=1-1, to=2-1]
            \arrow["{l_p}", from=1-2, to=2-2]
            \arrow["f"', from=2-1, to=2-2]
        \end{tikzcd}
    \end{displaymath}
    Then the counit $\mathbf{L}g_p^\ast \mathbf{R}(g_p)_\ast \to 1$ is an isomorphism on $D_{\operatorname{qc}}(Y\times_X\operatorname{Spec}(\mathcal{O}_{X,p}))$.
    Moreover,  
    \begin{displaymath}
        \mathbf{L}g_p^\ast \colon D^b_{\operatorname{coh}}(Y) \to D^b_{\operatorname{coh}}(Y\times_X\operatorname{Spec}(\mathcal{O}_{X,p}))
    \end{displaymath}
    is essentially surjective.
\end{lemma}

\begin{proof}
    This is motivated by the proof of \cite[Theorem 4.4]{Elagin/Lunts/Schnurer:2020}. 
    We first prove that the counit $\mathbf{L}g_p^\ast \mathbf{R}(g_p)_\ast \to 1$ is an isomorphism on $D_{\operatorname{qc}}(Y\times_X\operatorname{Spec}(\mathcal{O}_{X,p}))$.
    Let $s\colon U \to Y$ be an \'{e}tale presentation by an affine scheme. 
    Form the commutative diagram 
    \begin{displaymath}
        \begin{tikzcd}
            {U\times_X \operatorname{Spec}(\mathcal{O}_{X,p}) } & {Y\times_X \operatorname{Spec}(\mathcal{O}_{X,p}) } & {\operatorname{Spec}(\mathcal{O}_{X,p}) } \\
            U & Y & X.
            \arrow["{s^\prime}", from=1-1, to=1-2]
            \arrow["{h_p}"', from=1-1, to=2-1]
            \arrow["{f^\prime}", from=1-2, to=1-3]
            \arrow["{g_p}"', from=1-2, to=2-2]
            \arrow["{l_p}", from=1-3, to=2-3]
            \arrow["s"', from=2-1, to=2-2]
            \arrow["f"', from=2-2, to=2-3]
        \end{tikzcd}
    \end{displaymath}
    By \cite[\href{https://stacks.math.columbia.edu/tag/01SG}{Tag 01SG}]{StacksProject}, $l_p$ is an affine morphism. 
    Hence, by base change, so are $h_p$ and $g_p$.
    Moreover, $l_p$ is flat.
    Thus, $h_p$ and $g_p$ are flat and affine.
    Consequently, $h_p^\ast$ is exact on modules, while $(h_p)_\ast$ is exact on quasi-coherent modules.
    Thus $\mathbf{L}h_p^\ast=h_p^\ast$, and $\mathbf{R}(h_p)_\ast$ agrees with $(h_p)_\ast$ on complexes of quasi-coherent modules (see e.g.\ \cite[\href{https://stacks.math.columbia.edu/tag/073H}{Tags 073H}, \href{https://stacks.math.columbia.edu/tag/08GY}{08GY}, \& \href{https://stacks.math.columbia.edu/tag/09TN}{09TN}]{StacksProject}).
    The analogous statements hold for $g_p$.

    We first show that the counit
    \begin{equation}
        \label{eq:counit-local}
        h_p^\ast (h_p)_\ast E \to E
    \end{equation} 
    is an isomorphism for all $E\in D_{\operatorname{qc}} (U\times_X \operatorname{Spec}(\mathcal{O}_{X,p}))$.
    This can be checked on cohomology.
    Hence, by exactness of $h_p^\ast$ and $(h_p)_\ast$, we may assume that $E$ is a quasi-coherent sheaf.
    Indeed, apply \cite[\href{https://stacks.math.columbia.edu/tag/08GY}{Tags 08GY} \& \href{https://stacks.math.columbia.edu/tag/09TN}{09TN}]{StacksProject}.
    Since $X$ and $U$ are affine, $h_p$ is induced by a localization of rings.
    Therefore, \eqref{eq:counit-local} is an isomorphism.

    We now consider the fibered square
    \begin{displaymath}
        \begin{tikzcd}
            {U\times_X \operatorname{Spec}(\mathcal{O}_{X,p})}
                & {Y\times_X \operatorname{Spec}(\mathcal{O}_{X,p})} \\
            U & Y.
            \arrow["{s^\prime}", from=1-1, to=1-2]
            \arrow["{h_p}"', from=1-1, to=2-1]
            \arrow["{g_p}", from=1-2, to=2-2]
            \arrow["s"', from=2-1, to=2-2]
        \end{tikzcd}
    \end{displaymath}
    Denote by $\epsilon^{g_p}\colon \mathbf{L}g_p^\ast \mathbf{R}(g_p)_\ast \to 1$ and $\epsilon^{h_p}\colon \mathbf{L}h_p^\ast \mathbf{R}(h_p)_\ast \to 1$ the respective counits.
    Pseudofunctoriality of derived pullback gives a natural isomorphism
    \begin{displaymath}
        \alpha\colon
        \mathbf{L}(s^\prime)^\ast\mathbf{L}g_p^\ast
        \xrightarrow{\cong}
        \mathbf{L}h_p^\ast\mathbf{L}s^\ast.
    \end{displaymath}
    Denote its inverse by
    \begin{displaymath}
        \widetilde{\alpha}\colon
        \mathbf{L}h_p^\ast\mathbf{L}s^\ast
        \xrightarrow{\cong}
        \mathbf{L}(s^\prime)^\ast\mathbf{L}g_p^\ast.
    \end{displaymath}
    By the construction of the pullback--pushforward base change morphism
    (see \cite[Appendix D]{GuisadoVillaalgordo/Lank:2026} or \cite[\href{https://stacks.math.columbia.edu/tag/07A7}{Tag 07A7}]{StacksProject}),
    \begin{displaymath}
        \beta\colon
        \mathbf{L}s^\ast\mathbf{R}(g_p)_\ast
        \longrightarrow
        \mathbf{R}(h_p)_\ast\mathbf{L}(s^\prime)^\ast
    \end{displaymath}
    is the right mate of $\widetilde{\alpha}$ with respect to the adjunctions
    \begin{displaymath}
        \mathbf{L}g_p^\ast\dashv\mathbf{R}(g_p)_\ast
        \qquad\text{and}\qquad
        \mathbf{L}h_p^\ast\dashv\mathbf{R}(h_p)_\ast.
    \end{displaymath}
    By tor-independent base change
    \cite[\href{https://stacks.math.columbia.edu/tag/08IR}{Tag 08IR}]{StacksProject},
    $\beta$ is an isomorphism.

    Applying \cite[Lemma D.4(2)]{GuisadoVillaalgordo/Lank:2026}
    to the mate pair $(\widetilde{\alpha},\beta)$, and using
    $\widetilde{\alpha}=\alpha^{-1}$, gives a commutative diagram
    \begin{displaymath}
        \begin{tikzcd}
            {\mathbf{L}(s^\prime)^\ast\mathbf{L}g_p^\ast\mathbf{R}(g_p)_\ast E} & {\mathbf{L}h_p^\ast\mathbf{R}(h_p)_\ast\mathbf{L}(s^\prime)^\ast E} \\
            {\mathbf{L}(s^\prime)^\ast E} & {\mathbf{L}(s^\prime)^\ast E.}
            \arrow["\cong", from=1-1, to=1-2]
            \arrow["{\mathbf{L}(s^\prime)^\ast(\epsilon^{g_p}_E)}"', from=1-1, to=2-1]
            \arrow["{\epsilon^{h_p}_{\mathbf{L}(s^\prime)^\ast E}}", from=1-2, to=2-2]
            \arrow[equals, from=2-1, to=2-2]
        \end{tikzcd}
    \end{displaymath}
    Here the upper horizontal morphism is the composite
    \begin{displaymath}
        \mathbf{L}(s^\prime)^\ast\mathbf{L}g_p^\ast\mathbf{R}(g_p)_\ast
        \xrightarrow{\alpha\,\mathbf{R}(g_p)_\ast}
        \mathbf{L}h_p^\ast\mathbf{L}s^\ast\mathbf{R}(g_p)_\ast
        \xrightarrow{\mathbf{L}h_p^\ast(\beta)}
        \mathbf{L}h_p^\ast\mathbf{R}(h_p)_\ast
        \mathbf{L}(s^\prime)^\ast.
    \end{displaymath}
    By \eqref{eq:counit-local}, the right vertical morphism is an isomorphism.
    Hence, $\mathbf{L}(s^\prime)^\ast(\epsilon^{g_p}_E)$ is an isomorphism.
    Base change implies $s^\prime$ is \'{e}tale and surjective.
    Then \cite[Lemma 2.14]{GuisadoVillaalgordo/Lank:2026} says $\mathbf{L}(s^\prime)^\ast$ is conservative.
    Therefore, $\epsilon^{g_p}_E$ is an isomorphism, which proves the first claim.

    Next, we prove the second claim.
    Since $g_p$ is flat, $g_p^\ast$ is exact, and so the canonical morphism $\mathbf{L}g_p^\ast \to g_p^\ast$ is an isomorphism \cite[\href{https://stacks.math.columbia.edu/tag/015F}{Tags 015F} \& \href{https://stacks.math.columbia.edu/tag/0735}{0735}]{StacksProject}.
    We first establish the following observation:  
    Suppose $M$ is a bounded complex of quasi-coherent $\mathcal{O}_Y$-modules such that $g_p^\ast M$ is a bounded complex of coherent $\mathcal{O}_{Y\times_X \operatorname{Spec}(\mathcal{O}_{X,p})}$-modules. 
    Then there exists a subcomplex $K \subseteq M$ consisting of coherent $\mathcal{O}_Y$-modules such that $g_p^\ast K = g_p^\ast M$. 
    Indeed, every quasi-coherent $\mathcal{O}_Y$-module is the filtered colimit of its coherent $\mathcal{O}_Y$-submodules \cite[\href{https://stacks.math.columbia.edu/tag/07V9}{Tag 07V9}]{StacksProject}. 
    Consequently, every complex $M$ of quasi-coherent $\mathcal{O}_Y$-modules is a filtered colimit of subcomplexes consisting of coherent $\mathcal{O}_Y$-modules. 
    Indeed, each $M^n$ is a filtered colimit of coherent $\mathcal{O}_Y$-submodules. 
    If $N^n \subseteq M^n$ is a coherent $\mathcal{O}_Y$-submodule for each $n \in \mathbb Z$, then $N^n + d(N^{n-1}) \subseteq M^n$ defines a subcomplex of $M$ consisting of coherent $\mathcal{O}_Y$-modules.
    Write $M = \varinjlim_i M_i$ as a filtered colimit of such subcomplexes. Since $g_p^\ast$ commutes with filtered colimits, $g_p^\ast M \cong \varinjlim_i g_p^\ast M_i$.
    Hence, for each $n \in \mathbb Z$, $(g_p^\ast M)^n \cong \varinjlim_i g_p^\ast(M_i^n)$. 
    As $(g_p^\ast M)^n$ is coherent, there exists some $i$ such that $(g_p^\ast M)^n = g_p^\ast(M_i^n)$.
    Indeed, $g^\ast_p$ is exact, so the $g_p^\ast(M_i^n)$ form a filtered family of quasi-coherent subsheaves of $(g_p^\ast M)^n$ whose union is $(g_p^\ast M)^n$.
    As $(g_p^\ast M)^n$ is of finite type and $Y\times_X \operatorname{Spec}(\mathcal{O}_{X,p})$ is quasi-compact, one of them is all of $(g_p^\ast M)^n$. 
    Since $M$ is bounded, we may choose a single $i$ such that this equality holds for every $n \in \mathbb Z$. Setting $K := M_i$, we obtain $\mathbf{L}g_p^\ast K = g_p^\ast M$ as desired.
    
    We now complete the proof of the second claim.
    Let $E\in D^b_{\operatorname{coh}}(Y\times_X \operatorname{Spec}(\mathcal{O}_{X,p}))$.
    By \Cref{lem:enhanced_to_D_coh_equivalences}, there exists 
    \begin{displaymath}
        E_0 \in D^b(\operatorname{coh}(Y\times_X \operatorname{Spec}(\mathcal{O}_{X,p})))
    \end{displaymath}
    such that $E_0 \cong E$. 
    As discussed earlier, $\mathbf{R}(g_p)_\ast E_0$ can be represented by applying $(g_p)_\ast$ component-wise. 
    Hence, $(g_p)_\ast E_0$ is a bounded complex of quasi-coherent $\mathcal{O}_Y$-modules. 
    The counit $g_p^\ast (g_p)_\ast E_0 \to E_0$ is an isomorphism. 
    Then, by the observation above, there exists a bounded complex of coherent $\mathcal{O}_Y$-modules $A$ such that $\mathbf{L}g_p^\ast A \cong E_0$. 
    Since $A\in D^b_{\operatorname{coh}}(Y)$, $\mathbf{L}g_p^\ast$ is essentially surjective.
\end{proof}

\begin{lemma}
    \label{lem:perfect_over_local_ring_iff_perfect_over_closed_fiber}
    Let $(R,\mathfrak{m},k)$ be a Noetherian local ring. 
    Suppose $f\colon Y \to \operatorname{Spec}(R)$ is a proper morphism of algebraic spaces. 
    Denote by $i\colon \operatorname{Spec}(k)\to \operatorname{Spec}(R)$ the canonical closed immersion. 
    Form the fibered square
    \begin{displaymath}
        \begin{tikzcd}
            {Y\times_{\operatorname{Spec}(R)} \operatorname{Spec}(k)} & {\operatorname{Spec}(k)} \\
            Y & {\operatorname{Spec}(R).}
            \arrow["{f^\prime}", from=1-1, to=1-2]
            \arrow["{i^\prime}"', from=1-1, to=2-1]
            \arrow["i", from=1-2, to=2-2]
            \arrow["f"', from=2-1, to=2-2]
        \end{tikzcd}
    \end{displaymath}
    Then $E\in D^b_{\operatorname{coh}}(Y)$ is perfect if, and only if, $\mathbf{L}(i^\prime)^\ast E\in \operatorname{Perf}(Y\times_{\operatorname{Spec}(R)} \operatorname{Spec}(k))$. 
\end{lemma}

\begin{proof}
    If $E$ is perfect, then $\mathbf{L}(i^\prime)^\ast E$ is perfect \cite[\href{https://stacks.math.columbia.edu/tag/08H6}{Tag 08H6}]{StacksProject}, and so 
    \begin{displaymath}
        \mathbf{L}(i^\prime)^\ast E
        \in\operatorname{Perf}(Y\times_{\operatorname{Spec}(R)} \operatorname{Spec}(k)).
    \end{displaymath}
    Conversely, suppose that $\mathbf{L}(i^\prime)^\ast E$ is perfect.
    As $f$ is proper, $|f|$ is a closed morphism.
    Choose any closed point $q\in |Y|$. 
    Then $|f|(q) \in |\operatorname{Spec}(R)|$ is the unique closed point. 
    Hence, by \cite[\href{https://stacks.math.columbia.edu/tag/03H4}{Tag 03H4}]{StacksProject}, every closed point of $|Y|$ is contained in 
    \begin{displaymath}
        |i^\prime| (|Y\times_{\operatorname{Spec}(R)} \operatorname{Spec}(k)|).
    \end{displaymath}
    Indeed, $q \in |Y|$ is an arbitrary closed point, and it lies in the inverse image of the unique closed point of $\operatorname{Spec}(R)$ along $f$. 
    Applying \cite[\href{https://stacks.math.columbia.edu/tag/04CD}{Tag 04CD}]{StacksProject}, it follows that $q \in |Y|$ can be represented by a morphism of the form 
    \begin{displaymath}
        t_q \colon \operatorname{Spec}(k^\prime)\to Y\times_{\operatorname{Spec}(R)} \operatorname{Spec}(k) \xrightarrow{i^\prime} Y.
    \end{displaymath} 
    To see this, use that $|i^\prime|$ is a homeomorphism onto a closed subset of $|Y|$. 
    Since
    \begin{displaymath}
        \mathbf{L}(i^\prime)^\ast E\in \operatorname{Perf}(Y\times_{\operatorname{Spec}(R)} \operatorname{Spec}(k))
    \end{displaymath}
    it follows that $\mathbf{L}t_q^\ast E\in D^b_{\operatorname{coh}}(k^\prime)$. 
    By \cite[Proposition 3.12]{GuisadoVillaalgordo/Lank:2026}, $E$ is perfect. 
\end{proof}

\begin{lemma}
    \label{lem:fully_faithful_iff_cone_of_classical_generator_is_perfect}
    Assume \Cref{setup:relatively_perfect_adjoint_pair}.
    Then $\Phi^{\operatorname{sg}}_E$ (resp.\ $\Phi^{\operatorname{sg}}_{E_R}$) is fully faithful if, and only if, $\operatorname{cone}(\eta_{G_1})\in \operatorname{Perf}(Y_1)$ (resp.\ $\operatorname{cone}(\epsilon_{G_2})\in \operatorname{Perf}(Y_2)$).
\end{lemma}

\begin{proof}
    We prove the case for full faithfulness of $\Phi^{\operatorname{sg}}_E$. 
    The other can be argued analogously. 
    By \Cref{lem:fully_faithfull_left_adjoint_dsg}, if $\Phi^{\operatorname{sg}}_E$ is fully faithful, then $\operatorname{cone}(\eta_{G_1})\in \operatorname{Perf}(Y_1)$.
    Conversely, let $\mathcal{T}_1$ be the strictly full subcategory of objects $A\in D^b_{\operatorname{coh}}(Y_1)$ such that $\operatorname{cone}(\eta_A)\in \operatorname{Perf}(Y_1)$.
    By \Cref{lem:restrict_adjunction_to_Verdier_quotient}, $\mathcal{T}_1$ is thick. 
    The hypothesis implies $G_1 \in \mathcal{T}_1$, and so $\mathcal{T}_1 = D^b_{\operatorname{coh}}(Y_1)$.
    Applying \Cref{lem:fully_faithfull_left_adjoint_dsg}, the desired claim follows.
\end{proof}

\section{Germs}
\label{sec:germs}

\begin{definition}
    \label{def:germ}
    Let $X$ be an algebraic space.
    Pick $p\in |X|$. 
    An \textbf{\'{e}tale germ} of $(X,p)$ is an \'{e}tale morphism $s\colon U \to X$ from a scheme and a choice of point $u$ in $|s|^{-1}(p)$.
    At times, it is denoted by $s\colon (U,u) \to (X,p)$. 
\end{definition}

\begin{definition}
    \label{def:singular_FM_locus}
    Assume \Cref{setup:relatively_perfect_adjoint_pair}.
    Write $\operatorname{fm}_{\operatorname{sg}}(E)$ (resp.\ $\operatorname{FM}_{\operatorname{sg}}(E)$) for the collection of $p\in |S|$ such that there exists an \'{e}tale germ $s\colon (U,u)\to (S,p)$ such that $\Phi_{\mathbf{L}(s^\prime)^\ast E}^{\operatorname{sg}}$ is fully faithful (resp.\ an equivalence); here
    \begin{displaymath}
        s^\prime \colon Y_1\times_S Y_2 \times_S \operatorname{Spec}(\mathcal{O}_{U,u}) \to Y_1\times_S Y_2
    \end{displaymath}
    is the canonical projection induced by the composite
    \begin{displaymath}
        \operatorname{Spec}(\mathcal{O}_{U,u}) \to U \xrightarrow{s} S
    \end{displaymath}
    of the canonical morphism and $s$.
    In particular, we say that $\Phi_E^{\operatorname{sg}}$ is \textbf{fully faithful} (resp.\ an \textbf{equivalence}) \textbf{at $p\in |S|$} if $p\in \operatorname{fm}_{\operatorname{sg}}(E)$ (resp.\ $p\in \operatorname{FM}_{\operatorname{sg}}(E)$).
\end{definition}

\begin{remark}
    Assume \Cref{setup:relatively_perfect_adjoint_pair}.
    Let $s\colon (U,u) \to (S,p)$ be an \'{e}tale germ.
    Since $U$ is quasi-excellent, so is $\mathcal{O}_{U,u}$ \cite[\href{https://stacks.math.columbia.edu/tag/07QU}{Tag 07QU}]{StacksProject}. 
    By \cite[Corollary 4.11]{GuisadoVillaalgordo/Lank:2026}, $E_{s\circ t}$ is relatively perfect over each $Y_i \times_S \operatorname{Spec}(\mathcal{O}_{U,u})$ where $t\colon \operatorname{Spec}(\mathcal{O}_{U,u}) \to U$ is the canonical morphism.
    A similar statement holds for $E_R$ by loc.\ cit.
    By \cite[Lemma 2.18]{GuisadoVillaalgordo/Lank:2026}, $s \circ t$ is quasi-affine.
    Moreover, \cite[Proposition 4.19]{GuisadoVillaalgordo/Lank:2026} says that $\Phi_{(E_R)_{s \circ t}} \cong \Phi_{(E_{s\circ t})_R}$.
\end{remark}

\begin{remark}
    \label{rmk:kollar_curve}
    A field-valued analog of \Cref{def:singular_FM_locus} is not useful.
    Indeed, equivalences on singularity categories need not base change.
    See e.g.\ \cite[Example 5.31]{GuisadoVillaalgordo/Lank:2026}.
    Loc.\ cit.\ gives a projective scheme $X$ over a field $k$ which is regular but its base change $X_L$ along a suitable field extension $L/k$ is singular.
    Hence, $D_{\operatorname{sg}}(X)$ is the trivial category, whereas $D_{\operatorname{sg}}(X_L)$ is nontrivial.
\end{remark}

\begin{lemma}
    \label{lem:quasi-perfect_morphism}
    A proper morphism $f\colon Y \to X$ of finite tor-dimension between Noetherian algebraic spaces is quasi-perfect.
\end{lemma}

\begin{proof}
    As $f$ has finite tor-dimension, $\mathcal{O}_Y$ is $f$-perfect.
    By \cite[Theorem 1.8]{GuisadoVillaalgordo/Lank:2026}, $\mathcal{O}_Y$ is $f$-quasi-perfect.
    Hence, $\mathbf{R}f_\ast \operatorname{Perf}(Y)\subseteq \operatorname{Perf}(X)$.
    By \cite[Lemma 3.1]{DeDeyn/Lank/ManaliRahul:2025b}, $f$ is quasi-perfect.
\end{proof}

\begin{lemma}
    \label{lem:singular_FM_locus_independence}
    Assume \Cref{setup:relatively_perfect_adjoint_pair}.
    For any $p\in |S|$, the following are equivalent, both for full faithfulness and for equivalence:
    \begin{enumerate}
        \item \label{lem:singular_FM_locus_independence1} For every \'{e}tale germ $s\colon (U,u)\to (S,p)$, the base-changed functor $\Phi_{E_{s\circ l_u}}^{\operatorname{sg}}$ is fully faithful (resp.\ an equivalence), where $l_u\colon \operatorname{Spec}(\mathcal{O}_{U,u})\to U$ is the canonical morphism.
        \item \label{lem:singular_FM_locus_independence2} There is some \'{e}tale germ $s\colon (U,u)\to (S,p)$ for which $\Phi_{E_{s\circ l_u}}^{\operatorname{sg}}$ is fully faithful (resp.\ an equivalence).
    \end{enumerate}
\end{lemma}

\begin{proof}
    We prove the case for full faithfulness.
    The case for equivalence is similar. 
    It is obvious that $\eqref{lem:singular_FM_locus_independence1} \implies \eqref{lem:singular_FM_locus_independence2}$. 
    We prove the converse. 

    Suppose $\Phi_{E_{s\circ l_u}}^{\operatorname{sg}}$ is fully faithful for some \'{e}tale germ $s\colon (U,u)\to (S,p)$.
    After choosing an affine open neighborhood of $u$ in $U$, we can impose that $U$ is affine. 
    Fix an \'{e}tale germ $t\colon (V,v)\to (S,p)$. 
    Again, after choosing an affine open neighborhood of $v$ in $V$, we can impose that $V$ is affine. 
    For each $i$, there exist diagrams with fibered faces:
    \begin{displaymath}
        \begin{tikzcd}
            {Y_{i,\kappa(u)}} & {Y_{i,u}} & {Y_{i,U}} & {Y_i} \\
            {\operatorname{Spec}(\kappa(u))} & {\operatorname{Spec}(\mathcal{O}_{U,u})} & U & S
            \arrow["{u_3^i}", from=1-1, to=1-2]
            \arrow["{f^{\kappa(u)}_i}", from=1-1, to=2-1]
            \arrow["{u_2^i}", from=1-2, to=1-3]
            \arrow["{f^u_i}", from=1-2, to=2-2]
            \arrow["{u_1^i}", from=1-3, to=1-4]
            \arrow["{f^U_i}", from=1-3, to=2-3]
            \arrow["{f_i}", from=1-4, to=2-4]
            \arrow["{r^U_u}"', from=2-1, to=2-2]
            \arrow["{l^U_u}"', from=2-2, to=2-3]
            \arrow["u"', from=2-3, to=2-4]
        \end{tikzcd}
    \end{displaymath}
    and 
    \begin{displaymath}
        \begin{tikzcd}
            {Y_{i,\kappa(v)}} & {Y_{i,v}} & {Y_{i,V}} & {Y_i} \\
            {\operatorname{Spec}(\kappa(v))} & {\operatorname{Spec}(\mathcal{O}_{V,v})} & V & {S.}
            \arrow["{v_3^i}", from=1-1, to=1-2]
            \arrow["{f^{\kappa(v)}_i}", from=1-1, to=2-1]
            \arrow["{v_2^i}", from=1-2, to=1-3]
            \arrow["{f^v_i}", from=1-2, to=2-2]
            \arrow["{v_1^i}", from=1-3, to=1-4]
            \arrow["{f^V_i}", from=1-3, to=2-3]
            \arrow["{f_i}", from=1-4, to=2-4]
            \arrow["{r^V_v}"', from=2-1, to=2-2]
            \arrow["{l^V_v}"', from=2-2, to=2-3]
            \arrow["v"', from=2-3, to=2-4]
        \end{tikzcd}
    \end{displaymath}
    Here $r^\#_{?}$ and $l^\#_{?}$ are the canonical morphisms where $\# \in \{U,V\}$ and $?\in \{u,v\}$. 
    Base change implies each $f^\#_i$, $f^{?}_i$, and $f^{\kappa(?)}_i$ are proper and flat where $i\in \{1,2\}$, $\# \in \{U,V\}$, and $?\in \{u,v\}$. 
    Since $s$ and $t$ are \'{e}tale germs of $p\in |S|$, there exist a field $k$ and a commutative diagram (see \cite[\href{https://stacks.math.columbia.edu/tag/03BT}{Tag 03BT}]{StacksProject} and \cite[Lemma 2.7]{GuisadoVillaalgordo/Lank:2026})
    \begin{displaymath}
        \begin{tikzcd}
            {\operatorname{Spec}(k)} & {\operatorname{Spec}(\kappa(u))} \\
            {\operatorname{Spec}(\kappa(v))} & {S.}
            \arrow["{a_u}", from=1-1, to=1-2]
            \arrow["{a_v}"', from=1-1, to=2-1]
            \arrow["{u\circ l^U_u \circ r^U_u}", from=1-2, to=2-2]
            \arrow["{v\circ l^V_v \circ r^V_v}"', from=2-1, to=2-2]
        \end{tikzcd}
    \end{displaymath}
    Combining these diagrams, we obtain a commutative diagram
    \begin{equation}
        \label{eq:cookie}
        \begin{tikzcd}
            & {Y_{i,k,u}} & {Y_{i,\kappa(u)}} \\
            {Y_{i,k,v}} & {\operatorname{Spec}(k)} & {\operatorname{Spec}(\kappa(u))} \\
            {Y_{i,\kappa(v)}} & {\operatorname{Spec}(\kappa(v))}
            \arrow["{a^i_u}", from=1-2, to=1-3]
            \arrow["{g_u^i}"', from=1-2, to=2-2]
            \arrow["{f^{\kappa(u)}_i}"', from=1-3, to=2-3]
            \arrow["{g_v^i}", from=2-1, to=2-2]
            \arrow["{a^i_v}"', from=2-1, to=3-1]
            \arrow["{a_u}", from=2-2, to=2-3]
            \arrow["{a_v}"', from=2-2, to=3-2]
            \arrow["{f^{\kappa(v)}_i}", from=3-1, to=3-2]
        \end{tikzcd}
    \end{equation}
    with fibered faces.
    Note that 
    \begin{displaymath}
        v \circ l^V_v \circ r^V_v \circ a_v = u \circ l^U_u \circ r^U_u \circ a_u.
    \end{displaymath} 
    It follows that the rectangles 
    \begin{displaymath}
        \begin{tikzcd}
            & {Y_{i,k,u}} & {Y_{i,\kappa(u)}} & {Y_{i,u}} & {Y_{i,U}} & {Y_i} \\
            {Y_{i,k,v}} & {\operatorname{Spec}(k)} & {\operatorname{Spec}(\kappa(u))} & {\operatorname{Spec}(\mathcal{O}_{U,u})} & U & S \\
            {Y_{i,\kappa(v)}} & {\operatorname{Spec}(\kappa(v))} \\
            {Y_{i,v}} & {\operatorname{Spec}(\mathcal{O}_{V,v})} \\
            {Y_{i,V}} & V \\
            {Y_i} & S
            \arrow["{a^i_u}", from=1-2, to=1-3]
            \arrow["{g^i_u}"', from=1-2, to=2-2]
            \arrow["{u^i_3}", from=1-3, to=1-4]
            \arrow["{f^{\kappa(u)}_i}", from=1-3, to=2-3]
            \arrow["{u^i_2}", from=1-4, to=1-5]
            \arrow["{f^u_i}", from=1-4, to=2-4]
            \arrow["{u^i_1}", from=1-5, to=1-6]
            \arrow["{f^U_i}", from=1-5, to=2-5]
            \arrow["{f_i}", from=1-6, to=2-6]
            \arrow["{g^i_v}", from=2-1, to=2-2]
            \arrow["{a^i_v}"', from=2-1, to=3-1]
            \arrow["{a_u}"', from=2-2, to=2-3]
            \arrow["{a_v}", from=2-2, to=3-2]
            \arrow["{r^U_u}"', from=2-3, to=2-4]
            \arrow["{l^U_u}"', from=2-4, to=2-5]
            \arrow["u"', from=2-5, to=2-6]
            \arrow["{f^{\kappa(v)}_i}"', from=3-1, to=3-2]
            \arrow["{v^i_3}"', from=3-1, to=4-1]
            \arrow["{r^V_v}", from=3-2, to=4-2]
            \arrow["{f^v_i}"', from=4-1, to=4-2]
            \arrow["{v^i_2}"', from=4-1, to=5-1]
            \arrow["{l^V_v}", from=4-2, to=5-2]
            \arrow["{f^V_i}"', from=5-1, to=5-2]
            \arrow["{v^i_1}"', from=5-1, to=6-1]
            \arrow["v", from=5-2, to=6-2]
            \arrow["{f_i}"', from=6-1, to=6-2]
        \end{tikzcd}
    \end{displaymath}
    both form the base change of $f_i$ along the composition $\operatorname{Spec}(k) \to S$.
    This gives an isomorphism $h \colon Y_{i,k,v} \to Y_{i,k,u}$ compatible with the diagrams above.
    Indeed, both rectangles above are pastings of fibered squares, which are fiber products of $f_i$ and the common map $\operatorname{Spec}(k) \to S$.
    That is, we have a commutative diagram
    \begin{displaymath}
        \begin{tikzcd}
            {Y_{i,k,v}} &&& {Y_i} && \\
            & {\operatorname{Spec}(k)} &&&& S \\
            {Y_{i,k,u}} &&& {Y_i} \\
            & {\operatorname{Spec}(k)} &&&& {S.}
            \arrow["{v^i_1 \circ v^i_2 \circ v^i_3 \circ a^i_v}", from=1-1, to=1-4]
            \arrow[equals, from=2-2, to=4-2]
            \arrow["{g^i_v}", from=1-1, to=2-2]
            \arrow["h"', from=1-1, to=3-1]
            \arrow["{f_i}", from=1-4, to=2-6]
            \arrow[equals, dashed, from=1-4, to=3-4]
            \arrow["{v \circ l^V_v \circ r^V_v \circ a_v}"{pos=0.3}, from=2-2, to=2-6]
            \arrow[equals, from=2-6, to=4-6]
            \arrow["{u^i_1 \circ u^i_2 \circ u^i_3 \circ a^i_u}"{pos=0.6}, dashed, from=3-1, to=3-4]
            \arrow["{g^i_u}", from=3-1, to=4-2]
            \arrow["{f_i}", from=3-4, to=4-6]
            \arrow["{u \circ l^U_u \circ r^U_u \circ a_u}"', from=4-2, to=4-6]
        \end{tikzcd}
    \end{displaymath}   
    In particular, it follows that $v^i_1 \circ v^i_2 \circ v^i_3 \circ a^i_v = u^i_1 \circ u^i_2 \circ u^i_3 \circ a^i_u \circ h$. 

    By \cite[Proposition 1.4]{Lank:2026c} and \Cref{lem:counit_isomorphism_for_localization}, 
    \begin{displaymath}
        \mathbf{L}(?_1^1 \circ ?_2^1)^\ast \colon D^b_{\operatorname{coh}}(Y_1) \to D^b_{\operatorname{coh}}(Y_{1,?})
    \end{displaymath}
    is essentially dense for $?\in \{u,v\}$. 
    Hence, $\mathbf{L}(?_1^1\circ ?_2^1)^\ast G_1$ is a classical generator for $D^b_{\operatorname{coh}}(Y_{1,?})$ for $?\in \{u,v\}$. 
    Denote by $\eta^{1,?}$ the unit for the adjoint pair $\Phi_{E_{?_1^1 \circ ?_2^1}} \dashv \Phi_{(E_R)_{?_1^1 \circ ?_2^1}}$ for $?\in \{u,v\}$.
    By \Cref{lem:fully_faithful_iff_cone_of_classical_generator_is_perfect}, it suffices to prove that 
    \begin{displaymath}
        \operatorname{cone}(\eta^{1,v}_{\mathbf{L}(v_1^1 \circ v_2^1)^\ast G_1})\in \operatorname{Perf}(Y_{1,v}).
    \end{displaymath}
    Applying \cite[Lemma 2.18]{GuisadoVillaalgordo/Lank:2026}, $w \circ l^W_w$ is quasi-affine for each $w\in \{u,v\}$ and $W\in \{U,V\}$.
    Base change implies that $w_1^i \circ w_2^i$ is quasi-affine for $w\in \{u,v\}$.
    By \cite[Remark 4.29]{GuisadoVillaalgordo/Lank:2026}, there exist diagrams for $w\in \{u,v\}$:
    \begin{displaymath}
        \begin{tikzcd}[column sep=large,row sep=large]
        \mathbf{L}(w_1^1\circ w_2^1)^\ast G_1
        \arrow[r,"\eta^{1,w}_{\mathbf{L}(w_1^1\circ w_2^1)^\ast G_1}"]
        \arrow[d,"\mathbf{L}(w_1^1\circ w_2^1)^\ast (\eta_{G_1})"']
        &
        \Phi_{(E_R)_{w_1^1\circ w_2^1}}
        \Phi_{E_{w_1^1\circ w_2^1}}
        (\mathbf{L}(w_1^1\circ w_2^1)^\ast G_1)
        \arrow[d,"\cong"]
        \\
        \mathbf{L}(w_1^1\circ w_2^1)^\ast \Phi_{E_R}\Phi_E(G_1)
        \arrow[r,"\cong"']
        &
        \Phi_{(E_R)_{w_1^1\circ w_2^1}}
        \mathbf{L}(w_1^2\circ w_2^2)^\ast \Phi_E(G_1).
        \end{tikzcd}
    \end{displaymath}
    Applying the octahedral axiom to the diagrams above, it follows that there exist isomorphisms for $w\in \{u,v\}$:
    \begin{equation}
        \label{eq:12}
        \mathbf{L}(w_1^1\circ w_2^1)^\ast \operatorname{cone}(\eta_{G_1})
        \cong \operatorname{cone}(\eta^{1,w}_{\mathbf{L}(w_1^1\circ w_2^1 )^\ast G_1}).
    \end{equation}

    Now we finish the proof. 
    The hypothesis says that $\operatorname{cone}(\eta^{1,u}_{\mathbf{L}(u_1^1\circ u_2^1)^\ast G_1}) \in \operatorname{Perf}(Y_{1,u})$. 
    By \eqref{eq:12},
    \begin{displaymath}
        \mathbf{L}(u_1^1\circ u_2^1)^\ast \operatorname{cone}(\eta_{G_1})
        \cong \operatorname{cone}(\eta^{1,u}_{\mathbf{L}(u_1^1\circ u_2^1 )^\ast G_1}).
    \end{displaymath}
    By \cite[\href{https://stacks.math.columbia.edu/tag/08H6}{Tag 08H6}]{StacksProject},
    \begin{displaymath}
        \begin{aligned}
            \mathbf{L}(u_1^1\circ u_2^1 \circ u_3^1 \circ a_u^1)^\ast 
            & \operatorname{cone}(\eta_{G_1})
            \\&\cong \mathbf{L}(u_3^1 \circ a_u^1)^\ast \mathbf{L} (u_1^1\circ u_2^1)^\ast \operatorname{cone}(\eta_{G_1}) && (\textrm{pseudofunctoriality})
            \\&\in \operatorname{Perf}(Y_{1,k,u}).
        \end{aligned}
    \end{displaymath}
    However,
    \begin{displaymath}
        v^i_1 \circ v^i_2 \circ v^i_3 \circ a^i_v = u^i_1 \circ u^i_2 \circ u^i_3 \circ a^i_u \circ h,
    \end{displaymath}
    and so it follows that 
    \begin{displaymath}
        \mathbf{L}(u_1^1\circ u_2^1 \circ u_3^1 \circ a_u^1 \circ h)^\ast \operatorname{cone}(\eta_{G_1})
        \cong \mathbf{L}(v_1^1\circ v_2^1 \circ v_3^1 \circ a_v^1)^\ast \operatorname{cone}(\eta_{G_1})
    \end{displaymath}
    is perfect. 
    Since $a^i_v$ is affine and faithfully flat, \cite[Lemma 3.7]{GuisadoVillaalgordo/Lank:2026} implies 
    \begin{displaymath}
        \mathbf{L}(v_1^1\circ v_2^1 \circ v_3^1)^\ast \operatorname{cone}(\eta_{G_1})
    \end{displaymath}
    is perfect. 
    Consequently, \Cref{lem:perfect_over_local_ring_iff_perfect_over_closed_fiber} yields that $\mathbf{L}(v_1^1\circ v_2^1)^\ast \operatorname{cone}(\eta_{G_1})$ is perfect.
    Applying \eqref{eq:12}, it follows that $\operatorname{cone}(\eta^{1,v}_{\mathbf{L}(v_1^1\circ v_2^1 )^\ast G_1})$ is perfect. 
    This completes the proof.
\end{proof}

\section{Openness \& base change}
\label{sec:opennes_base_change}

\begin{theorem}
    \label{thm:openness_singular_fm}
    Assume \Cref{setup:relatively_perfect_adjoint_pair}.
    Then $\operatorname{fm}_{\operatorname{sg}}(E)$ and $\operatorname{FM}_{\operatorname{sg}}(E)$ are open subsets of $|S|$. 
\end{theorem}

\begin{proof} 
    Denote by $U_1$ the collection of $y\in |Y_1|$ such that $\mathbf{L}t^\ast \operatorname{cone}(\eta_{G_1})\in D^b_{\operatorname{coh}}(k)$ where $t$ is a representative of $y$. 
    Moreover, let $U_2$ be the collection of $y\in |Y_2|$ such that $\mathbf{L}t^\ast \operatorname{cone}(\epsilon_{G_2})\in D^b_{\operatorname{coh}}(k)$ where $t$ is a representative of $y$. 
    By \cite[Lemma 3.46]{GuisadoVillaalgordo/Lank:2026}, $U_1$ and $U_2$ are open. 

    We first claim that 
    \begin{displaymath}
        |S| \setminus \operatorname{fm}_{\operatorname{sg}}(E) = |f_1|(|Y_1| \setminus U_1).
    \end{displaymath}
    Since $|f_1|$ is closed, the desired statement for full faithfulness would follow. 
    Choose $p\in |S|\setminus \operatorname{fm}_{\operatorname{sg}}(E)$.
    Pick an \'{e}tale germ $s\colon (U,u) \to (S,p)$ from an affine scheme.
    Denote by $s_u \colon \operatorname{Spec}(\mathcal{O}_{U,u}) \to U$ the canonical morphism.
    By the choice of $p$, $\Phi_{E_{s\circ s_u}}^{\operatorname{sg}}$ is not fully faithful.
    Form the fibered rectangle 
    \begin{displaymath}
        \begin{tikzcd}
            {Y_1 \times_S \operatorname{Spec}(\mathcal{O}_{U,u})} & {\operatorname{Spec}(\mathcal{O}_{U,u})} \\
            {Y_1\times_S U} & U \\
            {Y_1} & {S.}
            \arrow["{f^u_1}", from=1-1, to=1-2]
            \arrow["{s^\prime_u}"', from=1-1, to=2-1]
            \arrow["{s_u}", from=1-2, to=2-2]
            \arrow["{f_1^\prime}"', from=2-1, to=2-2]
            \arrow["{s^\prime}"', from=2-1, to=3-1]
            \arrow["s", from=2-2, to=3-2]
            \arrow["{f_1}"', from=3-1, to=3-2]
        \end{tikzcd}
    \end{displaymath}
    By \cite[Proposition 1.4]{Lank:2026c} and \Cref{lem:counit_isomorphism_for_localization}, 
    \begin{displaymath}
        \mathbf{L}(s^\prime \circ s^\prime_u)^\ast \colon D^b_{\operatorname{coh}}(Y_1) \to D^b_{\operatorname{coh}}(Y_1\times_S \operatorname{Spec}(\mathcal{O}_{U,u}))
    \end{displaymath}
    is essentially dense. 
    Hence, $\mathbf{L}(s^\prime \circ s^\prime_u)^\ast G_1$ is a classical generator for $D^b_{\operatorname{coh}}(Y_1\times_S \operatorname{Spec}(\mathcal{O}_{U,u}))$. 
    Denote by $\eta^\prime$ the unit for the adjoint pair $\Phi_{E_{s \circ s_u}} \dashv \Phi_{(E_R)_{s \circ s_u}}$. 
    By \cite[Lemma 2.18]{GuisadoVillaalgordo/Lank:2026}, $s\circ s_u$ is quasi-affine, and so base change yields that $s^\prime \circ s^\prime_u$ is quasi-affine.
    By \cite[Remark 4.29]{GuisadoVillaalgordo/Lank:2026}, there exist diagrams 
    \begin{displaymath}
        \begin{tikzcd}[column sep=large,row sep=large]
        \mathbf{L}(s^\prime \circ s^\prime_u)^\ast G_1
        \arrow[r,"\eta^\prime_{\mathbf{L}(s^\prime \circ s^\prime_u)^\ast G_1}"]
        \arrow[d,"\mathbf{L}(s^\prime \circ s^\prime_u)^\ast (\eta_{G_1})"']
        &
        \Phi_{(E_R)_{s \circ s_u}}
        \Phi_{E_{s \circ s_u}}
        (\mathbf{L}(s^\prime \circ s^\prime_u)^\ast G_1)
        \arrow[d,"\cong"]
        \\
        \mathbf{L}(s^\prime \circ s^\prime_u)^\ast \Phi_{E_R}\Phi_E(G_1)
        \arrow[r,"\cong"']
        &
        \Phi_{(E_R)_{s \circ s_u}}
        \mathbf{L}(s^\prime \circ s^\prime_u)^\ast \Phi_E(G_1).
        \end{tikzcd}
    \end{displaymath}
    Applying the octahedral axiom yields an isomorphism,
    \begin{displaymath}
        \mathbf{L}(s^\prime \circ s^\prime_u)^\ast \operatorname{cone}(\eta_{G_1})
        \cong \operatorname{cone}(\eta^\prime_{\mathbf{L}(s^\prime \circ s^\prime_u)^\ast G_1}).
    \end{displaymath}
    By \Cref{lem:fully_faithful_iff_cone_of_classical_generator_is_perfect}, the choice of $p$ implies
    \begin{displaymath}
        \operatorname{cone}(\eta^\prime_{\mathbf{L}(s^\prime \circ s^\prime_u)^\ast G_1})\not\in \operatorname{Perf}(Y_1\times_S \operatorname{Spec}(\mathcal{O}_{U,u})).
    \end{displaymath}
    Hence, $\mathbf{L}(s^\prime \circ s^\prime_u)^\ast \operatorname{cone}(\eta_{G_1})$ is not perfect. 
    Then, by \cite[Proposition 3.12]{GuisadoVillaalgordo/Lank:2026}, there exist a closed point $q\in |Y_1\times_S \operatorname{Spec}(\mathcal{O}_{U,u})|$ and a representative $t\colon \operatorname{Spec}(k)\to Y_1 \times_S \operatorname{Spec}(\mathcal{O}_{U,u})$ such that 
    \begin{displaymath}
        \mathbf{L}(s^\prime \circ s^\prime_u \circ t)^\ast \operatorname{cone}(\eta_{G_1})
    \end{displaymath}
    is not in $D^b_{\operatorname{coh}}(k)$. 
    Consequently, by \cite[Lemmas 2.7 \&  3.8]{GuisadoVillaalgordo/Lank:2026}, $|s^\prime \circ s^\prime_u|(q)\not\in U_1$. 
    As $q$ is a closed point and $f^u_1$ is proper, it follows that $|f^u_1|(q)$ is the unique closed point. 
    Thus, $s\circ s_u \circ f^u_1 \circ t$ represents $p$, and so 
    \begin{displaymath}
       |S| \setminus \operatorname{fm}_{\operatorname{sg}}(E) \subseteq |f_1|(|Y_1| \setminus U_1).
    \end{displaymath}
    We show the reverse inclusion.

    Pick $p\in |f_1|(|Y_1| \setminus U_1)$. 
    Choose $q \in |Y_1| \setminus U_1$ such that $|f_1|(q)=p$. 
    Consider an \'{e}tale presentation $s\colon U \to S$ from an affine scheme.
    We can find $u\in |U|$ such that $|s|(u)=p$.
    Denote by $s_u \colon \operatorname{Spec}(\mathcal{O}_{U,u}) \to U$ the canonical morphism. 
    Form the fibered rectangle 
    \begin{displaymath}
        \begin{tikzcd}
            {Y_1 \times_S \operatorname{Spec}(\mathcal{O}_{U,u})} & {\operatorname{Spec}(\mathcal{O}_{U,u})} \\
            {Y_1\times_S U} & U \\
            {Y_1} & {S.}
            \arrow["{f^u_1}", from=1-1, to=1-2]
            \arrow["{s^\prime_u}"', from=1-1, to=2-1]
            \arrow["{s_u}", from=1-2, to=2-2]
            \arrow["{f_1^\prime}"', from=2-1, to=2-2]
            \arrow["{s^\prime}"', from=2-1, to=3-1]
            \arrow["s", from=2-2, to=3-2]
            \arrow["{f_1}"', from=3-1, to=3-2]
        \end{tikzcd}
    \end{displaymath}
    By \cite[Proposition 1.4]{Lank:2026c} and \Cref{lem:counit_isomorphism_for_localization}, 
    \begin{displaymath}
        \mathbf{L}(s^\prime \circ s^\prime_u)^\ast \colon D^b_{\operatorname{coh}}(Y_1) \to D^b_{\operatorname{coh}}(Y_1\times_S \operatorname{Spec}(\mathcal{O}_{U,u}))
    \end{displaymath}
    is essentially dense. 
    Hence, $\mathbf{L}(s^\prime \circ s^\prime_u)^\ast G_1$ is a classical generator for $D^b_{\operatorname{coh}}(Y_1\times_S \operatorname{Spec}(\mathcal{O}_{U,u}))$. 
    Denote by $\eta^\prime$ the unit for the adjoint pair $\Phi_{E_{s \circ s_u}} \dashv \Phi_{(E_R)_{s \circ s_u}}$. 
    By \cite[Lemma 2.18]{GuisadoVillaalgordo/Lank:2026}, $s\circ s_u$ is quasi-affine, and base change implies $s^\prime \circ s^\prime_u$ is quasi-affine.
    Moreover, \cite[Remark 4.29]{GuisadoVillaalgordo/Lank:2026} yields diagrams 
    \begin{displaymath}
        \begin{tikzcd}[column sep=large,row sep=large]
        \mathbf{L}(s^\prime \circ s^\prime_u)^\ast G_1
        \arrow[r,"\eta^\prime_{\mathbf{L}(s^\prime \circ s^\prime_u)^\ast G_1}"]
        \arrow[d,"\mathbf{L}(s^\prime \circ s^\prime_u)^\ast (\eta_{G_1})"']
        &
        \Phi_{(E_R)_{s \circ s_u}}
        \Phi_{E_{s \circ s_u}}
        (\mathbf{L}(s^\prime \circ s^\prime_u)^\ast G_1)
        \arrow[d,"\cong"]
        \\
        \mathbf{L}(s^\prime \circ s^\prime_u)^\ast \Phi_{E_R}\Phi_E(G_1)
        \arrow[r,"\cong"']
        &
        \Phi_{(E_R)_{s \circ s_u}}
        \mathbf{L}(s^\prime \circ s^\prime_u)^\ast \Phi_E(G_1).
        \end{tikzcd}
    \end{displaymath}
    Applying the octahedral axiom gives an isomorphism,
    \begin{displaymath}
        \mathbf{L}(s^\prime \circ s^\prime_u)^\ast \operatorname{cone}(\eta_{G_1})
        \cong \operatorname{cone}(\eta^\prime_{\mathbf{L}(s^\prime \circ s^\prime_u)^\ast G_1}).
    \end{displaymath}
    Denote by $i\colon \operatorname{Spec}(\kappa(u))\to \operatorname{Spec}(\mathcal{O}_{U,u})$ the canonical closed immersion. 
    By construction, $s\circ s_u \circ i$ represents $p$. 
    After picking a representative $t\colon \operatorname{Spec}(k) \to Y_1$ of $q$, $f_1 \circ t$ represents $p$ \cite[Lemma 2.7]{GuisadoVillaalgordo/Lank:2026}.
    Hence, there exist a field $\ell$ and a commutative diagram
    \begin{displaymath}
        \begin{tikzcd}
            {\operatorname{Spec}(\ell)} & {\operatorname{Spec}(\kappa(u))} \\
            {\operatorname{Spec}(k)} & S
            \arrow["b", from=1-1, to=1-2]
            \arrow["a"', from=1-1, to=2-1]
            \arrow["{s\circ s_u \circ i}", from=1-2, to=2-2]
            \arrow["{f_1 \circ t}"', from=2-1, to=2-2]
        \end{tikzcd}
    \end{displaymath}
    Consequently, there exist a unique (dashed) morphism $h$ and a commutative diagram
    \begin{displaymath}
        \begin{tikzcd}
            {\operatorname{Spec}(\ell)} && \\
            & {Y_1 \times_S \operatorname{Spec}(\kappa(u))} & {\operatorname{Spec}(\kappa(u))} \\
            & {Y_1 \times_S \operatorname{Spec}(\mathcal{O}_{U,u})} & {\operatorname{Spec}(\mathcal{O}_{U,u})} \\
            & {Y_1\times_S U} & U \\
            & {Y_1} & {S.}
            \arrow["h"{description}, dashed, from=1-1, to=2-2]
            \arrow["b"', bend right = -10pt, from=1-1, to=2-3]
            \arrow["{t\circ a}"', bend right = 30pt, from=1-1, to=5-2]
            \arrow["g"', from=2-2, to=2-3]
            \arrow["{i_u}", from=2-2, to=3-2]
            \arrow["i"', from=2-3, to=3-3]
            \arrow["{f^u_1}", from=3-2, to=3-3]
            \arrow["{s^\prime_u}"', from=3-2, to=4-2]
            \arrow["{s_u}", from=3-3, to=4-3]
            \arrow["{f_1^\prime}"', from=4-2, to=4-3]
            \arrow["{s^\prime}"', from=4-2, to=5-2]
            \arrow["s", from=4-3, to=5-3]
            \arrow["{f_1}"', from=5-2, to=5-3]
        \end{tikzcd}
    \end{displaymath}
    By the choice of $q$, it follows that 
    \begin{displaymath}
        \begin{aligned}
            \mathbf{L}(t\circ a)^\ast \operatorname{cone}(\eta_{G_1})
            &\cong \mathbf{L}(s^\prime \circ s^\prime_u\circ i_u \circ h)^\ast \operatorname{cone}(\eta_{G_1})
            \\&\cong \mathbf{L}(i_u \circ h)^\ast \operatorname{cone}(\eta^\prime_{\mathbf{L}(s^\prime \circ s^\prime_u)^\ast G_1})
        \end{aligned}
    \end{displaymath}
    is not in $D^b_{\operatorname{coh}}(\ell)$. 
    Hence, $\operatorname{cone}(\eta^\prime_{\mathbf{L}(s^\prime \circ s^\prime_u)^\ast G_1})$ cannot be perfect. 
    By \Cref{lem:fully_faithfull_left_adjoint_dsg}, $\Phi_{E_{s\circ s_u}}^{\operatorname{sg}}$ is not fully faithful.
    By \Cref{lem:singular_FM_locus_independence}, $p\not\in \operatorname{fm}_{\operatorname{sg}}(E)$. 

    Next we prove that $\operatorname{FM}_{\operatorname{sg}}(E)$ is open. 
    Arguing as above with $G_2$, $\epsilon$, and \Cref{lem:fully_faithfull_right_adjoint_dsg} in place of $G_1$, $\eta$, and \Cref{lem:fully_faithfull_left_adjoint_dsg}, it follows that $|f_2|(|Y_2|\setminus U_2)$ is the collection of $p\in |S|$ such that $\Phi_{(E_R)_{s\circ s_u}}^{\operatorname{sg}}$ is not fully faithful for some (equivalently, every) \'{e}tale germ $s\colon (U,u)\to (S,p)$ from an affine scheme where $s_u \colon \operatorname{Spec}(\mathcal{O}_{U,u}) \to U$ is the canonical morphism.
    Recall that $\Phi_{(E_R)_{s\circ s_u}}^{\operatorname{sg}}$ is right adjoint to $\Phi_{E_{s\circ s_u}}^{\operatorname{sg}}$.
    Note that a functor admitting a right adjoint is an equivalence if, and only if, both the functor and its right adjoint are fully faithful.
    Combining this with the first part of the proof, it follows that
    \begin{displaymath}
       |S| \setminus \operatorname{FM}_{\operatorname{sg}}(E) 
       = |f_1|(|Y_1| \setminus U_1) \bigcup |f_2|(|Y_2| \setminus U_2).
    \end{displaymath}
    This completes the proof.
\end{proof}

\begin{proposition}
    \label{prop:germwise_criteria}
    Assume \Cref{setup:relatively_perfect_adjoint_pair}.
    The following are equivalent:
    \begin{enumerate}
        \item \label{prop:germwise_criteria1} $\Phi_E^{\operatorname{sg}}$ is fully faithful (resp.\ an equivalence)
        \item \label{prop:germwise_criteria2} $\Phi_E^{\operatorname{sg}}$ is fully faithful (resp.\ an equivalence) at all $p\in |S|$
        \item \label{prop:germwise_criteria3} $\Phi_E^{\operatorname{sg}}$ is fully faithful (resp.\ an equivalence) at all closed points $p\in |S|$.
    \end{enumerate}
\end{proposition}

\begin{proof}
    We prove the case for $\Phi_E^{\operatorname{sg}}$ being fully faithful. 
    The case of equivalences follows by applying the same argument to $\Phi_{E_R}^{\operatorname{sg}}$.
    Indeed, $\Phi_E^{\operatorname{sg}}$ is an equivalence if, and only if, both $\Phi_E^{\operatorname{sg}}$ and its right adjoint $\Phi_{E_R}^{\operatorname{sg}}$ are fully faithful.
    It is obvious that $\eqref{prop:germwise_criteria2} \implies \eqref{prop:germwise_criteria3}$.
    We are left to prove that $\eqref{prop:germwise_criteria1} \implies \eqref{prop:germwise_criteria2}$ and $\eqref{prop:germwise_criteria3} \implies \eqref{prop:germwise_criteria1}$.

    First, we check that $\eqref{prop:germwise_criteria1} \implies \eqref{prop:germwise_criteria2}$.
    Choose $p\in |S|$.
    Pick an \'{e}tale germ $s\colon U \to S$ of $p$ from an affine scheme. 
    Fix $u\in |U|$ such that $|s|(u)=p$.
    Denote by $s_u \colon \operatorname{Spec}(\mathcal{O}_{U,u}) \to U$ the canonical morphism.
    Form the fibered rectangle 
    \begin{displaymath}
        \begin{tikzcd}
            {Y_1 \times_S \operatorname{Spec}(\mathcal{O}_{U,u})} & {\operatorname{Spec}(\mathcal{O}_{U,u})} \\
            {Y_1\times_S U} & U \\
            {Y_1} & {S.}
            \arrow["{f^u_1}", from=1-1, to=1-2]
            \arrow["{s^\prime_u}"', from=1-1, to=2-1]
            \arrow["{s_u}", from=1-2, to=2-2]
            \arrow["{f_1^\prime}"', from=2-1, to=2-2]
            \arrow["{s^\prime}"', from=2-1, to=3-1]
            \arrow["s", from=2-2, to=3-2]
            \arrow["{f_1}"', from=3-1, to=3-2]
        \end{tikzcd}
    \end{displaymath}
    By \cite[Proposition 1.4]{Lank:2026c} and \Cref{lem:counit_isomorphism_for_localization}, 
    \begin{displaymath}
        \mathbf{L}(s^\prime \circ s^\prime_u)^\ast \colon D^b_{\operatorname{coh}}(Y_1) \to D^b_{\operatorname{coh}}(Y_1\times_S \operatorname{Spec}(\mathcal{O}_{U,u}))
    \end{displaymath}
    is essentially dense. 
    Hence, $\mathbf{L}(s^\prime \circ s^\prime_u)^\ast G_1$ is a classical generator for $D^b_{\operatorname{coh}}(Y_1\times_S \operatorname{Spec}(\mathcal{O}_{U,u}))$. 
    Denote by $\eta^\prime$ the unit for the adjoint pair $\Phi_{E_{s \circ s_u}} \dashv \Phi_{(E_R)_{s \circ s_u}}$. 
    By \cite[Lemma 2.18]{GuisadoVillaalgordo/Lank:2026}, $s \circ s_u$ is quasi-affine, and so after base change, $s^\prime \circ s^\prime_u$ is quasi-affine.
    By \cite[Remark 4.29]{GuisadoVillaalgordo/Lank:2026}, there exist diagrams 
    \begin{displaymath}
        \begin{tikzcd}[column sep=large,row sep=large]
        \mathbf{L}(s^\prime \circ s^\prime_u)^\ast G_1
        \arrow[r,"\eta^\prime_{\mathbf{L}(s^\prime \circ s^\prime_u)^\ast G_1}"]
        \arrow[d,"\mathbf{L}(s^\prime \circ s^\prime_u)^\ast (\eta_{G_1})"']
        &
        \Phi_{(E_R)_{s \circ s_u}}
        \Phi_{E_{s \circ s_u}}
        (\mathbf{L}(s^\prime \circ s^\prime_u)^\ast G_1)
        \arrow[d,"\cong"]
        \\
        \mathbf{L}(s^\prime \circ s^\prime_u)^\ast \Phi_{E_R}\Phi_E(G_1)
        \arrow[r,"\cong"']
        &
        \Phi_{(E_R)_{s \circ s_u}}
        \mathbf{L}(s^\prime \circ s^\prime_u)^\ast \Phi_E(G_1).
        \end{tikzcd}
    \end{displaymath}
    Applying the octahedral axiom yields an isomorphism,
    \begin{displaymath}
        \mathbf{L}(s^\prime \circ s^\prime_u)^\ast \operatorname{cone}(\eta_{G_1})
        \cong \operatorname{cone}(\eta^\prime_{\mathbf{L}(s^\prime \circ s^\prime_u)^\ast G_1}).
    \end{displaymath}
    By \Cref{lem:fully_faithful_iff_cone_of_classical_generator_is_perfect}, 
    $\operatorname{cone}(\eta_{G_1})$ is perfect. 
    This implies
    \begin{displaymath}
        \operatorname{cone}(\eta^\prime_{\mathbf{L}(s^\prime \circ s^\prime_u)^\ast G_1})\in \operatorname{Perf}(Y_1\times_S \operatorname{Spec}(\mathcal{O}_{U,u})).
    \end{displaymath}
    Applying \Cref{lem:fully_faithful_iff_cone_of_classical_generator_is_perfect} over $\operatorname{Spec}(\mathcal{O}_{U,u})$, it follows that $\Phi_E^{\operatorname{sg}}$ is fully faithful at $p$.

    Lastly, we prove that $\eqref{prop:germwise_criteria3} \implies \eqref{prop:germwise_criteria1}$.
    By \Cref{thm:openness_singular_fm}, $\operatorname{fm}_{\operatorname{sg}}(E)$ is open, and hence closed under generalization. 
    If $\eqref{prop:germwise_criteria3}$ holds, then $\operatorname{fm}_{\operatorname{sg}}(E) = |S|$.
    Indeed, every point of $|S|$ is a generalization of a closed point, see \cite[Remark 3.11]{GuisadoVillaalgordo/Lank:2026}.
    Denote by $U_1$ the collection of $y\in |Y_1|$ such that $\mathbf{L}t^\ast \operatorname{cone}(\eta_{G_1})\in D^b_{\operatorname{coh}}(k)$ where $t$ is a representative of $y$. 
    The proof of \Cref{thm:openness_singular_fm} shows that 
    \begin{displaymath}
        |S| \setminus \operatorname{fm}_{\operatorname{sg}}(E) = |f_1|(|Y_1| \setminus U_1).
    \end{displaymath}
    Hence, $\operatorname{fm}_{\operatorname{sg}}(E) = |S|$ implies $|Y_1| \setminus U_1 = \emptyset$, and so $|Y_1| = U_1$.
    By \cite[Proposition 3.12]{GuisadoVillaalgordo/Lank:2026}, $\operatorname{cone}(\eta_{G_1})\in \operatorname{Perf}(Y_1)$.
    Thus, \Cref{lem:fully_faithful_iff_cone_of_classical_generator_is_perfect} asserts that $\Phi^{\operatorname{sg}}_E$ is fully faithful.
\end{proof}

\begin{corollary}
    \label{cor:criteria}
    Assume \Cref{setup:relatively_perfect_adjoint_pair} with $S=\operatorname{Spec}(k)$ for a field $k$.
    The following are equivalent:
    \begin{enumerate}
        \item $\Phi_E^{\operatorname{sg}}$ is fully faithful (resp.\ an equivalence)
        \item $\Phi_{E_L}^{\operatorname{sg}}$ is fully faithful (resp.\ an equivalence) for all finite separable field extensions $L/k$
        \item $\Phi_{E_L}^{\operatorname{sg}}$ is fully faithful (resp.\ an equivalence) for some finite separable field extension $L/k$.
    \end{enumerate}
\end{corollary}

\begin{proof}
    Denote by $p$ the unique point of $|S|$, and let $(U,u) \to (S,p)$ be an \'{e}tale germ.
    By \cite[\href{https://stacks.math.columbia.edu/tag/02GL}{Tag 02GL}]{StacksProject}, $U$ is a disjoint union of spectra of finite separable field extensions of $k$, and so $\mathcal{O}_{U,u}$ is a finite separable field extension of $k$.
    Conversely, for any finite separable field extension $L/k$, the morphism $\operatorname{Spec}(L)\to S$ together with the unique point of $\operatorname{Spec}(L)$ is an \'{e}tale germ of $(S,p)$.
    Therefore, the claim follows from \Cref{prop:germwise_criteria} and \Cref{lem:singular_FM_locus_independence}.
\end{proof}

\begin{proposition}
    \label{prop:base_change}
    Assume \Cref{setup:relatively_perfect_adjoint_pair}.
    For any \'{e}tale morphism $t\colon S^\prime \to S$ of Noetherian algebraic spaces, $|t|^{-1}(\operatorname{fm}_{\operatorname{sg}}(E)) = \operatorname{fm}_{\operatorname{sg}}(E_t)$ and $|t|^{-1}(\operatorname{FM}_{\operatorname{sg}}(E)) = \operatorname{FM}_{\operatorname{sg}}(E_t)$.
\end{proposition}

\begin{proof}
    Choose an \'{e}tale presentation $s\colon U \to S^\prime$ by an affine scheme.
    We first check that \Cref{setup:relatively_perfect_adjoint_pair} holds for $E_t$ with respect to the base changes $Y_i\times_S S^\prime \to S^\prime$.
    This ensures that $\operatorname{fm}_{\operatorname{sg}}(E_t)$ and $\operatorname{FM}_{\operatorname{sg}}(E_t)$ are defined.
    As $t\circ s$ is \'{e}tale, \Cref{rmk:quasi-excellet_is_j-2} implies that $U$ is quasi-excellent, and so $S^\prime$ is quasi-excellent.
    Base change implies each $Y_i\times_S S^\prime \to S^\prime$ is proper and flat.
    By \cite[Corollary 4.11]{GuisadoVillaalgordo/Lank:2026}, $E_t$ is relatively perfect over each $Y_i\times_S S^\prime$.
    Denote by $(E_t)_R$ the object associated to $E_t$ over $S^\prime$ as in \Cref{setup:relatively_perfect_adjoint_pair}, and similarly for $(E_{t\circ s})_R$ over $U$.
    By \cite[Proposition 4.19]{GuisadoVillaalgordo/Lank:2026}, $(E_t)_R\in D^b_{\operatorname{coh}}(Y_1\times_S Y_2\times_S S^\prime)$.
    By \cite[Lemma 2.18]{GuisadoVillaalgordo/Lank:2026}, $s$ and $t\circ s$ are quasi-affine.
    Hence, applying \cite[Lemma 4.18]{GuisadoVillaalgordo/Lank:2026} along $s$ and along $t\circ s$ yields isomorphisms
    \begin{displaymath}
        ((E_t)_R)_s \cong (E_{t\circ s})_R \cong (E_R)_{t\circ s}.
    \end{displaymath}
    Here $((E_t)_R)_s$ denotes the derived pullback of $(E_t)_R$ to $Y_1\times_S Y_2\times_S U$.
    By \cite[Corollary 4.11]{GuisadoVillaalgordo/Lank:2026}, $(E_R)_{t\circ s}$ is relatively perfect over each $Y_i\times_S U$.
    As $s$ is faithfully flat, \cite[Theorem 4.9]{GuisadoVillaalgordo/Lank:2026} implies that $(E_t)_R$ is relatively perfect over each $Y_i\times_S S^\prime$.

    For any $u\in |U|$, let $l_u \colon \operatorname{Spec}(\mathcal{O}_{U,u}) \to U$ be the canonical morphism. 
    Moreover, for any $u\in |U|$, there exist commutative diagrams
    \begin{displaymath}
        \begin{tikzcd}
            {Y_1\times_S Y_2\times_S \operatorname{Spec}(\mathcal{O}_{U,u})} & {Y_1\times_S Y_2\times_S U} & {Y_1\times_S Y_2\times_S S^\prime} & {Y_1\times_S Y_2} \\
            {\operatorname{Spec}(\mathcal{O}_{U,u})} & U & {S^\prime} & S
            \arrow["{l^\prime_u}", from=1-1, to=1-2]
            \arrow[from=1-1, to=2-1]
            \arrow["{s^\prime}", from=1-2, to=1-3]
            \arrow[from=1-2, to=2-2]
            \arrow["{t^\prime}", from=1-3, to=1-4]
            \arrow[from=1-3, to=2-3]
            \arrow[from=1-4, to=2-4]
            \arrow["{l_u}", from=2-1, to=2-2]
            \arrow["s", from=2-2, to=2-3]
            \arrow["t", from=2-3, to=2-4]
        \end{tikzcd}
    \end{displaymath}
    where the vertical morphisms are compositions of the canonical ones obtained from the base changes of the $f_i$. 
    We use these diagrams implicitly below.
    We prove the claim for $\operatorname{fm}_{\operatorname{sg}}$.
    The claim for $\operatorname{FM}_{\operatorname{sg}}$ follows from the same argument after replacing `fully faithful' by `an equivalence'.

    Pick some $p\in |t|^{-1}(\operatorname{fm}_{\operatorname{sg}}(E))$. 
    Let $u\in|s|^{-1}(p)$. 
    Then $t\circ s \colon (U,u)\to (S,|t|(p))$ is an \'{e}tale germ. 
    Since $|t|(p)\in \operatorname{fm}_{\operatorname{sg}}(E)$, \Cref{lem:singular_FM_locus_independence} implies that $\Phi^{\operatorname{sg}}_{\mathbf{L}(t^\prime \circ s^\prime \circ l^\prime_u)^\ast E}$ is fully faithful.
    By pseudofunctoriality, $\Phi^{\operatorname{sg}}_{\mathbf{L}(s^\prime \circ l^\prime_u)^\ast \mathbf{L}(t^\prime)^\ast E}$ is fully faithful.
    As $s\colon (U,u) \to (S^\prime,p)$ is an \'{e}tale germ, $\Phi^{\operatorname{sg}}_{E_t}$ is fully faithful at $p$, which yields that $|t|^{-1}(\operatorname{fm}_{\operatorname{sg}}(E)) \subseteq \operatorname{fm}_{\operatorname{sg}}(E_t)$.

    Conversely, let $p\in \operatorname{fm}_{\operatorname{sg}}(E_t)$. 
    Let $u\in |s|^{-1}(p)$.
    By \Cref{lem:singular_FM_locus_independence} applied to $E_t$, it follows that $\Phi^{\operatorname{sg}}_{\mathbf{L}(s^\prime \circ l^\prime_u)^\ast \mathbf{L}(t^\prime)^\ast E}$ is fully faithful.
    Again, by pseudofunctoriality, $\Phi^{\operatorname{sg}}_{\mathbf{L}(t^\prime \circ s^\prime \circ l^\prime_u)^\ast E}$ is fully faithful.
    Since $t\circ s\colon (U,u) \to (S,|t|(p))$ is an \'{e}tale germ, $\Phi^{\operatorname{sg}}_{E}$ is fully faithful at $|t|(p)$. 
    This shows that $\operatorname{fm}_{\operatorname{sg}}(E_t) \subseteq |t|^{-1}(\operatorname{fm}_{\operatorname{sg}}(E))$, which completes the proof.
\end{proof}

\section{Examples}
\label{sec:examples}

\begin{example}
    \label{ex:kernels_for_smooth}
    Let $S$ be a Noetherian algebraic space, $f_i \colon Y_i \to S$ be smooth proper morphisms of algebraic spaces, and let $E\in D^b_{\operatorname{coh}}(Y_1 \times_S Y_2)$ be relatively perfect over each $Y_i$.
    Then $E$ is perfect.
    Indeed, form the fibered square
    \begin{displaymath}
        \begin{tikzcd}
            {Y_1 \times_S Y_2} & {Y_2} \\
            {Y_1} & {S.}
            \arrow["{p_2}", from=1-1, to=1-2]
            \arrow["{p_1}"', from=1-1, to=2-1]
            \arrow["{f_2}", from=1-2, to=2-2]
            \arrow["{f_1}"', from=2-1, to=2-2]
        \end{tikzcd}
    \end{displaymath}
    Base change implies each $p_i$ is smooth and proper. 
    Choose $p\in |S|$.
    Set $t_p \colon \operatorname{Spec}(\kappa(p)) \to S$ to be the canonical morphism from the associated residue field. 
    Consider the fibered cube 
    \begin{displaymath}
        \begin{tikzcd}[column sep=small,row sep=small]
            {Y_1\times_S Y_2 \times_S\operatorname{Spec}(\kappa(p))} && {Y_2 \times_S \operatorname{Spec}(\kappa(p))} & \\
            & {Y_1\times_S \operatorname{Spec}(\kappa(p))} && {\operatorname{Spec}(\kappa(p))} \\
            {Y_1} && {Y_2} \\
            & {Y_1} && {S.}
            \arrow["{q_2}", from=1-1, to=1-3]
            \arrow["{q_1}", from=1-1, to=2-2]
            \arrow["{t^\prime_p}"', from=1-1, to=3-1]
            \arrow["{f^p_2}", from=1-3, to=2-4]
            \arrow["{t_2^p}"{pos=0.3}, dashed, from=1-3, to=3-3]
            \arrow["{f^p_1}"{pos=0.2}, from=2-2, to=2-4]
            \arrow["{t_1^p}"'{pos=0.7}, from=2-2, to=4-2]
            \arrow["{t_p}", from=2-4, to=4-4]
            \arrow["{p_2}"{pos=0.7}, dashed, from=3-1, to=3-3]
            \arrow["{p_1}"', from=3-1, to=4-2]
            \arrow["{f_2}", dashed, from=3-3, to=4-4]
            \arrow["{f_1}"', from=4-2, to=4-4]
        \end{tikzcd}
    \end{displaymath}
    By \cite[Corollary 4.11]{GuisadoVillaalgordo/Lank:2026}, $\mathbf{L}(t^\prime_p)^\ast E$ is relatively perfect over each $Y_i \times_S \operatorname{Spec}(\kappa(p))$.
    Moreover, base change implies each $f^p_i$ and each $q_i$ are smooth, and so $Y_1 \times_S Y_2 \times_S \operatorname{Spec}(\kappa(p))$ is smooth over $\kappa(p)$.
    Hence, $Y_1 \times_S Y_2 \times_S \operatorname{Spec}(\kappa(p))$ is regular. 
    By \cite[Lemma 5.28]{GuisadoVillaalgordo/Lank:2026}, 
    \begin{displaymath}
        D^b_{\operatorname{coh}}(Y_1 \times_S Y_2 \times_S \operatorname{Spec}(\kappa(p))) = \operatorname{Perf}(Y_1 \times_S Y_2 \times_S \operatorname{Spec}(\kappa(p))).
    \end{displaymath}
    Thus, $\mathbf{L}(t^\prime_p)^\ast E$ is perfect for all $p\in |S|$, and so \cite[Lemma 5.30]{GuisadoVillaalgordo/Lank:2026} implies $E$ is perfect. 
\end{example}

\begin{example}
    \label{ex:perfect_kernels}
    Let $S$ be a quasi-compact quasi-separated quasi-excellent algebraic space. 
    Consider proper Gorenstein morphisms $f_i \colon Y_i \to S$ of Noetherian algebraic spaces. 
    Denote by $p_i \colon Y_1 \times_S Y_2 \to Y_i$ the canonical projections. 
    For any $Q\in \operatorname{Perf}(Y_1 \times_S Y_2)$, $Q$ is relatively perfect over each $Y_i$.
    Indeed, base change implies each $p_i$ is proper and Gorenstein (hence, flat), and so \Cref{lem:quasi-perfect_morphism} implies each $p_i$ is quasi-perfect.
    Hence, $\mathbf{R}(p_i)_\ast \operatorname{Perf}(Y_1\times_S Y_2) \subseteq \operatorname{Perf}(Y_i)$ for each $i$. 
    As $Q \otimes^{\mathbf{L}} -$ is an endofunctor on $\operatorname{Perf}(Y_1\times_S Y_2)$, it follows that $Q$ is relatively perfect over each $Y_i$. 
    Indeed, apply \cite[Corollary 1.9]{GuisadoVillaalgordo/Lank:2026} to $\Phi_Q$ and to the integral transform $D_{\operatorname{qc}}(Y_2) \to D_{\operatorname{qc}}(Y_1)$ with kernel $Q$, i.e.\ with the roles of the two factors switched.
    Set $Q_R:=\operatorname{\mathbf{R}\mathcal{H}\!\mathit{om}}(Q, p_2^! \mathcal{O}_{Y_2})$.
    By \cite[Proposition 4.19 \& Corollary 4.20]{GuisadoVillaalgordo/Lank:2026}, $Q_R\in D^b_{\operatorname{coh}}(Y_1\times_S Y_2)$, $\Phi_Q$ is left adjoint to $\Phi_{Q_R}$, and the adjunction restricts to $D^b_{\operatorname{coh}}$. 
    By \cite[Lemma 5.20]{GuisadoVillaalgordo/Lank:2026}, $p_2^! \mathcal{O}_{Y_2}$ is perfect. 
    By \cite[Lemma 2.10]{GuisadoVillaalgordo/Lank:2026},
    \begin{displaymath}
        \operatorname{\mathbf{R}\mathcal{H}\!\mathit{om}}(Q, p_2^! \mathcal{O}_{Y_2}) 
        \cong \operatorname{\mathbb{R}\mathcal{H}\!\mathit{om}}(Q, p_2^! \mathcal{O}_{Y_2}).
    \end{displaymath}
    By \cite[\href{https://stacks.math.columbia.edu/tag/08JJ}{Tag 08JJ}]{StacksProject}, 
    \begin{displaymath}
        \operatorname{\mathbb{R}\mathcal{H}\!\mathit{om}}(Q, p_2^! \mathcal{O}_{Y_2})
        \cong \operatorname{\mathbb{R}\mathcal{H}\!\mathit{om}}(Q, \mathcal{O}_{Y_1\times_S Y_2}) \otimes^{\mathbf{L}} p_2^! \mathcal{O}_{Y_2}.
    \end{displaymath}
    A further application of \cite[Lemma 2.10]{GuisadoVillaalgordo/Lank:2026} yields that 
    \begin{displaymath}
        \operatorname{\mathbb{R}\mathcal{H}\!\mathit{om}}(Q, \mathcal{O}_{Y_1\times_S Y_2}) 
        \cong \operatorname{\mathbf{R}\mathcal{H}\!\mathit{om}}(Q, \mathcal{O}_{Y_1\times_S Y_2}).
    \end{displaymath}
    Hence, we obtain 
    \begin{displaymath}
        \operatorname{\mathbf{R}\mathcal{H}\!\mathit{om}}(Q, p_2^! \mathcal{O}_{Y_2})
        \cong \operatorname{\mathbf{R}\mathcal{H}\!\mathit{om}}(Q, \mathcal{O}_{Y_1\times_S Y_2}) \otimes^{\mathbf{L}} p_2^! \mathcal{O}_{Y_2}.
    \end{displaymath}
    Since $\operatorname{\mathbf{R}\mathcal{H}\!\mathit{om}}(Q, \mathcal{O}_{Y_1\times_S Y_2})$ and $p_2^! \mathcal{O}_{Y_2}$ are perfect, so is $Q_R$.
    Consequently, the first observation implies that $Q_R$ is relatively perfect over each $Y_i$. 
    Thus, \Cref{setup:relatively_perfect_adjoint_pair} applies to every $Q\in \operatorname{Perf}(Y_1\times_S Y_2)$.
\end{example}

\begin{proposition}
    \label{prop:compare_dsg_for_smooth}
    Let $S$ be a quasi-compact quasi-separated quasi-excellent algebraic space, $f\colon Y \to S$ be a proper flat morphism of algebraic spaces, and $g \colon W \to Y$ be a proper Gorenstein morphism of algebraic spaces.
    Set $\Gamma_g = (1_W,g) \colon W \to W \times_S Y$ to be the graph of $g$. 
    Then $\mathbf{R}(\Gamma_g)_\ast \mathcal{O}_W$ satisfies \Cref{setup:relatively_perfect_adjoint_pair} with respect to $f\circ g$ and $f$.
    Moreover, $\Phi_{\mathbf{R}(\Gamma_g)_\ast \mathcal{O}_W}$ is naturally isomorphic to $\mathbf{R}g_\ast$, and $\operatorname{fm}_{\operatorname{sg}}(\mathbf{R}(\Gamma_g)_\ast \mathcal{O}_W)$ and $\operatorname{FM}_{\operatorname{sg}}(\mathbf{R}(\Gamma_g)_\ast \mathcal{O}_W)$ are open subsets of $|S|$.
\end{proposition}

\begin{proof}
    Form the fibered square
    \begin{displaymath}
        \begin{tikzcd}
            {W \times_S Y} & Y \\
            W & {S.}
            \arrow["{p_2}", from=1-1, to=1-2]
            \arrow["{p_1}"', from=1-1, to=2-1]
            \arrow["f", from=1-2, to=2-2]
            \arrow["{f\circ g}"', from=2-1, to=2-2]
        \end{tikzcd}
    \end{displaymath}
    Note that $f\circ g$ is proper and flat.
    Base change implies $p_1$ and $p_2$ are proper and flat.
    As $f$ is proper, it is separated, and so $\Gamma_g$ is a closed immersion \cite[\href{https://stacks.math.columbia.edu/tag/03KO}{Tag 03KO}]{StacksProject}.
    In particular, $(\Gamma_g)_\ast$ is exact, and so $\mathbf{R}(\Gamma_g)_\ast \mathcal{O}_W \cong (\Gamma_g)_\ast \mathcal{O}_W$ \cite[\href{https://stacks.math.columbia.edu/tag/073H}{Tag 073H}]{StacksProject}.
    Hence, $\mathbf{R}(\Gamma_g)_\ast \mathcal{O}_W\in D^b_{\operatorname{coh}}(W\times_S Y)$.

    We first identify $\Phi_{\mathbf{R}(\Gamma_g)_\ast \mathcal{O}_W}$ with $\mathbf{R}g_\ast$.
    For $A\in D_{\operatorname{qc}}(W)$, there exist natural isomorphisms:
    \begin{displaymath}
        \begin{aligned}
            \Phi_{\mathbf{R}(\Gamma_g)_\ast \mathcal{O}_W}(A)
            &= \mathbf{R}(p_2)_\ast (\mathbf{L}p_1^\ast A \otimes^{\mathbf{L}} \mathbf{R}(\Gamma_g)_\ast \mathcal{O}_W)
            \\&\cong \mathbf{R}(p_2)_\ast \mathbf{R}(\Gamma_g)_\ast \mathbf{L}\Gamma_g^\ast \mathbf{L}p_1^\ast A && \textrm{(\cite[\href{https://stacks.math.columbia.edu/tag/08IN}{Tag 08IN}]{StacksProject})}
            \\&\cong \mathbf{R}(p_2\circ \Gamma_g)_\ast \mathbf{L}(p_1\circ \Gamma_g)^\ast A && \textrm{(pseudofunctoriality)}
            \\&= \mathbf{R}g_\ast A && (p_1 \circ \Gamma_g = 1_W, p_2 \circ \Gamma_g = g).
        \end{aligned}
    \end{displaymath}
    As $g$ is proper and flat, it has finite tor-dimension, and so \Cref{lem:quasi-perfect_morphism} implies $g$ is quasi-perfect.
    Hence, $\mathbf{R}g_\ast$ preserves perfect complexes.
    Moreover, properness of $g$ implies $\mathbf{R}g_\ast$ preserves bounded pseudocoherent complexes.
    By \cite[Corollary 1.9]{GuisadoVillaalgordo/Lank:2026}, $\mathbf{R}(\Gamma_g)_\ast \mathcal{O}_W$ is relatively perfect over $W$ and $Y$.

    Next, we treat the right adjoint.
    By \cite[Proposition 4.19 \& Corollary 4.20]{GuisadoVillaalgordo/Lank:2026}, 
    \begin{displaymath}
        \operatorname{\mathbf{R}\mathcal{H}\!\mathit{om}}(\mathbf{R}(\Gamma_g)_\ast \mathcal{O}_W, p_2^! \mathcal{O}_Y) \in D^b_{\operatorname{coh}}(W \times_S Y)
    \end{displaymath}
    and on $D_{\operatorname{qc}}$,
    \begin{displaymath}
        \Phi_{\mathbf{R}(\Gamma_g)_\ast \mathcal{O}_W} 
        \dashv \Phi_{\operatorname{\mathbf{R}\mathcal{H}\!\mathit{om}}(\mathbf{R}(\Gamma_g)_\ast \mathcal{O}_W, p_2^! \mathcal{O}_Y)}.
    \end{displaymath}
    On the other hand, $\mathbf{R}g_\ast \dashv g^\times$ on $D_{\operatorname{qc}}$, and $g^\times \cong g^!$ by \cite[Lemma C.2]{GuisadoVillaalgordo/Lank:2026}.
    By the uniqueness of adjoints, there exists a natural isomorphism 
    \begin{displaymath}
        \Phi_{\operatorname{\mathbf{R}\mathcal{H}\!\mathit{om}}(\mathbf{R}(\Gamma_g)_\ast \mathcal{O}_W, p_2^! \mathcal{O}_Y)} \cong g^!.
    \end{displaymath}
    As $g$ has finite tor-dimension, \cite[Lemma C.4]{GuisadoVillaalgordo/Lank:2026} gives a natural isomorphism for all $P\in D_{\operatorname{qc}}(Y)$,
    \begin{displaymath}
        \mathbf{L}g^\ast P \otimes^{\mathbf{L}} g^! \mathcal{O}_Y \xrightarrow{\cong} g^! P.
    \end{displaymath}
    Moreover, $g$ being proper, flat, and Gorenstein implies $g^! \mathcal{O}_Y$ is invertible \cite[Lemma 5.20]{GuisadoVillaalgordo/Lank:2026}, and  hence perfect.
    Consequently, there exists a natural isomorphism
    \begin{displaymath}
        \Phi_{\operatorname{\mathbf{R}\mathcal{H}\!\mathit{om}}(\mathbf{R}(\Gamma_g)_\ast \mathcal{O}_W, p_2^! \mathcal{O}_Y)} (-) \cong \mathbf{L}g^\ast (-) \otimes^{\mathbf{L}} g^! \mathcal{O}_Y.
    \end{displaymath}
    The functor $\mathbf{L}g^\ast$ preserves perfect complexes.
    Since $g$ is flat, $\mathbf{L}g^\ast$ preserves bounded pseudocoherent complexes.
    Moreover, as $g^! \mathcal{O}_Y$ is invertible, $(-)\otimes^{\mathbf{L}} g^! \mathcal{O}_Y$ is an endofunctor on $D^b_{\operatorname{coh}}(W)$ and $\operatorname{Perf}(W)$. 
    Hence, $\Phi_{\operatorname{\mathbf{R}\mathcal{H}\!\mathit{om}}(\mathbf{R}(\Gamma_g)_\ast \mathcal{O}_W, p_2^! \mathcal{O}_Y)}$ preserves perfect complexes and bounded pseudocoherent complexes.
    By \cite[Corollary 1.9]{GuisadoVillaalgordo/Lank:2026},
    \begin{displaymath}
        \operatorname{\mathbf{R}\mathcal{H}\!\mathit{om}}(\mathbf{R}(\Gamma_g)_\ast \mathcal{O}_W, p_2^! \mathcal{O}_Y)
    \end{displaymath}
    is relatively perfect over $W$ and $Y$.

    Consequently, $\mathbf{R}(\Gamma_g)_\ast \mathcal{O}_W$ satisfies \Cref{setup:relatively_perfect_adjoint_pair} with respect to $f\circ g \colon W \to S$ and $f\colon Y \to S$.
    By \Cref{thm:openness_singular_fm}, $\operatorname{fm}_{\operatorname{sg}}(\mathbf{R}(\Gamma_g)_\ast \mathcal{O}_W)$ and $\operatorname{FM}_{\operatorname{sg}}(\mathbf{R}(\Gamma_g)_\ast \mathcal{O}_W)$ are open subsets of $|S|$.
\end{proof}

\begin{proof}
    [Proof of \Cref{prop:ideal_sheaf_diagonal}]
    As $f$ is proper, it is separated, and so $\Delta_f$ is a closed immersion \cite[\href{https://stacks.math.columbia.edu/tag/03HL}{Tag 03HL}]{StacksProject}.
    Form the fibered square
    \begin{displaymath}
        \begin{tikzcd}
            {Y\times_S Y} & Y \\
            Y & {S.}
            \arrow["{p_2}", from=1-1, to=1-2]
            \arrow["{p_1}"', from=1-1, to=2-1]
            \arrow["f", from=1-2, to=2-2]
            \arrow["f"', from=2-1, to=2-2]
        \end{tikzcd}
    \end{displaymath}
    Denote by $\eta$ the unit of $\mathbf{L}\Delta_f^\ast \dashv \mathbf{R}(\Delta_f)_\ast$.
    By \cite[\href{https://stacks.math.columbia.edu/tag/073H}{Tag 073H}]{StacksProject}, the canonical morphism $(\Delta_f)_\ast \mathcal{O}_Y \to \mathbf{R} (\Delta_f)_\ast \mathcal{O}_Y$ is an isomorphism.
    Hence, the short exact sequence 
    \begin{displaymath}
        0 \to \mathcal{I}_{\Delta_f} \to \mathcal{O}_{Y\times_S Y} \to (\Delta_f)_\ast \mathcal{O}_Y \to 0
    \end{displaymath}
    yields a morphism of distinguished triangles
    \begin{displaymath}
        \begin{tikzcd}
            {\mathcal{I}_{\Delta_f}} & {\mathcal{O}_{Y\times_S Y}} & {(\Delta_f)_\ast \mathcal{O}_Y} & {\mathcal{I}_{\Delta_f}[1]} \\
            {\operatorname{cone}(\eta_{\mathcal{O}_{Y\times_S Y}})[-1]} & {\mathcal{O}_{Y\times_S Y}} & {\mathbf{R} (\Delta_f)_\ast \mathcal{O}_Y} & {\operatorname{cone}(\eta_{\mathcal{O}_{Y\times_S Y}})}
            \arrow[from=1-1, to=1-2]
            \arrow[from=1-1, to=2-1]
            \arrow[from=1-2, to=1-3]
            \arrow["{=}", from=1-2, to=2-2]
            \arrow["\delta", from=1-3, to=1-4]
            \arrow["\cong", from=1-3, to=2-3]
            \arrow[from=1-4, to=2-4]
            \arrow[from=2-1, to=2-2]
            \arrow["{\eta_{\mathcal{O}_{Y\times_S Y}}}", from=2-2, to=2-3]
            \arrow[from=2-3, to=2-4]
        \end{tikzcd}
    \end{displaymath}
    whose vertical morphisms are isomorphisms.
    In particular, $\mathcal{I}_{\Delta_f} \cong \operatorname{cone}(\eta_{\mathcal{O}_{Y\times_S Y}})[-1]$.
    By \cite[Lemma 5.32]{GuisadoVillaalgordo/Lank:2026}, $\mathcal{I}_{\Delta_f}$ is relatively perfect over each projection $p_i \colon Y\times_S Y \to Y$.
    Then \cite[Corollary 1.9]{GuisadoVillaalgordo/Lank:2026} implies $\Phi_{\mathcal{I}_{\Delta_f}}$ restricts to an exact endofunctor on $D^b_{\operatorname{coh}}(Y)$ which preserves $\operatorname{Perf}(Y)$.
    Denote by $\pi_Y \colon D^b_{\operatorname{coh}}(Y) \to D_{\operatorname{sg}}(Y)$ the Verdier localization functor.
    By \cite[\href{https://stacks.math.columbia.edu/tag/05RJ}{Tag 05RJ}]{StacksProject}, there exists an exact functor $\Phi_{\mathcal{I}_{\Delta_f}}^{\operatorname{sg}}$ making the following square commute:
    \begin{displaymath}
        \begin{tikzcd}
            {D^b_{\operatorname{coh}}(Y)} & {D^b_{\operatorname{coh}}(Y)} \\
            {D_{\operatorname{sg}}(Y)} & {D_{\operatorname{sg}}(Y).}
            \arrow["{\Phi_{\mathcal{I}_{\Delta_f}}}", from=1-1, to=1-2]
            \arrow["{\pi_Y}"', from=1-1, to=2-1]
            \arrow["{\pi_Y}", from=1-2, to=2-2]
            \arrow["{\Phi_{\mathcal{I}_{\Delta_f}}^{\operatorname{sg}}}"', from=2-1, to=2-2]
        \end{tikzcd}
    \end{displaymath}
    We do not assume $f$ is Gorenstein, so it is not clear a priori whether $\Phi_{\mathcal{I}_{\Delta_f}}^{\operatorname{sg}}$ admits an adjoint.
    Nevertheless, we will prove that there exists a natural isomorphism $\Phi_{\mathcal{I}_{\Delta_f}}^{\operatorname{sg}} \cong [-1]$.
    In such a case, $\Phi_{\mathcal{I}_{\Delta_f}}^{\operatorname{sg}}$ is an autoequivalence of $D_{\operatorname{sg}}(Y)$.

    Apply integral transforms to the top row of the diagram above.
    Use the projection formula to identify $\Phi_{(\Delta_f)_\ast \mathcal{O}_Y}$ with the identity.
    This yields a distinguished triangle which is natural in $A\in D^b_{\operatorname{coh}}(Y)$:
    \begin{displaymath}
        \Phi_{\mathcal{I}_{\Delta_f}}(A) \to \mathbf{R}(p_2)_\ast \mathbf{L}p_1^\ast A \to A \xrightarrow{h_A} \Phi_{\mathcal{I}_{\Delta_f}}(A)[1].
    \end{displaymath}
    Here $h$ is the natural transformation induced by $\delta \colon (\Delta_f)_\ast \mathcal{O}_Y \to \mathcal{I}_{\Delta_f}[1]$.
    By flat base change, there exists a natural isomorphism $\mathbf{R}(p_2)_\ast \mathbf{L}p_1^\ast A \cong \mathbf{L}f^\ast \mathbf{R}f_\ast A$.
    As $f$ is proper, $\mathbf{R}f_\ast A \in D^b_{\operatorname{coh}}(S)$.
    As $S$ is regular, \cite[Lemma 5.28]{GuisadoVillaalgordo/Lank:2026} implies $\mathbf{R}f_\ast A$ is perfect, and so $\mathbf{L}f^\ast \mathbf{R}f_\ast A$ is perfect \cite[\href{https://stacks.math.columbia.edu/tag/08H6}{Tag 08H6}]{StacksProject}.
    Therefore, $\operatorname{cone}(h_A) \cong \mathbf{L}f^\ast \mathbf{R}f_\ast A[1] \in \operatorname{Perf}(Y)$, and so $\pi_Y(h_A)$ is an isomorphism for every $A\in D^b_{\operatorname{coh}}(Y)$.

    We check that the $\pi_Y(h_A)$ are natural in $D_{\operatorname{sg}}(Y)$.
    Write $\Psi := \Phi_{\mathcal{I}_{\Delta_f}}[1]$ and $\Psi^{\operatorname{sg}} := \Phi_{\mathcal{I}_{\Delta_f}}^{\operatorname{sg}}[1]$.
    Let $\alpha \colon \pi_Y(A^\prime) \to \pi_Y(A)$ be a morphism in $D_{\operatorname{sg}}(Y)$.
    Choose a roof which represents $\alpha$,
    \begin{displaymath}
        \begin{tikzcd}
            & {A_0} & \\
            {A^\prime} && A
            \arrow["{\alpha_1}"', from=1-2, to=2-1]
            \arrow["{\alpha_2}", from=1-2, to=2-3]
        \end{tikzcd}
    \end{displaymath}
    where $\operatorname{cone}(\alpha_1)\in \operatorname{Perf}(Y)$.
    Here $\alpha = \pi_Y(\alpha_2) \circ \pi_Y(\alpha_1)^{-1}$.
    As $\operatorname{cone}(\Psi(\alpha_1)) \cong \Psi(\operatorname{cone}(\alpha_1))$ is perfect, $\pi_Y(\Psi(\alpha_1))$ is an isomorphism.
    By the naturality of $h$ in $D^b_{\operatorname{coh}}(Y)$, $h_A \circ \alpha_2 = \Psi(\alpha_2) \circ h_{A_0}$ and $h_{A^\prime} \circ \alpha_1 = \Psi(\alpha_1) \circ h_{A_0}$.
    It follows that
    \begin{displaymath}
        \begin{aligned}
            \pi_Y (h_A) \circ \alpha
            &= \pi_Y(h_A \circ \alpha_2) \circ \pi_Y(\alpha_1)^{-1}
            \\&= \pi_Y(\Psi(\alpha_2)) \circ \pi_Y(h_{A_0}) \circ \pi_Y(\alpha_1)^{-1}
            \\&= \pi_Y(\Psi(\alpha_2)) \circ \pi_Y(\Psi(\alpha_1))^{-1} \circ \pi_Y(h_{A^\prime})
            \\&= \Psi^{\operatorname{sg}}(\alpha) \circ \pi_Y (h_{A^\prime}).
        \end{aligned}
    \end{displaymath}
    Hence, the morphisms $\pi_Y(h_A)$ give a natural isomorphism $1_{D_{\operatorname{sg}}(Y)} \to \Phi^{\operatorname{sg}}_{\mathcal{I}_{\Delta_f}}[1]$.
    That is, $\Phi^{\operatorname{sg}}_{\mathcal{I}_{\Delta_f}} \cong [-1]$, and so $\Phi^{\operatorname{sg}}_{\mathcal{I}_{\Delta_f}}$ is an autoequivalence of $D_{\operatorname{sg}}(Y)$ with quasi-inverse $[1]$.
\end{proof}

\begin{lemma}
    \label{lem:diagonal_ideal_kills_pullbacks}
    Let $f\colon Y \to S$ be a proper flat morphism of Noetherian algebraic spaces.
    Denote by $\mathcal{I}_{\Delta_f}$ the ideal sheaf of the diagonal $\Delta_f \colon Y \to Y\times_S Y$.
    Set $\epsilon$ to be the counit of $\mathbf{L}f^\ast \dashv \mathbf{R}f_\ast$.
    \begin{enumerate}
        \item \label{lem:diagonal_ideal_kills_pullbacks1} For every $B\in D_{\operatorname{qc}}(Y)$, there exists a distinguished triangle
        \begin{displaymath}
            \Phi_{\mathcal{I}_{\Delta_f}}(B) \to \mathbf{L}f^\ast \mathbf{R}f_\ast B \xrightarrow{\epsilon_B} B \to \Phi_{\mathcal{I}_{\Delta_f}}(B)[1].
        \end{displaymath}
        \item \label{lem:diagonal_ideal_kills_pullbacks2} If the canonical morphism $\mathcal{O}_S \to \mathbf{R}f_\ast \mathcal{O}_Y$ is an isomorphism, then for all $A\in D_{\operatorname{qc}}(S)$, we have $\Phi_{\mathcal{I}_{\Delta_f}}(\mathbf{L}f^\ast A) \cong 0$.
    \end{enumerate}
\end{lemma}

\begin{proof}
    We first record an observation. 
    Let $g\colon X^\prime \to X$ be a quasi-compact quasi-separated morphism of algebraic spaces.
    Denote by $\eta^g$ and $\epsilon^g$ the unit and counit of $\mathbf{L}g^\ast \dashv \mathbf{R}g_\ast$.
    For $M\in D_{\operatorname{qc}}(X)$, the projection formula \cite[\href{https://stacks.math.columbia.edu/tag/08IN}{Tag 08IN}]{StacksProject} gives an isomorphism
    \begin{displaymath}
        \pi_M \colon \mathbf{R}g_\ast \mathcal{O}_{X^\prime} \otimes^{\mathbf{L}} M \to \mathbf{R}g_\ast (\mathcal{O}_{X^\prime} \otimes^{\mathbf{L}} \mathbf{L}g^\ast M) = \mathbf{R}g_\ast \mathbf{L}g^\ast M.
    \end{displaymath}
    By construction (see loc.\ cit.), $\pi_M$ is adjoint to the composite
    \begin{displaymath}
        \begin{aligned}
            \mathbf{L}g^\ast (\mathbf{R}g_\ast \mathcal{O}_{X^\prime} \otimes^{\mathbf{L}} M)
            &\xrightarrow{\mu} \mathbf{L}g^\ast \mathbf{R}g_\ast \mathcal{O}_{X^\prime} \otimes^{\mathbf{L}} \mathbf{L}g^\ast M
            \\&\xrightarrow{\epsilon^g_{\mathcal{O}_{X^\prime}} \otimes 1} \mathcal{O}_{X^\prime} \otimes^{\mathbf{L}} \mathbf{L}g^\ast M.
        \end{aligned}
    \end{displaymath}
    Here $\mu$ is the monoidal structure of $\mathbf{L}g^\ast$ \cite[\href{https://stacks.math.columbia.edu/tag/07A4}{Tag 07A4}]{StacksProject}.
    After making the identifications $\mathcal{O}_X \otimes^{\mathbf{L}} M = M$ and $\mathbf{L}g^\ast \mathcal{O}_X = \mathcal{O}_{X^\prime}$, we claim that
    \begin{equation}
        \label{eq:projection_formula_unit}
        \pi_M \circ (\eta^g_{\mathcal{O}_X} \otimes 1_M) = \eta^g_M.
    \end{equation}
    Indeed, the adjoint of $\eta^g_M$ is $1_{\mathbf{L}g^\ast M}$.
    By the naturality of $\mu$ and the triangle identity $\epsilon^g_{\mathbf{L}g^\ast \mathcal{O}_X} \circ \mathbf{L}g^\ast (\eta^g_{\mathcal{O}_X}) = 1$, the adjoint of the left-hand side is
    \begin{displaymath}
        \begin{aligned}
            (\epsilon^g_{\mathcal{O}_{X^\prime}} \otimes 1) \circ \mu \circ \mathbf{L}g^\ast (\eta^g_{\mathcal{O}_X} \otimes 1_M)
            &= ((\epsilon^g_{\mathcal{O}_{X^\prime}} \circ \mathbf{L}g^\ast (\eta^g_{\mathcal{O}_X})) \otimes 1) \circ \mu
            \\&= \mu.
        \end{aligned}
    \end{displaymath}
    Since $\mu$ is compatible with the units of the monoidal structures, it is the identity of $\mathbf{L}g^\ast M$ under the identifications above.
    This proves $\eqref{eq:projection_formula_unit}$.

    Now we prove \eqref{lem:diagonal_ideal_kills_pullbacks1}.
    Denote by $p_1,p_2 \colon Y\times_S Y \to Y$ the canonical projections.
    As $f$ is proper, it is separated, and so $\Delta_f$ is a closed immersion \cite[\href{https://stacks.math.columbia.edu/tag/03HL}{Tag 03HL}]{StacksProject}.
    By \cite[\href{https://stacks.math.columbia.edu/tag/073H}{Tag 073H}]{StacksProject}, the canonical morphism $(\Delta_f)_\ast \mathcal{O}_Y \to \mathbf{R}(\Delta_f)_\ast \mathcal{O}_Y$ is an isomorphism.
    Denote by $\eta^{\Delta_f}$ the unit of $\mathbf{L}\Delta_f^\ast \dashv \mathbf{R}(\Delta_f)_\ast$.
    Under the identification above, the canonical surjection $\mathcal{O}_{Y\times_S Y} \to (\Delta_f)_\ast \mathcal{O}_Y$ has kernel $\mathcal{I}_{\Delta_f}$, and it is the unit
    \begin{displaymath}
        \eta^{\Delta_f}_{\mathcal{O}_{Y\times_S Y}} \colon \mathcal{O}_{Y\times_S Y} \to \mathbf{R}(\Delta_f)_\ast \mathbf{L}\Delta_f^\ast \mathcal{O}_{Y\times_S Y} = \mathbf{R}(\Delta_f)_\ast \mathcal{O}_Y.
    \end{displaymath}
    Indeed, both sides are concentrated in degree zero, and on zeroth cohomology sheaves the unit is $\Delta_f^\sharp$.
    Fix $B\in D_{\operatorname{qc}}(Y)$.
    Applying $\mathbf{R}(p_2)_\ast (\mathbf{L}p_1^\ast B \otimes^{\mathbf{L}} -)$ to the short exact sequence 
    \begin{displaymath}
        0 \to \mathcal{I}_{\Delta_f} \to \mathcal{O}_{Y\times_S Y} \to (\Delta_f)_\ast \mathcal{O}_Y \to 0
    \end{displaymath}
    yields a distinguished triangle
    \begin{displaymath}
        \begin{aligned}
            \Phi_{\mathcal{I}_{\Delta_f}}(B)
            &\to \mathbf{R}(p_2)_\ast \mathbf{L}p_1^\ast B
            \\&\to \mathbf{R}(p_2)_\ast (\mathbf{L}p_1^\ast B \otimes^{\mathbf{L}} \mathbf{R}(\Delta_f)_\ast \mathcal{O}_Y)
            \to \Phi_{\mathcal{I}_{\Delta_f}}(B)[1],
        \end{aligned}
    \end{displaymath}
    whose middle morphism is $\mathbf{R}(p_2)_\ast (1 \otimes \eta^{\Delta_f}_{\mathcal{O}_{Y\times_S Y}})$.
    By $\eqref{eq:projection_formula_unit}$ for $g = \Delta_f$ and $M = \mathbf{L}p_1^\ast B$ (and the symmetry of $\otimes^{\mathbf{L}}$), the isomorphism $\pi_{\mathbf{L}p_1^\ast B}$ identifies $1 \otimes \eta^{\Delta_f}_{\mathcal{O}_{Y\times_S Y}}$ with the unit $\eta^{\Delta_f}_{\mathbf{L}p_1^\ast B}$.
    By the pseudofunctoriality isomorphisms $\mathbf{R}(p_2)_\ast \mathbf{R}(\Delta_f)_\ast \cong 1$ and $\mathbf{L}\Delta_f^\ast \mathbf{L}p_1^\ast \cong 1$, the middle morphism of the triangle is identified with the composite
    \begin{equation}
        \label{eq:diagonal_middle_morphism}
        \mathbf{R}(p_2)_\ast \mathbf{L}p_1^\ast B
        \xrightarrow{\mathbf{R}(p_2)_\ast (\eta^{\Delta_f}_{\mathbf{L}p_1^\ast B})} \mathbf{R}(p_2)_\ast \mathbf{R}(\Delta_f)_\ast \mathbf{L}\Delta_f^\ast \mathbf{L}p_1^\ast B
        \cong B.
    \end{equation}

    Next, consider the commutative diagram
    \begin{displaymath}
        \begin{tikzcd}
            Y & {Y\times_S Y} & Y \\
            Y & Y & {S.}
            \arrow["{\Delta_f}", from=1-1, to=1-2]
            \arrow["{1_Y}"', from=1-1, to=2-1]
            \arrow["{p_1}", from=1-2, to=1-3]
            \arrow["{p_2}", from=1-2, to=2-2]
            \arrow["f", from=1-3, to=2-3]
            \arrow["{1_Y}"', from=2-1, to=2-2]
            \arrow["f"', from=2-2, to=2-3]
        \end{tikzcd}
    \end{displaymath}
    Here the right square is fibered, and the left square commutes because $p_2 \circ \Delta_f = 1_Y$.
    As $p_1 \circ \Delta_f = 1_Y$, its outer rectangle is the commutative square whose top and left sides are $1_Y$, and whose right side and bottom are $f$.
    Denote by $\vartheta^r \colon \mathbf{L}f^\ast \mathbf{R}f_\ast \to \mathbf{R}(p_2)_\ast \mathbf{L}p_1^\ast$, $\vartheta^l \colon \mathbf{R}(p_2)_\ast \to \mathbf{L}\Delta_f^\ast$, $\vartheta \colon \mathbf{L}f^\ast \mathbf{R}f_\ast \to 1$ respectively the ${}^\ast$-${}_\ast$ base change morphisms \cite[Definition D.9]{GuisadoVillaalgordo/Lank:2026} of the right square, the left square, and the outer rectangle.
    Flatness of $f$ implies that the right square is tor-independent, and so $\vartheta^r$ is an isomorphism \cite[\href{https://stacks.math.columbia.edu/tag/08IR}{Tag 08IR}]{StacksProject}.
    By \cite[Definition D.9(1)]{GuisadoVillaalgordo/Lank:2026}, each of these is the left mate of the corresponding pseudofunctoriality isomorphism for derived pushforward.
    We make these left mates explicit using \cite[Lemma D.2]{GuisadoVillaalgordo/Lank:2026}.
    For the left square, the adjunction along the bottom side is the identity adjunction of $D_{\operatorname{qc}}(Y)$.
    Hence, for $W\in D_{\operatorname{qc}}(Y\times_S Y)$, the morphism $\vartheta^l_W$ is the composite
    \begin{displaymath}
        \mathbf{R}(p_2)_\ast W \xrightarrow{\mathbf{R}(p_2)_\ast (\eta^{\Delta_f}_W)} \mathbf{R}(p_2)_\ast \mathbf{R}(\Delta_f)_\ast \mathbf{L}\Delta_f^\ast W \cong \mathbf{L}\Delta_f^\ast W,
    \end{displaymath}
    where the last morphism is induced by the pseudofunctoriality isomorphism
    \begin{displaymath}
        \mathbf{R}(p_2)_\ast \mathbf{R}(\Delta_f)_\ast \cong \mathbf{R}(p_2 \circ \Delta_f)_\ast = 1.
    \end{displaymath}
    For the outer rectangle, the adjunction along the top side is the identity adjunction of $D_{\operatorname{qc}}(Y)$, and the pseudofunctoriality isomorphism is the identity of $\mathbf{R}f_\ast$.
    Hence, $\vartheta = \epsilon$; compare \cite[Remark D.5(1)]{GuisadoVillaalgordo/Lank:2026}.
    By \cite[Lemma D.7]{GuisadoVillaalgordo/Lank:2026}, the mate correspondence is compatible with horizontal pasting.
    Since pseudofunctoriality is associative, the pasted pseudofunctoriality isomorphisms of the two squares give that of the outer rectangle.
    Therefore, under the canonical pseudofunctorial identifications, $\vartheta = (\vartheta^l \mathbf{L}p_1^\ast) \circ \vartheta^r$.
    This is the argument of \cite[Lemma D.12(1)]{GuisadoVillaalgordo/Lank:2026}, which does not use that the squares are fibered.
    By evaluating at $B$, and using the pseudofunctoriality isomorphism $\mathbf{L}\Delta_f^\ast \mathbf{L}p_1^\ast \cong \mathbf{L}(p_1 \circ \Delta_f)^\ast = 1$, it follows that $\epsilon_B$ is the composite of $\vartheta^r_B$ with the morphism $\eqref{eq:diagonal_middle_morphism}$.
    As $\vartheta^r_B$ is an isomorphism, the distinguished triangle above is isomorphic to the triangle
    \begin{displaymath}
        \Phi_{\mathcal{I}_{\Delta_f}}(B) \to \mathbf{L}f^\ast \mathbf{R}f_\ast B \xrightarrow{\epsilon_B} B \to \Phi_{\mathcal{I}_{\Delta_f}}(B)[1],
    \end{displaymath}
    and hence, is distinguished.
    This proves \eqref{lem:diagonal_ideal_kills_pullbacks1}.

    We prove \eqref{lem:diagonal_ideal_kills_pullbacks2}.
    Let $A\in D_{\operatorname{qc}}(S)$.
    Denote by $\eta^f$ the unit of $\mathbf{L}f^\ast \dashv \mathbf{R}f_\ast$.
    The canonical morphism $\mathcal{O}_S \to \mathbf{R}f_\ast \mathcal{O}_Y = \mathbf{R}f_\ast \mathbf{L}f^\ast \mathcal{O}_S$ is $\eta^f_{\mathcal{O}_S}$, which is an isomorphism by hypothesis.
    By $\eqref{eq:projection_formula_unit}$ for $g = f$ and $M = A$, we have $\eta^f_A = \pi_A \circ (\eta^f_{\mathcal{O}_S} \otimes 1_A)$.
    Hence, $\eta^f_A$ is an isomorphism.
    By the triangle identity $\epsilon_{\mathbf{L}f^\ast A} \circ \mathbf{L}f^\ast (\eta^f_A) = 1$, the counit $\epsilon_{\mathbf{L}f^\ast A}$ is an isomorphism.
    Hence, \eqref{lem:diagonal_ideal_kills_pullbacks1} with $B = \mathbf{L}f^\ast A$ gives $\Phi_{\mathcal{I}_{\Delta_f}}(\mathbf{L}f^\ast A) \cong 0$.
\end{proof}

\begin{proposition}
    \label{prop:explicit_loci_diagonal}
    Let $S$ be a quasi-compact quasi-separated quasi-excellent algebraic space.
    Suppose $f\colon Y \to S$ is a proper Gorenstein morphism of algebraic spaces such that the canonical morphism $\mathcal{O}_S \to \mathbf{R}f_\ast \mathcal{O}_Y$ is an isomorphism.
    Denote by $\mathcal{I}_{\Delta_f}$ the ideal sheaf of the diagonal $\Delta_f \colon Y \to Y\times_S Y$.
    Then $\mathcal{I}_{\Delta_f}$ satisfies \Cref{setup:relatively_perfect_adjoint_pair} with respect to $f$ and $f$.
    Moreover, 
    \begin{displaymath}
        \operatorname{fm}_{\operatorname{sg}}(\mathcal{I}_{\Delta_f}) = \operatorname{FM}_{\operatorname{sg}}(\mathcal{I}_{\Delta_f}) = \operatorname{reg}(S).
    \end{displaymath}
\end{proposition}

\begin{proof}
    Denote by $p_1,p_2 \colon Y\times_S Y \to Y$ the canonical projections.
    We first check \Cref{setup:relatively_perfect_adjoint_pair}.
    As in the proof of \Cref{prop:ideal_sheaf_diagonal}, $\mathcal{I}_{\Delta_f}$ is a coherent sheaf which is relatively perfect over each factor \cite[Lemma 5.32]{GuisadoVillaalgordo/Lank:2026}.
    By \cite[Proposition 4.19 \& Corollary 4.20]{GuisadoVillaalgordo/Lank:2026},
    \begin{displaymath}
        \operatorname{\mathbf{R}\mathcal{H}\!\mathit{om}}(\mathcal{I}_{\Delta_f}, p_2^! \mathcal{O}_Y) \in D^b_{\operatorname{coh}}(Y\times_S Y).
    \end{displaymath}
    By \cite[Corollary 1.9]{GuisadoVillaalgordo/Lank:2026}, 
    it suffices to show that the integral transform from the second factor to the first with kernel $\operatorname{\mathbf{R}\mathcal{H}\!\mathit{om}}(\mathcal{I}_{\Delta_f}, p_2^! \mathcal{O}_Y)$ preserves $D^b_{\operatorname{coh}}(Y)$ and $\operatorname{Perf}(Y)$.
    Applying $\operatorname{\mathbf{R}\mathcal{H}\!\mathit{om}}(-, p_2^! \mathcal{O}_Y)$ to the distinguished triangle
    \begin{displaymath}
        \mathcal{I}_{\Delta_f}
        \to \mathcal{O}_{Y\times_S Y}
        \to (\Delta_f)_\ast \mathcal{O}_Y
        \to \mathcal{I}_{\Delta_f}[1]
    \end{displaymath}
    gives a distinguished triangle
    \begin{displaymath}
        \begin{aligned}
            \operatorname{\mathbf{R}\mathcal{H}\!\mathit{om}}((\Delta_f)_\ast \mathcal{O}_Y, p_2^! \mathcal{O}_Y)
            &\to \operatorname{\mathbf{R}\mathcal{H}\!\mathit{om}}(\mathcal{O}_{Y\times_S Y}, p_2^! \mathcal{O}_Y)
            \\&\to \operatorname{\mathbf{R}\mathcal{H}\!\mathit{om}}(\mathcal{I}_{\Delta_f}, p_2^! \mathcal{O}_Y)
            \to \operatorname{\mathbf{R}\mathcal{H}\!\mathit{om}}((\Delta_f)_\ast \mathcal{O}_Y, p_2^! \mathcal{O}_Y)[1].
        \end{aligned}
    \end{displaymath}
    For each $B\in D_{\operatorname{qc}}(Y)$, applying the exact functor $\mathbf{R}(p_1)_\ast (\mathbf{L}p_2^\ast B \otimes^{\mathbf{L}} -)$ yields a distinguished triangle whose terms are the values at $B$ of the integral transforms with these three kernels.
    By \Cref{ex:perfect_kernels} with $Q = \mathcal{O}_{Y\times_S Y}$, the middle kernel is relatively perfect over each factor.
    Hence, by \cite[Corollary 1.9]{GuisadoVillaalgordo/Lank:2026}, its integral transform preserves $D^b_{\operatorname{coh}}(Y)$ and $\operatorname{Perf}(Y)$.
    As in the proof of \Cref{prop:ideal_sheaf_diagonal}, $\Phi_{(\Delta_f)_\ast \mathcal{O}_Y}$ is naturally isomorphic to the identity.
    Moreover, by \cite[Corollary 1.9]{GuisadoVillaalgordo/Lank:2026}, $(\Delta_f)_\ast \mathcal{O}_Y$ is relatively perfect over each factor.
    By \cite[Proposition 4.19]{GuisadoVillaalgordo/Lank:2026}, the integral transform with kernel $\operatorname{\mathbf{R}\mathcal{H}\!\mathit{om}}((\Delta_f)_\ast \mathcal{O}_Y, p_2^! \mathcal{O}_Y)$ is right adjoint to $\Phi_{(\Delta_f)_\ast \mathcal{O}_Y}$.
    Hence, it is naturally isomorphic to the identity by the uniqueness of adjoints.
    Since $D^b_{\operatorname{coh}}(Y)$ and $\operatorname{Perf}(Y)$ are triangulated subcategories, the integral transform with kernel $\operatorname{\mathbf{R}\mathcal{H}\!\mathit{om}}(\mathcal{I}_{\Delta_f}, p_2^! \mathcal{O}_Y)$ preserves both.
    Thus, \Cref{setup:relatively_perfect_adjoint_pair} applies.

    Before proving the claim, we set some notation and state a few facts for points of $|S|$.
    Fix $p\in |S|$.
    Pick an \'{e}tale germ $(U,u)\to (S,p)$ from an affine scheme.
    For simplicity, set $T := \operatorname{Spec}(\mathcal{O}_{U,u})$.
    Denote by $f_T \colon Y_T \to T$ the base change of $f$, and by $t^\prime \colon Y_T \times_T Y_T \to Y\times_S Y$ the canonical morphism.
    As $T \to S$ is the composite of a localization and an \'{e}tale morphism, it is flat, and so is $t^\prime$.
    Pull back the short exact sequence along $t^\prime$,
    \begin{displaymath}
        0 \to \mathcal{I}_{\Delta_f} \to \mathcal{O}_{Y\times_S Y} \to (\Delta_f)_\ast \mathcal{O}_Y \to 0.
    \end{displaymath}
    As $\Delta_{f_T}$ is the base change of $\Delta_f$, flat base change \cite[\href{https://stacks.math.columbia.edu/tag/073K}{Tag 073K}]{StacksProject} gives $\mathbf{L}(t^\prime)^\ast \mathcal{I}_{\Delta_f} \cong \mathcal{I}_{\Delta_{f_T}}$.
    Likewise, the canonical morphism $\mathcal{O}_T \to \mathbf{R}(f_T)_\ast \mathcal{O}_{Y_T}$ is an isomorphism.
    By \Cref{lem:singular_FM_locus_independence}, it suffices to check whether $\Phi^{\operatorname{sg}}_{\mathcal{I}_{\Delta_{f_T}}}$ is fully faithful (resp.\ an equivalence) for points in the various loci.
    Recall that $p\in \operatorname{reg}(S)$ if, and only if, $T$ is regular \cite[\href{https://stacks.math.columbia.edu/tag/0AH9}{Tag 0AH9}]{StacksProject}.

    Assume $T$ is regular; that is, $p\in \operatorname{reg}(S)$.
    By \Cref{prop:ideal_sheaf_diagonal},
    $\Phi^{\operatorname{sg}}_{\mathcal{I}_{\Delta_{f_T}}}$ is naturally isomorphic to $[-1]$, and hence is an equivalence.
    Thus, $p\in \operatorname{FM}_{\operatorname{sg}}(\mathcal{I}_{\Delta_f}) \subseteq \operatorname{fm}_{\operatorname{sg}}(\mathcal{I}_{\Delta_f})$, which gives the inclusions
    \begin{displaymath}
        \operatorname{fm}_{\operatorname{sg}}(\mathcal{I}_{\Delta_f})
        \supseteq \operatorname{FM}_{\operatorname{sg}}(\mathcal{I}_{\Delta_f}) \supseteq \operatorname{reg}(S).
    \end{displaymath}
    We check the reverse inclusions.

    Choose $p\in \operatorname{fm}_{\operatorname{sg}}(\mathcal{I}_{\Delta_f})$.
    Assume $p\not \in \operatorname{reg}(S)$, i.e.\ $T$ is not regular.
    By \cite[Lemma 5.28]{GuisadoVillaalgordo/Lank:2026}, there exists $A\in D^b_{\operatorname{coh}}(T)$ which is not perfect.
    As $f_T$ is flat, $\mathbf{L}f_T^\ast A \in D^b_{\operatorname{coh}}(Y_T)$, and the projection formula gives
    \begin{displaymath}
        \mathbf{R}(f_T)_\ast \mathbf{L}f_T^\ast A \cong A \otimes^{\mathbf{L}} \mathbf{R}(f_T)_\ast \mathcal{O}_{Y_T} \cong A.
    \end{displaymath}
    As $f_T$ is proper and flat, \Cref{lem:quasi-perfect_morphism} implies $\mathbf{R}(f_T)_\ast$ preserves perfect complexes.
    So $\mathbf{L}f_T^\ast A$ is not perfect, i.e.\ it is a nonzero object of $D_{\operatorname{sg}}(Y_T)$.
    However, \Cref{lem:diagonal_ideal_kills_pullbacks}\eqref{lem:diagonal_ideal_kills_pullbacks2} gives $\Phi_{\mathcal{I}_{\Delta_{f_T}}}(\mathbf{L}f_T^\ast A) \cong 0$.
    Hence, $\Phi^{\operatorname{sg}}_{\mathcal{I}_{\Delta_{f_T}}}$ sends the identity morphism of a nonzero object of $D_{\operatorname{sg}}(Y_T)$ to zero, and so it is not faithful.
    Thus, $p\notin \operatorname{fm}_{\operatorname{sg}}(\mathcal{I}_{\Delta_f})$, which is a contradiction.
\end{proof}

\begin{proof}
    [Proof of \Cref{cor:autoequivalence_detects_regularity}]
    If $S$ is regular, then take $f=1_S$.
    We prove the converse. 
    Pick $f\colon Y \to S$ as in \Cref{introprop:explicit_loci_diagonal} such that $\Phi^{\operatorname{sg}}_{\mathcal{I}_{\Delta_f}}$ is an autoequivalence of $D_{\operatorname{sg}}(Y)$.
    By \Cref{prop:explicit_loci_diagonal}, $\mathcal{I}_{\Delta_f}$ with $f$ and $f$ satisfies \Cref{setup:relatively_perfect_adjoint_pair}, and $\operatorname{FM}_{\operatorname{sg}}(\mathcal{I}_{\Delta_f}) = \operatorname{reg}(S)$.
    By \Cref{prop:germwise_criteria}, $\operatorname{FM}_{\operatorname{sg}}(\mathcal{I}_{\Delta_f}) = |S|$.
    Consequently, $\operatorname{reg}(S) = |S|$, i.e.\ $S$ is regular.
\end{proof}

\section{Arithmetic families}
\label{sec:arithmetic_families}

\begin{lemma}
    \label{lem:germ_regular_locus}
    Let $f\colon Y\to S$ be a proper morphism of Noetherian algebraic spaces.
    Fix $p\in|S|$.
    Choose an \'{e}tale germ $s\colon (U,u)\to(S,p)$.
    Then $Y\times_S\operatorname{Spec}(\mathcal{O}_{U,u})$ is regular if, and only if, $Y$ is regular at every point of $|f|^{-1}(p)$.
    In particular, if $\operatorname{reg}(Y)$ is open, then the collection of $p\in|S|$ for which $Y\times_S\operatorname{Spec}(\mathcal{O}_{U,u})$ is regular for some (equivalently, every) \'{e}tale germ $(U,u)\to(S,p)$ is the open subset $|S|\setminus|f|(|Y|\setminus\operatorname{reg}(Y))$.
\end{lemma}

\begin{proof}
    After choosing an affine open neighborhood of $u$ in $U$, we can impose that $U=\operatorname{Spec}(A)$ is affine (this does not change $\mathcal{O}_{U,u}$).
    Let $\mathfrak{q}\subseteq A$ be the prime ideal corresponding to $u$.
    For simplicity, set $T:=\operatorname{Spec}(A_{\mathfrak{q}})$. 
    Denote by $t_0$ its closed point.
    Form the commutative diagram with fibered squares
    \begin{displaymath}
        \begin{tikzcd}
            {Y\times_S T} & {Y\times_S U} & Y \\
            T & U & {S.}
            \arrow["h", from=1-1, to=1-2]
            \arrow["{f_T}"', from=1-1, to=2-1]
            \arrow["{s^\prime}", from=1-2, to=1-3]
            \arrow[from=1-2, to=2-2]
            \arrow["f", from=1-3, to=2-3]
            \arrow[from=2-1, to=2-2]
            \arrow["s"', from=2-2, to=2-3]
        \end{tikzcd}
    \end{displaymath}

    We first claim that $Y\times_S T$ is regular at $y\in|Y\times_S T|$ if, and only if, $Y\times_S U$ is regular at $|h|(y)$.
    Since $f$ is quasi-compact and $U$ is affine, $Y\times_S U$ is quasi-compact.
    Choose an \'{e}tale presentation $V\to Y\times_S U$ from an affine scheme $V=\operatorname{Spec}(B)$.
    Then $V\times_U T\to Y\times_S T$ is an \'{e}tale presentation.
    Also, the projection $V\times_U T\to V$ is $\operatorname{Spec}(B\otimes_A A_{\mathfrak{q}})\to\operatorname{Spec}(B)$.
    As $B\otimes_A A_{\mathfrak{q}}$ is a localization of $B$, this projection induces isomorphisms on local rings.
    Hence, by \Cref{rmk:regular_locus_by_presentation}, the claim follows.

    Note that the regular locus of $Y\times_S T$ is closed under generalization.
    Indeed, suppose $Y\times_S T$ is regular at $y$, and let $y^\prime$ be a generalization of $y$.
    Pick $v\in|V\times_U T|$ over $y$.
    By \cite[\href{https://stacks.math.columbia.edu/tag/0B7W}{Tag 0B7W}]{StacksProject}, there exists a generalization $v^\prime$ of $v$ lying over $y^\prime$.
    The local ring at $v^\prime$ is a localization of the regular local ring at $v$, and so it is regular \cite[Theorem 19.3]{Matsumura:1989}.
    Applying \Cref{rmk:regular_locus_by_presentation}, $Y\times_S T$ is regular at $y^\prime$.

    Moreover, by \cite[Remark 3.11]{GuisadoVillaalgordo/Lank:2026}, every point of $|Y\times_S T|$ specializes to a closed point. 
    Since $f_T$ is proper, $|f_T|$ is closed, and so every closed point of $|Y\times_S T|$ lies over $t_0$.
    Consequently, $Y\times_S T$ is regular if, and only if, it is regular at every point of $|f_T|^{-1}(t_0)$.

    Lastly, by \cite[\href{https://stacks.math.columbia.edu/tag/03H4}{Tag 03H4}]{StacksProject}, $|h|$ maps $|f_T|^{-1}(t_0)$ onto the points of $|Y\times_S U|$ lying over $u$, and $|s^\prime|$ maps the latter onto $|f|^{-1}(p)$.
    Since $s^\prime$ is \'{e}tale, $Y\times_S U$ is regular at a point if, and only if, $Y$ is regular at its image.
    Indeed, compose an \'{e}tale presentation of $Y\times_S U$ with $s^\prime$, and apply \cite[\href{https://stacks.math.columbia.edu/tag/0AH9}{Tag 0AH9}]{StacksProject}.
    Combining these observations, the first claim follows.
    The last claim follows because $|f|$ is closed.
    This completes the proof.
\end{proof}

\begin{corollary}
    \label{cor:regular_total_spaces}
    Assume \Cref{setup:relatively_perfect_adjoint_pair}.
    Then $\operatorname{FM}_{\operatorname{sg}}(E)$ contains the open subset
    \begin{displaymath}
        |S|\setminus(|f_1|(|Y_1|\setminus\operatorname{reg}(Y_1))\cup|f_2|(|Y_2|\setminus\operatorname{reg}(Y_2))).
    \end{displaymath}
\end{corollary}

\begin{proof}
    By \Cref{setup:relatively_perfect_adjoint_pair}, each $D^b_{\operatorname{coh}}(Y_i)$ admits a classical generator.
    Hence, \Cref{lem:existence_classical_generator_implies_open_cohomological_regular_locus} implies each $\operatorname{reg}(Y_i)$ is open.
    Since each $|f_i|$ is closed, 
    \begin{displaymath}
        |S|\setminus(|f_1|(|Y_1|\setminus\operatorname{reg}(Y_1))\cup|f_2|(|Y_2|\setminus\operatorname{reg}(Y_2)))
    \end{displaymath}
    is open.
    Choose $p$ in this subset.
    Fix an \'{e}tale germ $(U,u)\to(S,p)$.
    By \Cref{lem:germ_regular_locus}, each $Y_i\times_S\operatorname{Spec}(\mathcal{O}_{U,u})$ is regular.
    By \cite[Lemma 5.28]{GuisadoVillaalgordo/Lank:2026}, their singularity categories are trivial.
    Hence, the induced functor between them is an equivalence, and so $p\in\operatorname{FM}_{\operatorname{sg}}(E)$.
\end{proof}

\begin{corollary}
    \label{cor:dimension_one_base}
    Assume \Cref{setup:relatively_perfect_adjoint_pair} with $S$ an integral Noetherian scheme of Krull dimension one.
    Denote by $\xi$ the generic point of $S$.
    Set $E_\xi$ to be the derived pullback of $E$ to $Y_1\times_S Y_2 \times_S\operatorname{Spec}(\kappa(\xi))$.
    If $\Phi^{\operatorname{sg}}_{E_\xi}$ is fully faithful (resp.\ an equivalence), then $\operatorname{fm}_{\operatorname{sg}}(E)$ (resp.\ $\operatorname{FM}_{\operatorname{sg}}(E)$) contains all but finitely many closed points of $S$.
    Moreover, if $V\subseteq S$ is the corresponding open subscheme, then $\Phi^{\operatorname{sg}}_{E_V}$ is fully faithful (resp.\ an equivalence) where $E_V$ is the restriction of $E$ to $(Y_1\times_S Y_2)\times_S V$.
\end{corollary}

\begin{proof}
    We prove the case for full faithfulness.
    The case for equivalence is similar.
    The identity of $S$ together with $\xi$ is an \'{e}tale germ of $(S,\xi)$.
    Note that $\mathcal{O}_{S,\xi} = \kappa(\xi)$.
    By the hypothesis, $\xi\in\operatorname{fm}_{\operatorname{sg}}(E)$.
    By \Cref{thm:openness_singular_fm}, $\operatorname{fm}_{\operatorname{sg}}(E)$ is open.
    Its complement is a closed subset of $S$ not containing $\xi$.
    Moreover, each of its finitely many irreducible components is a proper irreducible closed subset of $|S|$.
    Since $S$ is integral, Noetherian, and of Krull dimension one, these irreducible components are closed points in $|S|$.
    This proves the first claim.
    We show the second claim. 
    Since $V\to S$ is \'{e}tale, \Cref{prop:base_change} implies $\operatorname{fm}_{\operatorname{sg}}(E_V) = |V|$.
    Thus, by \Cref{prop:germwise_criteria}, $\Phi^{\operatorname{sg}}_{E_V}$ is fully faithful.
\end{proof}

\begin{lemma}
    \label{lem:hypersurface_regular}
    Let $(A,\mathfrak{m})$ be a regular local ring.
    Suppose $h\in\mathfrak{m}$ is nonzero.
    Then $A/(h)$ is regular if, and only if, $h\notin\mathfrak{m}^2$.
\end{lemma}

\begin{proof}
    If $h\notin\mathfrak{m}^2$, then $A/(h)$ is regular \cite[Theorem 14.2]{Matsumura:1989}.
    Suppose $h\in\mathfrak{m}^2$.
    Then $\mathfrak{m}/\mathfrak{m}^2$ is the cotangent space of $A/(h)$.
    Hence, the maximal ideals of $A/(h)$ and $A$ have the same minimal number of generators (i.e.\ $\dim A$).
    By \cite[Theorem 14.3]{Matsumura:1989}, $A$ is an integral domain.
    As $h\neq 0$, every chain of prime ideals of $A$ containing $h$ can be lengthened by adjoining the zero ideal.
    It follows that $\dim A/(h)<\dim A$, and so $A/(h)$ is not regular.
\end{proof}

\begin{example}
    \label{ex:arithmetic_conic_bundle}
    Let $R=\mathbb{Z}[\sqrt{-3}]$ and $S=\operatorname{Spec}(R)$.
    Then $R$ is excellent \cite[\href{https://stacks.math.columbia.edu/tag/07QW}{Tag 07QW}]{StacksProject}.

    We first compute $\operatorname{reg}(S)$.
    Write $R=\mathbb{Z}[t]/(t^2+3)$.
    Let $\mathfrak{M}\subseteq\mathbb{Z}[t]$ be a maximal ideal containing $t^2+3$. 
    Denote by $\ell$ the characteristic of the residue field $\mathbb{Z}[t]/\mathfrak{M}$.
    The local ring $\mathbb{Z}[t]_{\mathfrak{M}}$ is regular \cite[Theorem 19.5]{Matsumura:1989}.
    By \Cref{lem:hypersurface_regular}, $R$ is regular at $\mathfrak{M}$ if, and only if, $t^2+3\notin\mathfrak{M}^2$.
    If $t^2+3\in\mathfrak{M}^2$, then the image of $t^2+3$ in $\mathbb{F}_\ell[t]$ is the square of a linear polynomial, and so $\ell$ divides the discriminant $-12$.
    If $\ell=3$, then $\mathfrak{M}=(3,t)$, and $t^2+3\equiv 3$ modulo $\mathfrak{M}^2$, which is nonzero in $\mathfrak{M}/\mathfrak{M}^2$.
    If $\ell=2$, then $t^2+3\equiv(t+1)^2$ modulo $2$, and so $\mathfrak{M}=(2,t+1)$.
    Hence, it follows that
    \begin{displaymath}
        t^2+3 = (t+1)^2-2(t+1)+4\in\mathfrak{M}^2.
    \end{displaymath}
    Since $R$ is an integral domain, its generic point is regular.
    Therefore, $|S|\setminus\operatorname{reg}(S) = \{\mathfrak{p}_2\}$ where $\mathfrak{p}_2 = (2,1+\sqrt{-3})$.
    It is the unique prime ideal of $R$ over $2$.
    Indeed, $R/2R\cong\mathbb{F}_2[t]/((t+1)^2)$, and its residue field is $\mathbb{F}_2$.
    Moreover, as $2$ is not a square modulo $5$, $t^2+3\equiv t^2-2$ modulo $5$ is irreducible over $\mathbb{F}_5$.
    Hence, $(5)$ is a maximal ideal of $R$ with residue field $\mathbb{F}_{25}$.
    Consequently, $R_{(5)}$ is a regular local ring of dimension one whose maximal ideal is generated by $5$.

    Consider the conic bundle
    \begin{displaymath}
        Y = \operatorname{Proj}(R[x,y,z]/(xy-25z^2))\subseteq\mathbf{P}^2_S
    \end{displaymath}
    with structure morphism $f\colon Y\to S$.
    Since $Y$ is closed in $\mathbf{P}^2_S$, $f$ is proper.
    We check that $f$ is flat.
    Let $w\in Y$.
    Denote by $q\in S$ the image of $w$.
    On a basic affine open of $\mathbf{P}^2_S$ containing $w$, $Y$ is cut out by a dehomogenization $h$ of $xy-25z^2$.
    The image of $xy-25z^2$ in $\kappa(q)[x,y,z]$ is nonzero, and so the image of $h$ in the domain $\mathcal{O}_{\mathbf{P}^2_{\kappa(q)},w}$ is a nonzerodivisor.
    Since $\mathbf{P}^2_S\to S$ is flat, \cite[\href{https://stacks.math.columbia.edu/tag/00MF}{Tag 00MF}]{StacksProject} implies $h$ is a nonzerodivisor in $\mathcal{O}_{\mathbf{P}^2_S,w}$ and $\mathcal{O}_{Y,w}$ is flat over $\mathcal{O}_{S,q}$.
    Note that each fiber of $f$ is an effective Cartier divisor in $\mathbf{P}^2$ over a field.
    Hence, it is a local complete intersection over that field, and so it is Gorenstein \cite[\href{https://stacks.math.columbia.edu/tag/0BVA}{Tag 0BVA}]{StacksProject}.
    As $f$ is also flat, it is Gorenstein \cite[\href{https://stacks.math.columbia.edu/tag/0C05}{Tag 0C05}]{StacksProject}.

    There exists a short exact sequence
    \begin{displaymath}
        0\to\mathcal{O}_{\mathbf{P}^2_S}(-2)\xrightarrow{xy-25z^2}\mathcal{O}_{\mathbf{P}^2_S}\to\mathcal{O}_Y\to 0.
    \end{displaymath}
    By \cite[\href{https://stacks.math.columbia.edu/tag/01XT}{Tag 01XT}]{StacksProject}, $H^j(\mathbf{P}^2_S,\mathcal{O}_{\mathbf{P}^2_S}(-2)) = 0$ for all $j$, the canonical morphism $R\to H^0(\mathbf{P}^2_S,\mathcal{O}_{\mathbf{P}^2_S})$ is an isomorphism, and $H^j(\mathbf{P}^2_S,\mathcal{O}_{\mathbf{P}^2_S})=0$ for all $j>0$.
    Hence, $R\to H^0(Y,\mathcal{O}_Y)$ is an isomorphism, and $H^j(Y,\mathcal{O}_Y)=0$ for all $j>0$.
    Since $S$ is affine, \cite[\href{https://stacks.math.columbia.edu/tag/01XJ}{Tags 01XJ} \& \href{https://stacks.math.columbia.edu/tag/01XK}{01XK}]{StacksProject} imply the canonical morphism $\mathcal{O}_S\to\mathbf{R}f_\ast\mathcal{O}_Y$ is an isomorphism.
    The same argument shows that $L\to\mathbf{R}\Gamma(Y_L,\mathcal{O}_{Y_L})$ is an isomorphism for every morphism $\operatorname{Spec}(L)\to S$ from a field (where $Y_L$ is the base change of $Y$).
    Applying \Cref{prop:explicit_loci_diagonal},
    \begin{displaymath}
        \operatorname{fm}_{\operatorname{sg}}(\mathcal{I}_{\Delta_f}) = \operatorname{FM}_{\operatorname{sg}}(\mathcal{I}_{\Delta_f}) = |S|\setminus\{\mathfrak{p}_2\}.
    \end{displaymath}

    We claim that this locus contains points where the singularity category is nontrivial.
    Indeed, the fiber of $f$ over $(5)$ is the pair of lines $xy=0$ over $\mathbb{F}_{25}$.
    Let $v$ be its point $x=y=0$.
    On the affine open $z\neq 0$, $Y$ is $\operatorname{Spec}(R[x,y]/(xy-25))$.
    Hence, $\mathcal{O}_{Y,v} \cong R[x,y]_{(5,x,y)}/ (xy-25)$.
    Note that $(xy-25)$ lies in the square of the maximal ideal of $R[x,y]_{(5,x,y)}$.
    Since $R[x,y]_{(5,x,y)}$ is a localization of $R_{(5)}[x,y]$, it is a regular local ring \cite[Theorem 19.5]{Matsumura:1989}.
    Hence, \Cref{lem:hypersurface_regular} implies $Y$ is not regular at $v$.
    That is, $v$ is an $A_1$-singularity of the form $xy=5^2$, where $5$ is a uniformizer of $R_{(5)}$.
    The identity of $S$ and the point $(5)$ form an \'{e}tale germ.
    For simplicity, set $T:=\operatorname{Spec}(R_{(5)})$.
    Since $T$ is regular, the proof of \Cref{prop:explicit_loci_diagonal} shows the functor induced on $D_{\operatorname{sg}}(Y\times_S T)$ is naturally isomorphic to $[-1]$.
    As $Y\times_S T$ contains $v$, it is not regular, and so $D_{\operatorname{sg}}(Y\times_S T)$ is nontrivial \cite[Lemma 5.28]{GuisadoVillaalgordo/Lank:2026}.
    In particular, \Cref{cor:regular_total_spaces} does not detect $(5)$.

    On the other hand, $25\equiv 1$ modulo $2$, and so the fiber of $f$ over $\mathfrak{p}_2$ is the conic $xy+z^2=0$ over $\mathbb{F}_2$.
    On the affine opens $x\neq 0$, $y\neq 0$, and $z\neq 0$ it is respectively isomorphic to $\operatorname{Spec}(\mathbb{F}_2[z])$, $\operatorname{Spec}(\mathbb{F}_2[z])$, and $\operatorname{Spec}(\mathbb{F}_2[x,x^{-1}])$.
    Hence, it is regular.
    However, $Y$ is not regular at any point $w$ over $\mathfrak{p}_2$.
    Otherwise, as $\mathcal{O}_{S,\mathfrak{p}_2}\to\mathcal{O}_{Y,w}$ is a flat local homomorphism, $R_{\mathfrak{p}_2}$ would be regular \cite[\href{https://stacks.math.columbia.edu/tag/00OF}{Tag 00OF}]{StacksProject}.
    Thus, the failure at $\mathfrak{p}_2$ comes from the singularity of the base, and not from the fiber.

    Lastly, we compare with $D^b_{\operatorname{coh}}$.
    By \cite[Lemma 5.32]{GuisadoVillaalgordo/Lank:2026}, $\mathcal{I}_{\Delta_f}$ is relatively perfect over each factor, and so \cite[Theorem 1.3]{GuisadoVillaalgordo/Lank:2026} applies.
    Let $\operatorname{Spec}(L)\to S$ be a morphism from a field.
    As $S$ is affine, $\operatorname{Spec}(L)\to S$ is affine.
    By \cite[Proposition 5.35]{GuisadoVillaalgordo/Lank:2026}, the derived pullback of $\mathcal{I}_{\Delta_f}$ to $Y_L\times_L Y_L$ is the ideal sheaf $\mathcal{I}_{\Delta_{f_L}}$ of the diagonal of $f_L\colon Y_L\to\operatorname{Spec}(L)$.
    Since $L\to\mathbf{R}\Gamma(Y_L,\mathcal{O}_{Y_L})$ is an isomorphism, \Cref{lem:diagonal_ideal_kills_pullbacks} implies $\Phi_{\mathcal{I}_{\Delta_{f_L}}}(\mathcal{O}_{Y_L})\cong 0$.
    As $\mathcal{O}_{Y_L}$ is nonzero, $\Phi_{\mathcal{I}_{\Delta_{f_L}}}$ is not faithful on $D^b_{\operatorname{coh}}(Y_L)$.
    Consequently, the loci of \cite[Theorem 1.3]{GuisadoVillaalgordo/Lank:2026}, where $\Phi_{\mathcal{I}_{\Delta_f}}$ is fiberwise fully faithful or an equivalence on $D^b_{\operatorname{coh}}$, are both empty.
\end{example}

\bibliographystyle{alpha}
\bibliography{mainbib}

\end{document}